\documentclass[a4paper]{amsart}

\usepackage{graphicx}
\graphicspath{ {./images/} }
\usepackage{subcaption}

\usepackage{amsfonts,amssymb,mathrsfs,bbold,comment,atableau}
\usepackage{stmaryrd,soul,adjustbox,circuitikz}
\usepackage[a4paper,hmargin=26mm,vmargin=28mm]{geometry}

\usepackage{enumitem}
\setlist[enumerate]{label=\upshape(\alph*)., ref=\alph*}

\usepackage[svgnames]{xcolor}
\definecolor{mblue}{HTML}{1F77B4}
\definecolor{mgreen}{HTML}{2CA02C}
\definecolor{mred}{HTML}{D62728}

\usepackage{booktabs,upgreek}
\usepackage{xparse,mathtools,etoolbox}

\usepackage{tikz}
\usetikzlibrary{arrows,decorations.markings}
\tikzset{%
  ->-/.style={decoration={markings,
                          mark=at position 0.6 with {\arrow[thin,scale=3]{>}}},
              postaction={decorate}},
  -<-/.style={decoration={markings,
                          mark=at position 0.6 with {\arrow[thin,scale=3]{<}}},
              postaction={decorate}},
  circled/.style = {fill=Tan},
}

\title{Stingray Patterns of Dominant Weights}
\author{Tao Qin}



\usepackage{cite}
\usepackage[pagebackref=false]{hyperref}
\hypersetup{
  pdftitle={\@title},
  pdfauthor={Tao Qin},
  colorlinks=false,
  linkcolor=black,
  citecolor=black,
  urlcolor=black,
  linkbordercolor=red,
  citebordercolor=red,
  urlbordercolor=red,
  pdfborder={0 0 1}
}

\newcommand\enumref[2]{\hyperref[#2]{\autoref*{#1}(\ref*{#2})}}

\AtBeginDocument{

}

\NewDocumentCommand{\Ab}{ O{} O{e} m }{%
  \ensuremath{%
    \operatorname{Ab}^{#1}_{#2}\!\left(#3\right)%
  }%
}

\newcommand{\N}{\mathbb{N}}

\newcommand{\Sym}{\mathfrak S}
\newcommand{\Z}{\mathbb{Z}}
\NewDocumentCommand{\triple}{O{r} O{e} O{w} }{(#1, #2, #3)}
\newcommand\Aone[1][e+1]{A^{(1)}_{#1}}

\newcommand\blam{{\boldsymbol\lambda}}

\def\c{\mathsf{c}}

\DeclarePairedDelimiterX{\set}[1]{\{}{\}}{\setargs{#1}}
\NewDocumentCommand{\setargs}{>{\SplitArgument{1}{|}}m}{\setargsaux#1}
\NewDocumentCommand{\setargsaux}{mm}
{\IfNoValueTF{#2}{#1} {#1\,\delimsize|\,\mathopen{}#2}}

\def\pmod#1{\text{ }(\textrm{mod } #1)\,}

\usepackage{aliascnt}
\def\NewTheorem#1{%
  \newaliascnt{#1}{equation}%
  \newtheorem{#1}[#1]{#1}%
  \aliascntresetthe{#1}%
  \expandafter\def\csname #1autorefname\endcsname{#1}%
}
\def\equationautorefname~#1\null{(#1)\null}
\def\itemautorefname~#1\null{(#1)\null}

\numberwithin{equation}{section}

\NewTheorem{Proposition}
\NewTheorem{Theorem}
\NewTheorem{Conjecture}
\NewTheorem{Corollary}
\NewTheorem{Lemma}
\NewTheorem{Definition}
\theoremstyle{definition}
\NewTheorem{Notation}
\NewTheorem{Example}
\AtEndEnvironment{Example}{\null\hfill$\Diamond$}%
\NewTheorem{Examples}
\AtEndEnvironment{Examples}{\null\hfill$\Diamond$}%
\theoremstyle{remark}
\NewTheorem{Remark}
\AtEndEnvironment{Remark}{\null\hfill$\Diamond$}%
\NewTheorem{Claim}

\begin{document}

\begin{abstract}
    We study the set $W_{r,e,w}$ of dominant weights of $\mathfrak{sl}_r$ arising from partitions of fixed $e$-weight $w$. For $e$-cores, we show that $W_{r,e,0}$ decomposes as a disjoint union of simplices indexed by compositions of $r$. For general $w$, we prove that $W_{r,e,w}$ is a disjoint union of copies of these simplices, with multiplicities determined by the corresponding quotient data, yielding in particular a closed counting formula for $|W_{r,e,w}|$. The geometry gives rise to the stingray patterns appearing in the title. More generally, it yields a natural labeling of the dominant $e$-alcoves meeting $W_{r,e,w}$ by weak compositions of $w$, together with a compatible partial action of the affine Weyl group via wall crossing. Finally, we give an explicit alcove-geometric proof of the empty runner removal theorem for Iwahori--Hecke algebras.
\end{abstract}

\maketitle
\setcounter{tocdepth}{1}
\tableofcontents

\section{Introduction}\label{sec:introduction}

The representation theory of symmetric groups, Iwahori–Hecke algebras, and $q$-Schur algebras is governed by the combinatorics of $e$-cores and $e$-weights. Over $\mathbb{C}$, decomposition numbers can be computed via canonical bases in highest weight representations of quantum affine algebras and parabolic Kazhdan–Lusztig polynomials; see classical sources \cite{kl-rep-coxeter,llt-llt-algorithm} or \cite[Chapter 6]{mathas-iwahori-hecke} for a summary.

A particularly striking manifestation of this structure is provided by the \emph{runner removal theorems}. In characteristic zero, James and Mathas proved an empty runner removal theorem for Iwahori--Hecke algebras and $q$--Schur algebras, relating decomposition numbers for different quantum characteristics by adding empty runners to the abacus \cite[Theorem 3.2 \& 4.5]{jamesmathas-empty-runner-removal}. Fayers subsequently established a full runner removal theorem \cite[Theorem 2.1]{fayers-full-runner-removal} and a more general version \cite[Theorem 3.4]{Fayers-general-runner-removal}. Chuang and Miyachi extended the empty runner removal theorem to positive characteristic by constructing appropriate Morita equivalences \cite{chuangmiyachi-runner-removal}. More recently, higher--level analogues for Ariki--Koike algebras have been obtained in \cite{alice-full-runner-removal,alice-empty-runner-removal} using higher--level Fock spaces. In related work, we aim to categorify these higher-level runner removal theorems \cite{qin-subdivision-runner-removal}.

Goodman suggested that the empty runner–removal theorems can also be obtained from a different viewpoint. In \cite{goodmanwenzl-crystal-bases}, Goodman and Wenzl identify certain parabolic Kazhdan–Lusztig polynomials for the affine Weyl group of type $A$ with polynomials whose specialization at $q=1$ yields the decomposition numbers of Iwahori–Hecke algebras. Via KLR theory, these polynomials are now known to be the graded decomposition numbers of Iwahori–Hecke algebras and $q$–Schur algebras; see \cite{bk-graded-decomposition-number}.

Combined with standard alcove–geometric techniques for affine Weyl groups \cite[Chapter 4]{humphreys-reflection-groups}, this suggests a conceptually clean proof of runner–removal theorems: the effect of adding or removing runners should be visible directly in the way suitable regions of the dominant chamber sit inside the alcove arrangement. In this paper we prove the empty–runner–removal theorem following Goodman’s idea, which leads to a shorter argument and reveals additional combinatorial structure.

The geometric foundation of our approach lies in identifying partitions with points in the weight lattice. By associating a partition of length at most $r$ with a dominant integral weight of $\mathfrak{sl}_r$ (shifted by the Weyl vector $\rho$), we can visualize the combinatorics of partitions inside the dominant chamber $\mathcal{C}^+$.

This embedding transforms algebraic constraints into geometric ones. Specifically, block theory for Hecke algebras and $q$-Schur algebras tells us that blocks are classified by the $e$-core and the $e$-weight, see \cite[Section 5.3]{mathas-iwahori-hecke}. This motivates the central question of this paper:

\medskip
\emph{Fix the rank $r$, the quantum characteristic $e$, and the $e$-weight $w$. How are the weights corresponding to partitions with these invariants distributed within the dominant chamber?}
\medskip

Denote the set of these weights by $\mathcal{W}_{r,e,w}$. Our analysis reveals that the geometry of $\mathcal{W}_{r,e,w}$ provides a good model for the interaction between the affine Weyl group and abacus combinatorics.

Our first main result describes the base case of $e$-cores ($w=0$). In \autoref{thm:simplex-of-core-weights} and \autoref{thm:simplicial-decomposition}, we show that the set of weights $\mathcal{W}_{r,e,0}$ forms a highly structured pattern: it is a disjoint union of simplices. These simplices are not arbitrary; they are naturally indexed by compositions of $r$, which correspond to specific configurations of beads on the abacus runners.

For the general case ($w>0$), we prove that this geometric structure is stable. The set $\mathcal{W}_{r,e,w}$ consists of copies of the simplices found in the core case, appearing with multiplicities determined by the number of ways to distribute the weight $w$ across the abacus. This yields a closed formula for the size of $\mathcal{W}_{r,e,w}$ as a sum over compositions of $r$ (\autoref{thm:composition-sum}).

This geometry is particularly transparent in type $A_2$. As $w$ increases, the weights in $\mathcal{W}_{3,e,w}$ form a characteristic ``stingray" pattern along the boundary and hexagonal patterns in the interior (see \autoref{fig:r3e8-grid-0-9}). These patterns generalise to arbitrary type $A_{r-1}$ for $r\ge3$, and we study some of their properties, see \autoref{subsec:stingray-hexagon}. This analysis also leads to a natural indexing of $e$–alcoves in the dominant chamber by weak compositions of $w$ of length $r$(see \autoref{subsec:index-alcoves}). 

Moreover, the set of weak compositions of length $r$ carries a partial action of the affine Weyl group, see \autoref{subsec:affine-weyl-group-action}. By transferring this action to the set of $e$-alcoves, we show that it indeed defines a geometric action of the affine Weyl group on $e$-alcoves via crossing the corresponding walls. This can be regarded as a generalization of indexing the dominant $e$-alcoves by the core partitions, and it provides another model for studying the affine Weyl group action on the dominant $e$-alcoves.

Finally, we return to runner–removal. In \autoref{sec:proof-empty}, we use Good\-man–Wenzl identification \cite[Theorem 5.3]{goodmanwenzl-crystal-bases}, to give an explicit alcove–geometric proof of the empty runner–removal theorem of James–Mathas \cite[Theorem 3.2]{jamesmathas-empty-runner-removal}. 

The paper is organised as follows. In \autoref{sec:preliminaries} we recall all the necessary background on partitions, $e$–cores, $e$–weights, abaci and the affine Weyl group of type~$A$, and we fix our conventions for $\Omega$ and alcove geometry. In \autoref{sec:pattern-dominant-weights} we study the pattern of dominant weights in $\mathcal{W}_{r,e,w}$, prove the simplicial decomposition and counting formulas, and analyse the stingray and regular patterns. In \autoref{sec:proof-empty}, we prove the empty runner–removal theorem.

\section{Preliminaries}\label{sec:preliminaries}
In this section, we collect the combinatorial and algebraic definitions required for our main results. Our primary goal is to establish the notation and conventions used throughout the paper. Consequently, we generally refer the reader to standard textbooks or comprehensive papers for more details.
\subsection{Partitions and compositions}\label{subsec:partitions-compositions}
Let $n$ be a non-negative integer. A \emph{partition} of $n$, denoted $\lambda \vdash n$, is a finite sequence $\lambda=(\lambda_1, \lambda_2, \cdots, \lambda_r)$ of weakly decreasing positive integers such that $n = \sum_{1\le i\le r} \lambda_i$. Here $n$ is called the \emph{size} of $\lambda$, denoted $|\lambda|$ and each $\lambda_i$ is called a \emph{part} of $\lambda$. The number of parts is the \emph{length} $\ell(\lambda)$. 

A \emph{composition} is a finite sequence $\mu=(\mu_1, \dots, \mu_k)$ of positive integers. The sum of the terms is the \emph{size} of $\mu$, denoted $|\mu|$. If $|\mu|=n$, we write $\mu \models n$.
A \emph{weak composition} is a sequence $\gamma=(\gamma_1, \dots, \gamma_r)$ of non-negative integers. Its \emph{size} is similarly defined as $|\gamma|=\sum \gamma_i$. Here $k$ and $r$ are called the \emph{length} of $\mu$ and $\gamma$, respectively.

\begin{Example}\label{eg:partition-composition}
    The tuple $(4,1,3)$ is a composition of size $8$ and length $3$ but not a partition; $(4,0,1,3)$ is a weak composition of size $8$ and length $4$ but not a composition (as it contains a zero).
\end{Example}

For an integer $e\ge 2$, a partition $\lambda=(\lambda_1,\dots,\lambda_r)$ is \emph{$e$–regular} if there are no $e$ equal parts.
\subsection{Multi-partitions and counting formulas}\label{subsec:multipartitions}
Let $\ell$ be a positive integer. A \emph{$\ell$-partition} is a tuple $\blam=\big(\blam^{(1)}\mid \ldots\mid\blam^{(\ell)}\big)$ of ordinary partitions.
Its \emph{size} is $|\blam|=\sum_{i=1}^{\ell} |\blam^{(i)}|$.

Given a composition $\mu=(\mu_1,\ldots,\mu_{\ell})$ of length $\ell$, a $\ell$-partition $\blam$ is said to be \emph{of type $\mu$} if it satisfies the length constraint:
\begin{equation*}
    \ell\big(\blam^{(i)}\big) \le \mu_i \quad \text{for all } 1 \le i \le \ell.
\end{equation*}

\begin{Definition}\label{dfn:A-mu-w}
    For a composition $\mu=(\mu_1,\ldots,\mu_{\ell})$ and $w\in\Z_{\ge 0}$, we define $A(\mu;w)$ to be the number of $\ell$-partitions of type $\mu$ with total size $w$:
        \begin{equation*}
            A(\mu;w)\;:=\;\#\Big\{\,\big(\blam^{(1)}\mid\ldots\mid\blam^{(\ell)}\big):\; |\blam^{(1)}|+\cdots+|\blam^{(\ell)}|=w,\ \ell\big(\blam^{(i)}\big)\le \mu_i\ \forall i\Big\}.
        \end{equation*}
\end{Definition}

\begin{Example}
    Let $\mu=(2,1)$ and $w=2$. To calculate $A\big((2,1); 2 \big)$, we look for $\big(\blam^{(1)}\mid \blam^{(2)}\big)$ such that $|\blam^{(1)}|+|\blam^{(2)}|=2$, $\ell(\blam^{(1)})\le 2$, and $\ell(\blam^{(2)})\le 1$.
    \begin{itemize}
        \item Case $|\blam^{(1)}|=2, |\blam^{(2)}|=0$: $\blam^{(1)} \in \{(2), (1,1)\}$, $\blam^{(2)}=\emptyset$. 
        
        \item Case $|\blam^{(1)}|=1, |\blam^{(2)}|=1$: $\blam^{(1)}=(1)$, $\blam^{(2)}=(1)$. 
        
        \item Case $|\blam^{(1)}|=0, |\blam^{(2)}|=2$: $\blam^{(1)}=\emptyset$. For $\blam^{(2)}$, the partition $(2)$ has length 1 which is valid, but $(1,1)$ has length 2 which is invalid since $\mu_2=1<\ell(1,1)$. 
    \end{itemize}
    Thus, $A\big((2,1); 2\big) = 2 + 1 + 1 = 4$.
\end{Example}
\subsection{Beta numbers and the $e$–abacus}\label{subsec:beta-abacus}
For a partition $\lambda=\big(\lambda_1,\dots,\lambda_{\ell(\lambda)}\big)$ and an integer $r\ge \ell(\lambda)$, the \emph{$r$–beta numbers} of $\lambda$ are
\begin{equation*}
  \beta_i(\lambda) \;=\; \lambda_i + r - i \qquad (1\le i\le r).
\end{equation*}
This sequence is strictly decreasing and consists of non-negative integers. Conversely, any strictly decreasing sequence $\{\beta_1>\beta_2>\cdots>\beta_{r}\}$ of $r$ non-negative integers determines a unique partition via $\lambda_i = \beta_i - (r-i)$.

James introduced the abacus model in \cite{james-abacus}. Fix $e \ge 2$. An \emph{$e$–abacus} has $e$ vertical \emph{runners} indexed by $0,1,\dots,e-1$ from left to right. The runner indexed by $i$ will be called the $i$–runner.

Each runner is divided into \emph{rows} indexed by the natural numbers $0,1,2,\dots$, increasing downwards from the top, and rows with the same index are aligned horizontally across all runners. For $0\le b\le e-1$ and $a\ge 0$ we refer to the position in row $a$ on the $b$–runner as the position $ae+b$. Each position may either carry a bead or be empty (a \emph{gap}). An \emph{$e$–abacus configuration}, or simply an \emph{$e$–abacus}, is a choice of finitely many beads placed on positions of this $e$–runner array. If it has $r$ beads, we call it an $e$–abacus with $r$ beads.

We can identify $r$–beta numbers with abacus configurations as follows. Given $r$–beta numbers $\{\beta_i\mid 1\leq i\leq r\}$, write $\beta_i=a_i e + b_i$ with $0\le b_i\le e-1$, form the $e$-abacus with $r$ beads by placing a bead on row $a_i$ of the $b_i$–runner for each $i$ and leave all other positions empty. Conversely, given an $e$–abacus with $r$ beads, record for each bead its position $ae+b$, and arrange these integers in decreasing order. This gives a strictly decreasing sequence of non-negative integers, i.e. a sequence of $r$–beta numbers.



Given a partition $\lambda=(\lambda_1,\dots,\lambda_r)$, we form its $r$–beta numbers $\beta_i(\lambda)$ and the corresponding $e$–abacus with $r$ beads, written $\operatorname{Ab}^r_e(\lambda)$ and called the \emph{$e$–abacus of $\lambda$ (with $r$ beads)}. By abuse of language we often identify a bead with its beta number and refer to the bead corresponding to $\beta_i$ simply as the bead $\beta_i$. In particular, we call the bead $\beta_1$ the \emph{first} bead and the bead $\beta_r$ the \emph{last} bead in $\operatorname{Ab}^r_e(\lambda)$.

In an $e$-abacus, a runner is called \emph{empty} if there is no bead on it, and it is called \emph{flush} if each bead on that runner is placed as high as possible. Formally, the $b$--runner is flush if, whenever there is a bead at position $ae+b$, there is also a bead at $(a-1)e+b$ (for $a \ge 1$).

\begin{Example}\label{eg:abaci-beta-numbers-4322}
    Consider the partition $\lambda=(4,3,2,2)$. We illustrate the abaci $\operatorname{Ab}^r_3(\lambda)$ for $r=4$ and $r=5$. 
    \begin{center}
        \begin{minipage}{0.45\textwidth}
            \centering
            \Abacus[runner labels={0,1,2}, entries=betas]{3}{4,3,2,2}
            
            \vspace{0.2cm}
            (a) $r=4$: beta numbers $(7,5,3,2)$
        \end{minipage}
        \hfill
        \begin{minipage}{0.45\textwidth}
            \centering
            \Abacus[runner labels={0,1,2}, entries=betas]{3}{4,3,2,2,0}
            
            \vspace{0.2cm}
            (b) $r=5$: beta numbers $(8,6,4,3,0)$
        \end{minipage}
    \end{center}
    In case (a), the $2$-runner is flush, whereas the $0$- and $1$-runners are not. In case (b), the $0$-runner is flush, while the $1$- and $2$-runners are not.
\end{Example}

We define elementary operations on abacus by moving a single bead to an adjacent position. Let a bead be located at position $x = ae+b$ (row $a$, runner $b$). A move is valid only if the target position is empty (a \emph{gap}). The atomic moves are:

\begin{itemize}
    \item \textbf{Vertical moves (Sliding):}
    \begin{itemize}
        \item \emph{Slide down:} $x \mapsto x+e$. This moves the bead to row $a+1$ on the same runner ($ae+b \mapsto (a+1)e+b$).
        \item \emph{Slide up:} $x \mapsto x-e$. This moves the bead to row $a-1$ on the same runner ($ae+b \mapsto (a-1)e+b$), provided $a \ge 1$.
    \end{itemize}
    
    \item \textbf{Horizontal moves (Shifting):}
    \begin{itemize}
        \item \emph{Shift right:} $x \mapsto x+1$. Generally, this moves the bead to the adjacent runner on the right ($b \mapsto b+1$) within the same row. However, if the bead is on the rightmost runner ($b=e-1$), it moves to the leftmost runner of the next row:
        \[
        ae + (e-1) \;\longmapsto\; (a+1)e + 0.
        \]
        \item \emph{Shift left:} $x \mapsto x-1$. Generally, this moves the bead to the adjacent runner on the left ($b \mapsto b-1$) within the same row. However, if the bead is on the leftmost runner ($b=0$), it moves to the rightmost runner of the previous row:
        \[
        ae + 0 \;\longmapsto\; (a-1)e + (e-1).
        \]
    \end{itemize}
\end{itemize}
General bead movements are obtained by iterating these atomic steps, provided the target position at each step is empty.
\subsection{Core, weight, and quotient}\label{subsec:core-quotient}
Fix $e\in\Z_{\ge 2}$. A partition $\lambda$ is called an \emph{$e$–core} if there exists some $r\ge \ell(\lambda)$ such that every runner of $\operatorname{Ab}^r_e(\lambda)$ is flush. For any partition $\lambda$, choose $r\ge \ell(\lambda)$, form the $e$–abacus with $r$ beads, and slide each bead upwards along its runner as far as possible. The resulting abacus corresponds to an $e$–core partition, denoted $\operatorname{core}_e(\lambda)$. We call $\operatorname{core}_e(\lambda)$ the \emph{$e$–core} of $\lambda$, and define the \emph{$e$–weight} of $\lambda$ to be the total number of upward moves needed to obtain $\operatorname{core}_e(\lambda)$ from $\operatorname{Ab}^r_e(\lambda)$, denoted by $w_e(\lambda)$. 
Equivalently,
\begin{equation*}
  w_e(\lambda)\;=\;\frac{|\lambda|-|\operatorname{core}_e(\lambda)|}{e}\;\in\;\mathbb{Z}_{\ge0},
\end{equation*}

The $e$–core and $e$–weight of a partition are independent of the choice of $r$; see \cite[Section~5.3]{mathas-iwahori-hecke}.
\begin{Example}
    In \autoref{eg:abaci-beta-numbers-4322}, for $r=4$ or $r=5$, there are runners that are not flush, hence $(4,3,2,2)$ is not a $3$-core. The corresponding $e$-cores correspond to the following abaci in the two cases, respectively:
    \begin{center}
        \Abacus[runner labels={0,1,2},entries=betas]{3}{2,0,0,0}\qquad\qquad\qquad
        \Abacus[runner labels={0,1,2},entries=betas]{3}{2,0,0,0,0}
    \end{center}
    Both abaci yield the same $3$-core partition $(2)$, and it is easy to verify that the $3$-weight is $3$.
\end{Example}
Fix an $e$-abacus of $\lambda$ with $r$ beads. For each runner $j \in \{0, \dots, e-1\}$, let $r_j$ be the number of beads on that runner. Let the row indices of these beads be $y_{j,1} > y_{j,2} > \dots > y_{j,r_j} \ge 0$.
We view this sequence as the $r_j$-beta numbers for a new partition $\lambda^{(j)}$. 
The \emph{$e$–quotient} of $\lambda$ is the $e$-tuple of partitions
\begin{equation}\label{eq:e-quotient}
    \operatorname{quot}_e(\lambda) = \big(\lambda^{(0)},\lambda^{(1)}, \dots, \lambda^{(e-1)}\big).
\end{equation}
A fundamental property relating the core, quotient, and weight is:
\[
    |\lambda| = |\operatorname{core}_e(\lambda)| + e \sum_{j=0}^{e-1} |\lambda^{(j)}|.
\]
Comparing this with the definition of the $e$-weight, we see that $w_e(\lambda) = \sum_{j=0}^{e-1} |\lambda^{(j)}|$.
A partition $\lambda$ is uniquely determined by its $e$-core and its $e$-quotient.

\begin{Example}
    In \autoref{eg:abaci-beta-numbers-4322}, $\lambda=(4,3,2,2)$ and $e=3$. The $4$-beta numbers are $(7,5,3,2)$.
    $0$-runner has a single bead at row $1$, which corresponds to the partition $\lambda^{(0)}=(1)$.
    $1$-runner has a bead at row $2$, yielding $\lambda^{(1)}=(2)$.
    $2$-runner contains beads at rows $1$ and $0$, which correspond to the empty partition $\lambda^{(2)}=\emptyset$.
    The $3$-quotient is $\big((1), (2), \emptyset\big)$. The total size is $1+2+0=3$, which matches the $3$-weight of $w_3(4,3,2,2)=3$.
\end{Example}
\subsection{Coxeter systems and Hecke algebras}
A \emph{Coxeter system} is a pair $(W,S)$ where $W$ is a group generated by a set $S$ subject to relations
\begin{equation*}
  \underbrace{s t s \cdots}_{m_{st}\ \text{factors}}
  \;=\;
  \underbrace{t s t \cdots}_{m_{st}\ \text{factors}}
  \qquad (s,t\in S).
\end{equation*}
determined by a matrix $(m_{st})_{s,t\in S}$. The entries satisfy $m_{ss}=1$ and $m_{st}=m_{ts}\in \{2,3,\dots,\infty\}$ for $s\ne t$. If $m_{st}=\infty$, no relation is imposed between $s$ and $t$. Let $\ell$ denote the length function and $\le$ the Bruhat order on $W$. An expression $w=s_1\cdots s_r$ for $w\in W$, $s_i\in S$ is \emph{reduced} if $\ell(w)=r$; see \cite[Chapter 5]{humphreys-reflection-groups} for standard theory.

Following the notation of \cite{soergel-kazhdan-lusztig-tilting}, the \emph{Hecke algebra} $\mathcal{H}(W)$ is the unital associative $\mathbb{Z}[v,v^{-1}]$-algebra generated by $\{H_s\mid s\in S\}$, subject to relations
\begin{align*}
    &(H_s+v)(H_s-v^{-1})=0 \qquad (s\in S),\\
    &\underbrace{H_s H_t H_s \cdots}_{m_{st}\ \text{factors}}
    \;=\;
    \underbrace{H_t H_s H_t \cdots}_{m_{st}\ \text{factors}}
    \qquad (s,t\in S,\ m_{st}<\infty).
\end{align*}
For a reduced expression $w=s_1\cdots s_r$, set $H_w=H_{s_1}\cdots H_{s_r}$; then $\{H_w\mid w\in W\}$ is a basis. We define the \emph{bar-involution} by $\overline{v}=v^{-1}$ and $\overline{H_w}=H_{w^{-1}}^{-1}$ and extend linearly. The \emph{Kazhdan–Lusztig basis} is the unique bar-invariant basis $\{\underline{H}_w\mid w\in W\}$ satisfying
\begin{equation*}
  \underline{H}_w \;=\; \sum_{y\le w} h_{y,w}(v)\,H_y,
  \qquad h_{w,w}=1,\quad h_{y,w}(v)\in v\,\mathbb{Z}[v]\ (y<w).
\end{equation*}
These notations differ from the original ones used in \cite{kl-rep-coxeter}; see \cite[Section 2]{soergel-kazhdan-lusztig-tilting} for a discussion of these conventions.
\subsection{Anti-spherical Kazhdan–Lusztig polynomials}\label{subsec:parabolic-kl}
While parabolic Kazhdan--Lusztig theory was originally developed by Deodhar \cite{deodhar-parabolic-kl}, we adopt the framework and notation of \cite[Section 3]{soergel-kazhdan-lusztig-tilting} for consistency.

Fix a Coxeter system $(W,S)$ and $I\subseteq S$. The subgroup $W_I$ generated by $I$ is called a \emph{parabolic subgroup}. The \emph{parabolic subalgebra} $\mathcal{H}_I$ of $\mathcal{H}(W)$ is the subalgebra generated by $\{H_s \mid s \in I\}$.
Let ${}^{I}W$ be the set of minimal length representatives for the right cosets
$W_I\backslash W$:
\begin{equation*}
  {}^{I}W \;=\; \{w \in W \mid \ell(sw) > \ell(w) \text{ for all } s \in I\}.
\end{equation*}
Each $w \in {}^{I}W$ is the unique element of minimal length in the coset $W_I w$.

Let $\operatorname{sgn}_I$ be the one–dimensional right $\mathcal{H}_I$–module where $H_s$ acts by multiplication by $-v$ for all $s\in I$. The \emph{anti-spherical module} is the induced right $\mathcal{H}(W)$-module:
\begin{equation*}
  \mathcal{N}_I \;=\; \operatorname{sgn}_I \otimes_{\mathcal{H}_I}\mathcal{H}(W).
\end{equation*}
This module has a standard basis given by $N_x=1\otimes H_x$ for $x\in{}^{I}W$. The bar-involution on $\mathcal{H}(W)$ extends to a bar-involution on $\mathcal{N}_I$, denoted $n \mapsto \overline{n}$, which is determined by the properties:
\begin{equation*}
    \overline{N_{\text{id}}} = N_{\text{id}} \quad \text{and} \quad \overline{n \cdot h} = \overline{n} \cdot \overline{h} \quad (n \in \mathcal{N}_I, h \in \mathcal{H}(W)).
\end{equation*}

By \cite[Proposition 3.2]{deodhar-parabolic-kl} or \cite[Theorem 3.1]{soergel-kazhdan-lusztig-tilting}, there exists a unique bar-invariant basis $\{\underline{N}_x\mid x\in{}^{I}W\}$ of $\mathcal{N}_I$ such that
\begin{equation*}
    \underline{N}_x \;=\; N_x \;+\!\!\sum_{\substack{y\in{}^{I}W, y<x}}\! \mathfrak{n}_{y,x}(v)\,N_y,
    \qquad \mathfrak{n}_{y,x}(v)\in v\,\mathbb{Z}[v].
\end{equation*}
The coefficients $\mathfrak{n}_{y,x}(v)$ are called the \emph{anti-spherical} \emph{Kazhdan–Lusztig polynomials}. \footnote{The reader is encouraged to visit \href{https://www.jgibson.id.au/lievis/affine_weyl/\#.labels:asklpoly-}{Lievis} to see the pattern of anti-spherical KL polynomials and many other interactive visualizations.}
\subsection{Type $A$ Cartan data}\label{subsec:type-A-data}
Fix $r\ge2$ and set $E=\Bigl\{(x_1,\dots,x_r)\in\mathbb{R}^r \,\Big|\, \sum_{i=1}^r x_i=0\Bigr\}$,
equipped with the restriction of the standard inner product $(\cdot,\cdot)$ on $\mathbb{R}^r$. Let $\varepsilon_i$ be the $i$th standard basis vector of $\mathbb{R}^r$.

The \emph{root system} of type $A_{r-1}$ is
\[
  R=\{\varepsilon_i-\varepsilon_j \mid 1\le i\ne j\le r\}\subset E.
\]
The set of \emph{simple roots} is
\begin{equation*}
  \Delta \;=\; \{\alpha_1,\dots,\alpha_{r-1}\}, \qquad \alpha_i=\varepsilon_i-\varepsilon_{i+1}.
\end{equation*}
This determines the set of \emph{positive roots}
\begin{equation*}
  R^+ \;=\; R \cap \bigoplus_{i=1}^{r-1} \mathbb{Z}_{\ge 0} \alpha_i \;=\; \{\varepsilon_i-\varepsilon_j \mid 1\le i<j\le r\}.
\end{equation*}
The \emph{highest root} is $\theta = \varepsilon_1-\varepsilon_r = \alpha_1+\alpha_2+\cdots+\alpha_{r-1}$.

For any root $\alpha \in R$, the corresponding \emph{coroot} is the vector defined by $\alpha^\vee = \frac{2\alpha}{(\alpha,\alpha)}$.
We identify the space $E$ with its dual $E^*$ via the inner product $(\cdot, \cdot)$.  Then, for any $x \in E$ and coroot $\alpha^\vee \in E$ (viewed as an element of the dual), the natural evaluation pairing is given by $\langle \alpha^\vee,x \rangle= (\alpha^\vee,x) = \frac{2(\alpha,x)}{(\alpha,\alpha)}$.

In type $A_{r-1}$, we have $(\alpha,\alpha)=2$ for all roots $\alpha\in R$. Consequently $\alpha^\vee=\alpha$, and the pairing simplifies to $\langle \alpha^\vee,x\rangle = (\alpha,x)$. 

The \emph{root lattice} is $Q=\bigoplus_{i=1}^{r-1}\mathbb{Z}\alpha_i$ and the \emph{fundamental weights} $\Lambda_1,\dots,\Lambda_{r-1}\in E$ are characterized by
\begin{equation*}
  \langle \alpha_i^\vee,\Lambda_j \rangle = \delta_{ij}.
\end{equation*}
The \emph{(integral) weight lattice} is $P=\bigoplus_{i=1}^{r-1}\mathbb{Z}\Lambda_i$ and an element of $P$ is referred as a \emph{weight}.

The set of \emph{dominant weights} is
\begin{equation*}
  P^+ \;=\; \{\lambda\in P \mid \langle \alpha_i^\vee,\lambda\rangle\ge0\ \text{for all }1\le i\le r-1\} \;=\; \bigoplus_{i=1}^{r-1}\mathbb{Z}_{\ge 0}\Lambda_i.
\end{equation*}
We refer to $P^+$ as the \emph{dominant weight lattice}. 

The \emph{Weyl vector} $\rho$ is the sum of the fundamental weights: $\rho=\sum_{i=1}^{r-1} \Lambda_i$. 
Equivalently, $\rho$ is equal to the half-sum of the positive roots; see \cite[Section 13.3]{humphreys-lie-algebra}.
\subsection{The geometric map $\Omega$}\label{subsec:map-omega}
We introduce a map from partitions to the weight lattice.
Recall that, for a partition $\lambda$ with $\ell(\lambda) \le r$, the $r$-beta numbers are given by $\beta_i(\lambda) = \lambda_i + r - i,  (1\le i\le r)$.

We define the map $\Omega$ by:
\[
    \Omega(\lambda) \;=\; \sum_{i=1}^{r-1} (\lambda_i - \lambda_{i+1} + 1) \Lambda_i\in P^+.
\]
The coefficients correspond precisely to the consecutive differences of the beta numbers:
\[
    \beta_i(\lambda) - \beta_{i+1}(\lambda) \;=\; (\lambda_i + r - i) - \big(\lambda_{i+1} + r - (i+1)\big) \;=\; \lambda_i - \lambda_{i+1} + 1.
\]
Thus, the map can be written in terms of beta numbers as:
\[
    \Omega(\lambda) \;=\; \sum_{i=1}^{r-1}\big(\beta_i(\lambda) - \beta_{i+1}(\lambda)\big) \Lambda_i.
\]
This map allows us to visualize partitions as points in the dominant weight lattice $P^+$. 
\subsection{Alcove geometry}
\subsubsection{Affine hyperplanes and alcoves}
Fix $r\ge 2$ and use type $A_{r-1}$  notations from \autoref{subsec:type-A-data}.

The \emph{affine Weyl group} of type $\Aone[r-1]$ is
$W \;=\; W_0 \ltimes Q$,
where $W_0 \cong \mathfrak{S}_r$ is the finite Weyl group of type $A_{r-1}$ acting on $E$ by permuting the coordinates, and $Q$ is the root lattice acting on $E$ by translations.  The affine Weyl group is a Coxeter group with generators $\sigma_0,\sigma_1,\cdots,\sigma_{r-1}$ where $\langle \sigma_1,\cdots,\sigma_{r-1}\rangle\cong \Sym_r$. See \cite[Chapter 4]{humphreys-reflection-groups} for details.

For $\alpha\in R$ and $k\in\mathbb{Z}$, define the affine hyperplane
\[
  H_{\alpha,k}=\{x\in E\mid \langle \alpha^\vee,x\rangle=k\}.
\]
This hyperplane separates $E$ into two open half-spaces: the \emph{positive side} $H_{\alpha,k}^+$ and the \emph{negative side} $H_{\alpha,k}^-$, defined by:
\begin{align*}
    H_{\alpha,k}^+ = \{x\in E \mid \langle \alpha^\vee,x\rangle > k\}, \qquad
    H_{\alpha,k}^- = \{x\in E \mid \langle \alpha^\vee,x\rangle < k\}.
\end{align*}

The set $\mathcal{H}=\{H_{\alpha,k}\mid\alpha\in R^+,k\in\Z\}$ is called the \emph{affine Coxeter arrangement}. The \emph{alcoves} are the connected components of the complement of the union of these hyperplanes, that is, of $E \setminus \bigcup_{H \in \mathcal{H}} H$.

A \emph{facet} of an alcove $\mathcal{A}$ is a codimension-one face of its closure $\overline{\mathcal{A}}$. Geometrically, it is the intersection of $\overline{\mathcal{A}}$ with a single hyperplane $H \in \mathcal{H}$ that forms part of the boundary of the alcove; this hyperplane is called a \emph{wall} of the alcove. 

The \emph{dominant chamber} is
\[
  \mathcal{C}^+=\{x\in E\mid \langle \alpha_i^\vee,x\rangle\ge0\ \text{ for all } i\}.
\]
An alcove contained in the dominant chamber is called a \emph{dominant alcove}. The \emph{fundamental alcove} is 
\[
    \mathcal{A}_0=\bigl\{x\in E \,\big|\, 0<\langle \alpha_i^\vee,x\rangle<1\ (1\le i\le r-1),\ \langle \theta^\vee,x\rangle<1\bigr\},
\]
where $\theta$ is the highest root. In particular, $\mathcal{A}_0$ is a dominant alcove.

Writing elements of $W$ as $w=w_0 t_\gamma$ with $w_0\in W_0$ and $\gamma\in Q$, the standard action of $W$ on $E$ is
\[
    w \cdot x \;=\; w_0(x) + \gamma, \qquad x \in E.
\]
This action preserves the affine Coxeter arrangement, and hence induces a simply transitive action on the set of alcoves \cite[Theorem 4.5]{humphreys-reflection-groups}.
\subsubsection{The level-$e$ action}\label{subsubsec:level-e-action}
Fix $e\in\Z_{\ge 2}$. The \emph{level-$e$ action} of $W$ on $E$ is defined by dilating the translation part by a factor of $e$:
\[
  w\star_e x \;=\; w_0(x) + e\gamma,\qquad x\in E.
\]
Taking $e=1$ recovers the standard action described above. Define the \emph{level-$e$ affine Coxeter arrangement} $\mathcal{H}^{(e)}$ by
\[
  \mathcal{H}^{(e)}
  :=
  \{H^{(e)}_{\alpha,k} \mid \alpha \in R^+,\, k \in \mathbb{Z}\},
  \qquad
  H^{(e)}_{\alpha,k}
  :=
  \{x \in E \mid \langle \alpha^\vee, x\rangle = ke\}.
\]
The connected components of $E \setminus \bigcup_{H\in\mathcal{H}^{(e)}}H$ are called \emph{$e$-alcoves}; those contained in the dominant chamber are called \emph{dominant $e$-alcoves}. The \emph{fundamental $e$-alcove} is
\[
  \mathcal{A}_0^{(e)}
  =
  \bigl\{
    x \in E
    \,\big|\,
    0 < \langle \alpha_i^\vee,x\rangle < e \ (1 \le i \le r-1),\
    \langle \theta^\vee,x\rangle < e
  \bigr\}.
\]
The level-$e$ action preserves $\mathcal{H}^{(e)}$, hence acts on the set of $e$-alcoves. This action is also simply transitive.
\subsubsection{Shi coefficients}
A convenient way to encode the position of an $e$-alcove is via \emph{Shi coefficients}, introduced in \cite{shi-alcove-shi-coefficients}. For $c\in\mathbb{R}$, the \emph{floor function} $\lfloor c \rfloor$ is the greatest integer less than or equal to $c$.

For any $e$-alcove $\mathcal{A}$ and positive root $\alpha \in R^+$, the value $\lfloor \langle \alpha^\vee, x \rangle / e \rfloor$ is constant for all $x \in \mathcal{A}$.
The \emph{level-$e$ Shi coefficient} of $\mathcal{A}$ with respect to $\alpha$ is:
\begin{equation*}
    k_\alpha^{(e)}(\mathcal{A}) \;:=\; \left\lfloor \frac{\langle \alpha^\vee, x \rangle}{e} \right\rfloor \qquad (\forall x \in \mathcal{A}).
\end{equation*}
The collection of integers $\{k_\alpha^{(e)}(\mathcal{A}) \mid \alpha \in R^+\}$ uniquely determines the $e$-alcove $\mathcal{A}$.
Specifically, two $e$-alcoves $\mathcal{A}$ and $\mathcal{A}'$ are the same if and only if
\begin{equation}\label{eq:shi-condition}
    k_\alpha^{(e)}(\mathcal{A})=k_\alpha^{(e)}(\mathcal{A}') \quad \text{for all } \alpha \in R^+.
\end{equation}

\subsubsection{Reindexing the anti-spherical KL polynomials}\label{subsubsec:relabel-antispherical-kl}
Recall from \autoref{subsec:parabolic-kl} that the anti-spherical Kazhdan–Lusztig polynomials $\mathfrak{n}_{y,x}(v)$ are indexed by elements $x, y \in {}^{I}W$.
For $e\in\Z_{\ge 2}$, the map
\[
  x \longmapsto \mathcal{A}_x := x \star_e \mathcal{A}_0^{(e)}
\]
defines a bijection between ${}^{I}W$ and the set of dominant $e$-alcoves; see \cite[Section 4]{soergel-kazhdan-lusztig-tilting}.
Thus, we may view the polynomials $\mathfrak{n}_{y,x}(v)$ as being indexed by pairs of dominant $e$-alcoves via the identification $\mathfrak{n}_{\mathcal{A}_y, \mathcal{A}_x}(v) := \mathfrak{n}_{y,x}(v)$.

We further relabel these polynomials using dominant weights. Following \cite[Section 5]{goodmanwenzl-crystal-bases}, we associate a unique dominant $e$-alcove $\mathcal{A}(\lambda)$ to every dominant weight $\lambda \in P^+$ as follows:
\begin{itemize}
  \item If $\lambda$ lies in the interior of a dominant $e$–alcove $\mathcal{A}$ (i.e., $\lambda \in \mathcal{A}$), set $\mathcal{A}(\lambda)=\mathcal{A}$.
  \item If $\lambda$ lies on some hyperplanes $H^{(e)}_{\alpha,k}$, let $\mathcal{A}(\lambda)$ be the unique dominant $e$–alcove such that $\lambda$ lies in its closure $\overline{\mathcal{A}(\lambda)}$ and $\lambda$ lies on the positive side of every hyperplane separating $\mathcal{A}(\lambda)$ from the origin.
\end{itemize}

For any pair of dominant weights $\mu, \lambda \in P^+$, we define the \emph{level-$e$ anti-spherical Kazhdan--Lusztig polynomial} denoted by $\mathfrak{n}^e_{\mu, \lambda}(v)$:
\begin{equation*}
  \mathfrak{n}^e_{\mu,\lambda}(v)\;:=\; \begin{cases}
      \mathfrak{n}_{\mathcal{A}(\mu),\mathcal{A}(\lambda)}(v) & \text{if } \mu \text{ and } \lambda \text{ are in the same } W\text{-orbit under } \star_e, \\
      0 & \text{otherwise.}
  \end{cases}
\end{equation*}
This allows us to work with anti-spherical Kazhdan–Lusztig polynomials indexed directly by dominant weights (and hence partitions via $\Omega$). We include the superscript $e$ to emphasize that the value depends on the level of the affine action used to define the $e$-alcoves.
See \autoref{fig:level3} and \autoref{fig:level4} for examples of associating a dominant weight to the corresponding $e$-alcove; we have shaded the $e$-alcove using a lighter shade of the color used for the point representing the dominant weight.
\subsection{Graded decomposition numbers}\label{subsec:decomposition-numbers}
Let $\mathbb{k}$ be a field, and let $\mathcal{H}_n=\mathcal{H}_{\mathbb{k},q}(\mathfrak{S}_n)$ be the Iwahori--Hecke algebra of $\mathfrak{S}_n$ over $\mathbb{k}$ with parameter $q$. Equivalently, $\mathcal{H}_n$ is the specialization of the Hecke algebra of type $A_{n-1}$; see \cite[Section 1.1]{bw-soergel-calculus} for details.

The algebra $\mathcal{H}_n$ has a distinguished family of modules $S^\lambda$, called \emph{Specht modules}, indexed by the partitions of $n$. If $q\in\mathbb{k}^\times$ has finite multiplicative order $e$, then the simple $\mathcal{H}_n$-modules are precisely the nonzero quotients $D^\mu$ of the Specht modules $S^\mu$, where $\mu$ runs over the $e$-regular partitions of $n$. The \emph{decomposition numbers} are the multiplicities $[S^\lambda:D^\mu]\in\mathbb{Z}_{\geq 0}$ of the simple module $D^\mu$ as a composition factor of the Specht module $S^\lambda$. We refer to \cite[Section 2.2 \& Chapter 6]{mathas-iwahori-hecke} for more details.

By works of \cite{khovanovlauda-klr-1,bkw-graded-specht,bk-graded-decomposition-number}, Specht modules and simple modules admit graded lifts, denoted by $\widetilde{S}^\lambda$ and $\widetilde{D}^\mu$, respectively. The \emph{graded decomposition numbers} are polynomials in $\mathbb{Z}[v, v^{-1}]$:
\[
  d^e_{\lambda\mu}(v)\;=\;\sum_{d\in\mathbb{Z}} [\widetilde{S}^\lambda:\widetilde{D}^\mu\langle d\rangle]\,v^d,
\]
where $\langle\cdot\rangle$ denotes the grading shift. Specializing $v=1$ gives $d^e_{\lambda\mu}(1)=[S^\lambda:D^\mu]$.

A fundamental feature of the representation theory of $\mathcal{H}_n$ is its decomposition into a direct sum of indecomposable two-sided ideals, called \emph{blocks}. We say that two partitions $\lambda$ and $\mu$ lie in the same block if the corresponding Specht modules $S^\lambda$ and $S^\mu$ belong to the same block of the algebra (i.e., all their composition factors lie in the same block).
By \cite[Corollary 5.38]{mathas-iwahori-hecke}, $\lambda$ and $\mu$ lie in the same block if and only if
\[
    \operatorname{core}_e(\lambda) \;=\; \operatorname{core}_e(\mu) \quad \text{and} \quad w_e(\lambda) \;=\; w_e(\mu).
\]
This implies that $d^e_{\lambda\mu}(v) = 0$ unless $\lambda$ and $\mu$ share the same $e$-core and $e$-weight.

\subsection{Empty runner removal theorem}
Let $e\in\Z_{\ge 2}$ and $n\in\Z_{\ge 0}$. Let $\lambda$ and $\mu$ be two partitions of $n$ lying in the same block.
Choose $r\ge \max\{\ell(\lambda),\ell(\mu)\}$ and form the abaci $\operatorname{Ab}_{e}^{r}(\lambda)$ and $\operatorname{Ab}_{e}^{r}(\mu)$.

Fix $j\in\{0,1,\dots,e-1\}$ and insert a new empty runner immediately to the left of the $j$--runner. The resulting abacus configurations yield new partitions $\lambda^+$ and $\mu^+$.

\begin{Theorem}[\cite{jamesmathas-empty-runner-removal}, Theorem 3.2]\label{thm:empty-runner-removal}
    In characteristic $0$, suppose $\lambda$ and $\mu$ lie in the same block and $\mu$ is $e$–regular.
    Form $\lambda^+$ and $\mu^+$ as above, then
    \[
      d^{e}_{\lambda,\mu}(v) \;=\; d^{e+1}_{\lambda^+,\mu^+}(v).
    \]
\end{Theorem}


\begin{Remark}
    The condition that $\mu$ is $e$-regular is necessary only in the context of the Iwahori–Hecke algebra (where simple modules are indexed by $e$-regular partitions). \autoref{thm:empty-runner-removal} was originally established in the context of $q$-Schur algebras, where simple modules exist for all partitions and the regularity condition can be dropped.
\end{Remark}

As noted in \cite{jamesmathas-empty-runner-removal}, Goodman suggested that \autoref{thm:empty-runner-removal} can be deduced from the following theorem.

\begin{Theorem}[\cite{goodmanwenzl-crystal-bases}]\label{thm:goodman-wenzl}
    In characteristic $0$, let $\lambda$ and $\mu$ be two partitions of the same size with no more than $r$ rows, and assume $\mu$ is $e$-regular. Then the graded decomposition number is given by the anti-spherical Kazhdan–Lusztig polynomial evaluated at the corresponding weights:
    \[
    d^e_{\lambda\mu}(v) \;=\; \mathfrak{n}^e_{\Omega(\lambda),\Omega(\mu)}(v).
    \]
\end{Theorem}

In \cite[Section 2.3]{fayers-full-runner-removal}, the author briefly describes the idea of how to realize Goodman's remark. However, a detailed proof is not provided there. We provide the proof in \autoref{sec:proof-empty}.
\section{Patterns of dominant weights}\label{sec:pattern-dominant-weights}
In \autoref{subsec:map-omega}, we defined the map $\Omega$ which sends partitions to the dominant weight lattice $P^+$. Since the blocks of the Hecke algebra are parametrized by pairs consisting of an $e$–core and an $e$–weight (see \autoref{subsec:core-quotient} and \autoref{subsec:decomposition-numbers}), it is natural to investigate the geometric pattern of the dominant weights corresponding to partitions with a fixed $e$-weight. In this section, we describe and analyze these patterns.

\begin{Definition}\label{def:generic-triple}
    A \emph{generic triple} is a triple $\triple\in\Z_{\ge 0}$ such that $e>r\ge 3$.
\end{Definition}

For any generic triple $\triple$, let $\mathscr{P}_{r,e,w}$ denote the set of partitions with at most $r$ parts and $e$–weight $w$. In particular, $\mathscr{P}_{r,e,0}$ consists of the $e$–core partitions of length at most $r$.

Let $P^+$ be the dominant weight lattice of $\mathfrak{sl}_r$. The corresponding set of dominant weights of $\mathscr{P}_{r,e,w}$ under $\Omega$ is:
\begin{equation*}
    \mathcal{W}_{r,e,w} \;:=\; \Omega(\mathscr{P}_{r,e,w}) \;\subset\; P^+.
\end{equation*}
Our primary goal is to characterize the set $\mathcal{W}_{r,e,w}$.
\begin{Example}
    We draw $\mathcal{W}_{3,10,8}$ in the dominant weight lattice in \autoref{fig:r3-e10-w8-black}, where each black ball represents a dominant weight in $\mathcal{W}_{3,10,8}$.
    \begin{figure}[htbp]
        \centering
        \captionsetup{type=figure}
        \includegraphics[height=6cm,keepaspectratio]{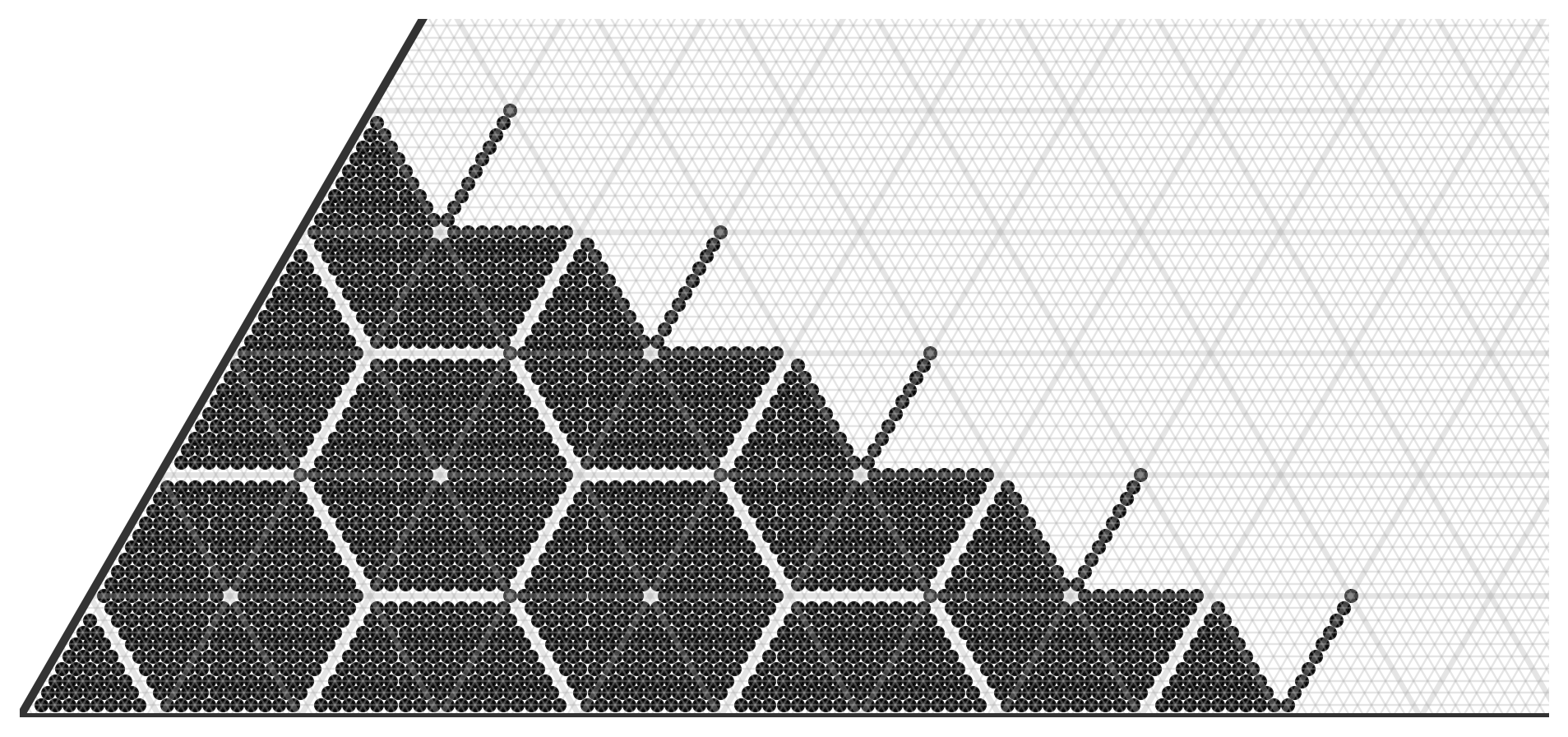}
        \caption{$r=3$, $e=10$, $w=8$}
        \label{fig:r3-e10-w8-black}
    \end{figure}
\end{Example}
\subsection{Simplicial structure of $e$-cores}
We start with the case of $e$-cores (i.e., $w=0$). To simplify notation, we define:
\[
    \mathscr{P}_{r,e} \;:=\; \mathscr{P}_{r,e,0} \quad \text{and} \quad \mathcal{W}_{r,e} \;:=\; \mathcal{W}_{r,e,0}.
\]
Thus $\mathscr{P}_{r,e}$ is the set of $e$-core partitions with at most $r$ parts, and $\mathcal{W}_{r,e} = \Omega(\mathscr{P}_{r,e})$ is the corresponding set of dominant weights. We fix a generic triple $\triple[r][e][0]$ throughout this section.

For each composition $\mu=(\mu_1,\dots,\mu_j)$ of $r$ (this implies $1\le j\leq r$), we define a subset of partitions $\mathscr{P}_{r,e}(\mu) \subset \mathscr{P}_{r,e}$ as follows. Consider the family of $e$–abaci obtained by choosing $j$ distinct runners $0\le s_1<\cdots<s_j\le e-1$ and placing $\mu_i$ beads on the $s_i$–runner (as high as possible) for each $i=1,\dots,j$, while leaving the remaining runners empty. Each such abacus corresponds to a unique $e$-core partition; let $\mathscr{P}_{r,e}(\mu)$ be the set of all such partitions and define the corresponding set of dominant weights by:
\begin{equation*}
    \mathcal{W}_{r,e}(\mu) \;:=\; \Omega(\mathscr{P}_{r,e}(\mu)).
\end{equation*}

Since every $e$-core partition with at most $r$ parts corresponds to an $e$-abacus with exactly $r$ beads placed on some flush runners, we have the following decomposition:
\begin{equation*}
    \mathscr{P}_{r,e} \;=\; \bigsqcup_{\mu \models r} \mathscr{P}_{r,e}(\mu)
    \quad \text{and} \quad
    \mathcal{W}_{r,e} \;=\; \bigsqcup_{\mu \models r} \mathcal{W}_{r,e}(\mu),
\end{equation*}

Before stating the main result of this subsection, we formalize our geometric terminology.

\begin{Definition}\label{def:simplex-lattice-points}
    Let $d \ge 0$ and $L > 0$ be integers. We define the \emph{standard $d$-dimensional simplex of dilation factor $L$} to be the set:
    \[
        \Delta^d(L) \;=\; \bigl\{ (x_0, \dots, x_d) \in \mathbb{R}^{d+1}_{\ge 0} \;\big|\; x_0 + \dots + x_d = L \bigr\}.
    \]
    The \emph{lattice points} of $\Delta^d(L)$ are the points in $\Delta^d(L) \cap \mathbb{Z}^{d+1}$.
    
\end{Definition}

\begin{Definition}\label{def:affine-linear-map}
    Let $V$ and $W$ be vector spaces. A map $\phi: V \to W$ is called \emph{affine linear} if there exists a linear transformation $T: V \to W$ and a constant vector $c \in W$ such that
    \[
        \phi(v) \;=\; T(v) + c \qquad \text{for all } v \in V.
    \]
\end{Definition}

In the context of the weight lattice $P$, we say that a subset $S \subset P$ forms the \emph{lattice points of a $d$-dimensional simplex of dilation factor $L$} if there exists an injective affine linear map $\phi: \mathbb{R}^{d+1} \to P \otimes \mathbb{R}$ such that $S = \phi(\Delta^d(L) \cap \mathbb{Z}^{d+1})$.

We require the following lemma to handle the translation-invariance of the abacus under $\Omega$.

\begin{Lemma}\label{lm:omega-shift-invariance}
    Let $\beta=(\beta_1,\dots,\beta_r)$ and $\gamma=(\gamma_1,\dots,\gamma_r)$ be beta numbers. The map $\Omega$, viewed as a function of the beta numbers via $\Omega(\beta) = \sum_{i=1}^{r-1} (\beta_i - \beta_{i+1}) \Lambda_i$, is only invariant under global shift. That is,
    $\Omega(\beta)=\Omega(\gamma)$
    if and only if there exists an integer $k \in \mathbb{Z}$ such that $\gamma_i = \beta_i + k$ for all $1 \le i \le r$.
\end{Lemma}

\begin{proof}
    If $\gamma_i = \beta_i + k$ for all $i$, then $\gamma_i - \gamma_{i+1} = (\beta_i + k) - (\beta_{i+1} + k) = \beta_i - \beta_{i+1}$. Since $\Omega$ depends only on these differences, $\Omega(\gamma) = \Omega(\beta)$.
    
    Conversely, if $\Omega(\gamma)=\Omega(\beta)$, then the coefficients in the fundamental weight basis must match:
    \[
        \gamma_i-\gamma_{i+1} \;=\; \beta_i-\beta_{i+1} \quad (1\le i\le r-1).
    \]
    This implies $\gamma_i - \beta_i = \gamma_{i+1} - \beta_{i+1}$ for all $i$. The difference $\gamma_i - \beta_i$ is a constant $k$ independent of $i$, proving the claim.
\end{proof}

\begin{Theorem}\label{thm:simplex-of-core-weights}
    Fix a generic triple $\triple[r][e][0]$, and let $\mu$ be a composition of $r$ of length $j\, (1\le j\le r)$. Then $\mathcal{W}_{r,e}(\mu)$ forms the lattice points of a $(j-1)$–dimensional simplex of dilation factor $e-j$.
\end{Theorem}

\begin{proof}
    Fix the composition $\mu$. Partitions in $\mathscr{P}_{r,e}(\mu)$ are determined by the positions of the $j$ non-empty runners, denoted $0 \le s_1 < s_2 < \dots < s_j \le e-1$.
    
    We define the \emph{gap variables} $g_0, \dots, g_{j}$ representing the spacing between non-empty runners:
    \begin{align*}
        g_0 &= s_1, \\
        g_k &= s_{k+1} - s_k - 1 \quad \text{for } 1 \le k \le j-1, \\
        g_j &= e - 1 - s_j.
    \end{align*}
    These variables satisfy $g_k \in \mathbb{Z}_{\ge 0}$ and $\sum_{k=0}^j g_k = e - j$.
    
    By \autoref{lm:omega-shift-invariance}, the map $\Omega$ is invariant under global shift. Hence we may shift all beads to the left by $s_1$, which does not change the image. With this normalization, $s_1=0=g_0$.

    The remaining variables $\mathbf{g} = (g_1, \dots, g_j)$ satisfy $\sum_{k=1}^j g_k = e-j$, corresponding to the lattice points of the standard simplex $\Delta^{j-1}(e-j)$. It remains to show that $\Omega$ induces an injective, affine linear map on this domain.
    
    Let $\kappa \in \mathscr{P}_{r,e}(\mu)$ be the core partition corresponding to $\mathbf{g}$ and let $\beta_i$ be the corresponding beta numbers of $\kappa$. The corresponding dominant weight is $\Omega(\kappa) = \sum_{i=1}^{r-1} (\beta_i - \beta_{i+1}) \Lambda_i$. 
    Suppose the beads $\beta_i$ and $\beta_{i+1}$ lie on the chosen runners indexed by  $s_u$ and $s_v$, respectively, where $1 \le u,v \le j$. Then
    \begin{equation}\label{eq:beta-diff}
        \beta_i - \beta_{i+1} \;=\; 
        \begin{dcases}
            u - v + \sum_{v \le t < u} g_t & \text{if } u > v \text{ (same rows),} \\
            j + u - v + \sum_{v \le t \le j} g_t + \sum_{1 \le t < u} g_t & \text{if } u \leq v \text{ (different rows).}
        \end{dcases}
    \end{equation}
    In both cases, $\beta_i - \beta_{i+1}$ is an affine linear function of the variables $\mathbf{g}$. Thus, $\Omega$ is an affine linear map in the sense of \autoref{def:affine-linear-map}.
    
    To verify injectivity, consider two distinct elements $(g_0,g_1, \dots, g_j)=\mathbf{g} \neq \mathbf{g}'=(g_0', g_1', \dots, g_j')$ in the domain. We may normalize as above, and hence $g_0=g_0'=0$. Let $k = \min \{1\le t \le j\mid g_t \neq g_t' \}$. Consider the top beads located on the $k$-th non-empty runner and the ${k+1}$-th non-empty runner. By \autoref{eq:beta-diff}, the difference between their values is $g_k+1$ and $g_k'+1$, respectively.
    Thus the coordinate associated with this pair differs, implying $\Omega(\mathbf{g}) \neq \Omega(\mathbf{g}')$.
\end{proof}

\begin{Remark}
    We do not need the assumption $r< e$ to show that $\mathcal{W}_{r,e}(\mu)$ is a simplex. However, without this condition the claim that $\mathcal{W}_{r,e}(\mu)$ has dimension $j-1$ may fail. For example, if $\ell(\mu)=e<r$, then $\mathcal{W}_{r,e}(\mu)$ collapses to a single point; see \autoref{eg:dimension-collaps}. 
\end{Remark}

\begin{Example}\label{eg:dimension-collaps}
    Consider the case $\mathfrak{sl}_3$ (so $r=3$) with $e=2$. Then $|\mathcal{W}_{3,2}(\mu)|=1$ for $\mu\in\{(2,1),(1,2)\}$, and $\mathcal{W}_{3,2}(1,1,1)=\varnothing$. See \autoref{fig:degenerate-sl3-e2w0}. Similarly, see \autoref{fig:degenerate-sl3-e3w0} for the case $\mathfrak{sl}_3$ with $e=3$ where $|\mathcal{W}_{3,3}(1,1,1)|=1$.
    
    \begin{figure}[htbp]
      \centering
      \begin{minipage}[t]{0.49\textwidth}
        \centering
        \captionsetup{type=figure}
        \includegraphics[height=4cm,keepaspectratio]{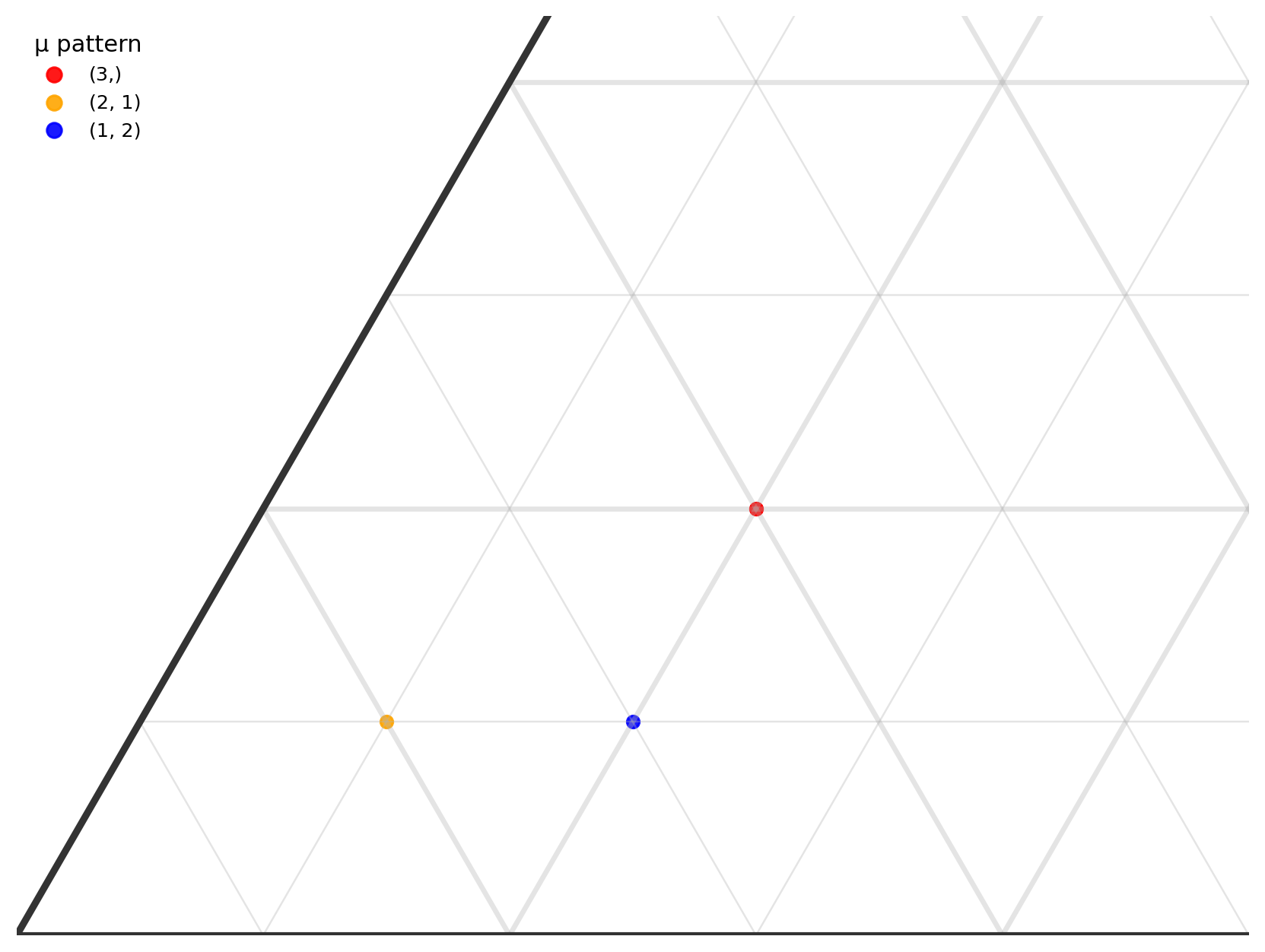}
        \caption{$r=3$, $e=2$, $w=0$}
        \label{fig:degenerate-sl3-e2w0}
      \end{minipage}\hfill
      \begin{minipage}[t]{0.49\textwidth}
        \centering
        \captionsetup{type=figure}
        \includegraphics[height=4cm,keepaspectratio]{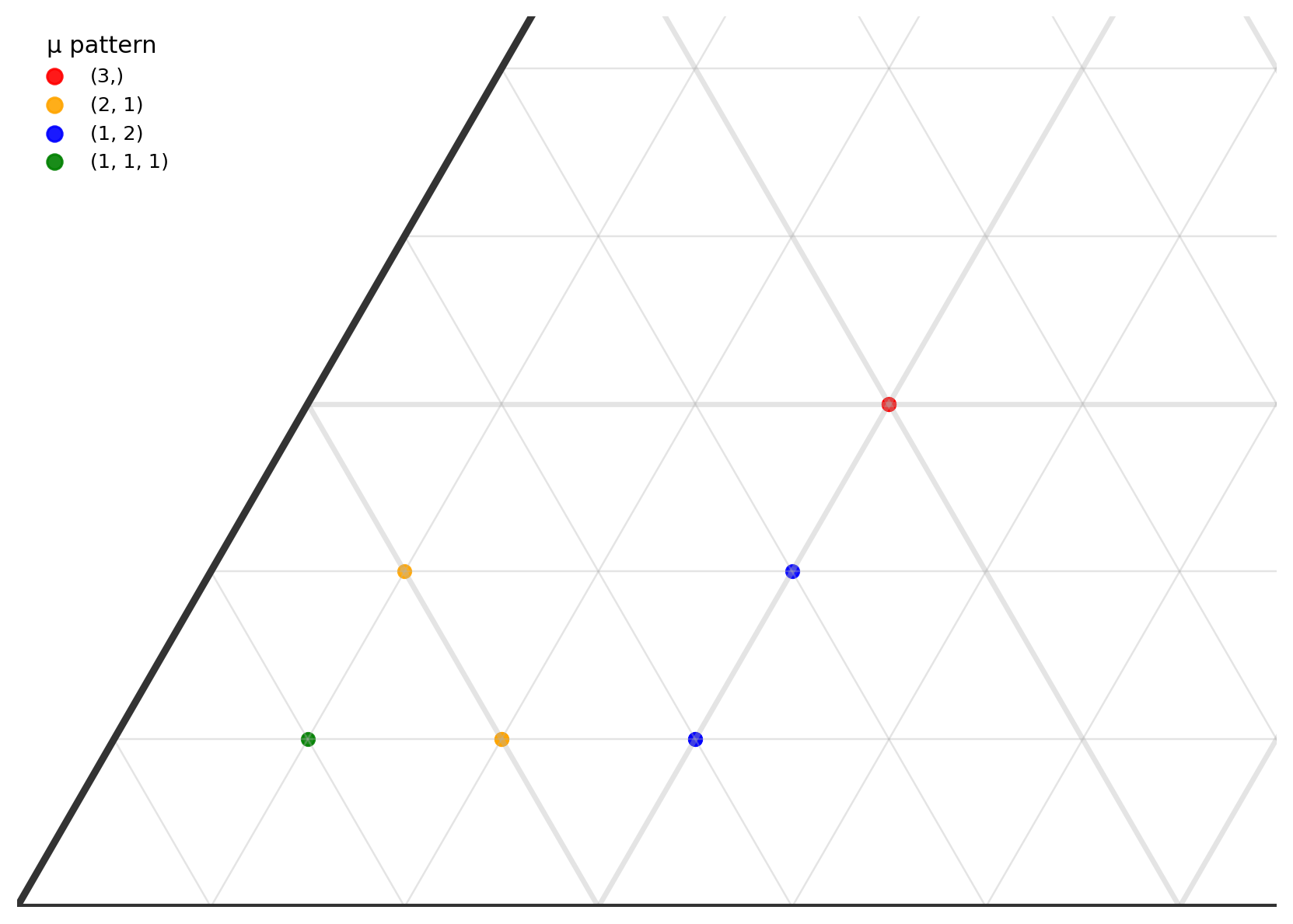}
        \caption{$r=3$, $e=3$, $w=0$}
        \label{fig:degenerate-sl3-e3w0}
      \end{minipage}
    \end{figure}
\end{Example}
\subsection{Simplicial Decomposition of \texorpdfstring{$\mathcal{W}_{r,e,w}$}{Wr,e,w}}\label{subsec:decomposition-w}
In this section, we fix a generic triple $\triple$. Our aim is to generalize the case $w=0$ to the case of arbitrary weight $w \in \mathbb{N}$.
For a fixed composition $\mu \models r$ of length $j$, define a subset of $\mathscr{P}_{r,e,w}$ as follows:
\begin{equation*}
    \mathscr{P}_{r,e,w}(\mu) \;:=\; \bigl\{ \lambda \in \mathscr{P}_{r,e,w} \;\mid\; \operatorname{core}_e(\lambda) \in \mathscr{P}_{r,e}(\mu) \bigr\}.
\end{equation*}

For any partition $\lambda \in \mathscr{P}_{r,e,w}(\mu)$, the $e$-quotient (see \autoref{eq:e-quotient}) encodes the vertical moves of beads relative to the $e$-core to have $e$-weight $w$. Since the core lies in $\mathscr{P}_{r,e}(\mu)$, the beads of $\lambda$ are supported on exactly $j$ runners. Consequently, the $e$-quotient has empty entries corresponding to the other empty runners. We define the \emph{restricted $e$-quotient} of $\lambda$ to be the subsequence of the $e$-quotient corresponding to these $j$ non-empty runners, denoted by the tuple $(\lambda^{(1)}, \dots, \lambda^{(j)})$. Since the $i$-th non-empty runner (counting from left to right) carries exactly $\mu_i$ beads, the corresponding partition $\lambda^{(i)}$ must satisfy the length constraint $\ell(\lambda^{(i)}) \le \mu_i$.

Recall from \autoref{subsec:multipartitions} that these tuples correspond to the $j$-partitions of $w$ of type $\mu$, and $A(\mu; w)$ is the number of such $j$-partitions. Set $\mathcal{W}_{r,e,w}(\mu):=\Omega\big(\mathscr{P}_{r,e,w}(\mu)\big)$. Then we have the decomposition:
\begin{equation*}
    \mathscr{P}_{r,e,w} \;=\; \bigsqcup_{\mu \models r} \mathscr{P}_{r,e,w}(\mu)
    \quad \text{and} \quad
    \mathcal{W}_{r,e,w} \;=\; \bigsqcup_{\mu \models r} \mathcal{W}_{r,e,w}(\mu).
\end{equation*}

\begin{Theorem}[\textbf{Simplicial Decomposition}]\label{thm:simplicial-decomposition}
    Fix a generic triple $\triple$. Let $\mu$ and $\nu$ be distinct compositions of $r$. Then
    \begin{enumerate}
        \item\label{itm:W-disjoint}
        $\mathcal{W}_{r,e,w}(\mu)\cap \mathcal{W}_{r,e,w}(\nu)=\varnothing$.
        
        \item\label{itm:W-copies}
        The set $\mathcal{W}_{r,e,w}(\mu)$ is a disjoint union of copies of the simplex $\mathcal{W}_{r,e}(\mu)$, and the number of copies is exactly $A(\mu;w)$.
    \end{enumerate}
\end{Theorem}

\begin{proof}
    \textbf{(a)}
    Take $\lambda\in \mathscr{P}_{r,e,w}(\mu)$ and $\eta\in \mathscr{P}_{r,e,w}(\nu)$ with $\Omega(\lambda)=\Omega(\eta)$. Let
    \[
        \beta(\lambda)=(\beta_1,\dots,\beta_r),\qquad
        \beta(\eta)=(\beta_1',\dots,\beta_r')
    \]
    be the corresponding $r$-beta numbers in decreasing order.
    By \autoref{lm:omega-shift-invariance}, there exists an integer $k$ such that $\beta_i'=\beta_i+k$ for all $1\le i\le r$.

    This corresponds to a global shift of all runners to the right by $k$ (or to the left by $-k$), which does not change the number of non-empty runners. Hence $\ell(\mu)=\ell(\nu)$, and $\nu$ is obtained from $\mu$ by a cyclic permutation of its parts:
    \[
        \nu=(\mu_i,\mu_{i+1},\dots,\mu_j,\mu_1,\dots,\mu_{i-1})
    \]
    for some $1\le i\le j$.
    
    If $i=1$ then $\nu=\mu$, as desired. Suppose $i\neq 1$.
    A nontrivial cyclic shift moves at least one non-empty runner from the rightmost to the left, which strictly increases the the $e$-weight.
    Since both $\lambda$ and $\eta$ lie in $\mathscr{P}_{r,e,w}$ and have the same $e$-weight $w$, this is impossible. Therefore $i=1$ and $\mu=\nu$.

    \medskip
    \textbf{(b)}
    Fix the composition $\mu$ of length $j$. Then by \autoref{thm:simplex-of-core-weights} $\mathcal{W}_{r,e}(\mu)$ is a $(j-1)$-dimensional simplex where the degrees of freedom come from choosing the runner gaps.
    For $w\ge0$, any $\lambda\in \mathscr{P}_{r,e,w}(\mu)$ is obtained from some core $\lambda^0\in \mathscr{P}_{r,e}(\mu)$ by moving beads downwards on the chosen $j$ runners, and the moves are recorded by the restricted $e$-quotient, which is a $j$-partition of $w$ of type $\mu$. 
    
    Thus each $j$-partition of $w$ of type $\mu$ determines a copy of the simplex $\mathcal{W}_{r,e}(\mu)$ inside $\mathcal{W}_{r,e,w}(\mu)$. Since there are $A(\mu;w)$ such $j$-partitions, there are $A(\mu;w)$ such copies.
    
    To see they are disjoint, let $\lambda,\eta\in \mathscr{P}_{r,e,w}(\mu)$ define the same dominant weight. Let $\beta_i(\lambda)$ and $\beta_i(\eta)$ be their beta numbers. By \autoref{lm:omega-shift-invariance}, there exists some integer $k$ such that $\beta_i(\eta)=\beta_i(\lambda)+k$ for all $1\leq i\leq r$. Write $k=qe+s$ with $0\le s<e$. We may assume $q\geq 0$.
    
    First, consider the term $qe$. Increasing a beta number by $e$ corresponds to sliding a bead down one row on its runner. If we increase every bead $\beta_i(\lambda)$ by $qe$, then the row index of every bead increase by $q$. Since there are $r$ beads in total, this operation increases the total $e$-weight by $q \cdot r$. Since $\lambda$ and $\eta$ both have fixed $e$-weight $w$, we must have $q=0$.
    
    Now we are left with the term $k=s$ where $0 \le s < e$. This corresponds to a cyclic shift of the runner indices by $s$. However, by definition of the simplex $\mathcal{W}_{r,e,0}(\mu)$, we have fixed the runner configuration (specifically, we normalized the first non-empty runner to be at position $0$ or fixed the gap variables). As argued in part (a), such a shift strictly changes the $e$-weight unless $s=0$.
    Thus $k=0$ and $\lambda=\eta$. Therefore, the copies are pairwise disjoint.
\end{proof}

\begin{Example}\label{eg:simplicial-decomposition-r3e9}
    In \autoref{fig:r3e8-grid-0-9}, we plot the sets $\mathcal{W}_{3,10,w}$ for weights $w$ ranging from $0$ to $9$. This figure is drawn in the dominant weight lattice of $\mathfrak{sl}_3$, where each colored ball represents a dominant weight in $\mathcal{W}_{3,10,w}$. We use distinct colors to distinguish the subsets $\mathcal{W}_{3,10,w}(\mu)$ arising from the simplicial decomposition, corresponding to the types $\mu\in\{(3),(2,1),(1,2),(1,1,1)\}$.
\end{Example}
\begin{Example}\label{eg:labels-of-simplicial-composition}
    For $\triple=(3,8,8)$, consider $W_{r,e,w}(\mu)$ with $\mu=(1,1,1)$. We label each copy of $W_{r,e}(\mu)$ by the corresponding $3$–partition of type $\mu$, see \autoref{fig:labels-of-simplicial-composition-1}. Similarly, see \autoref{fig:labels-of-simplicial-composition-2} for the case $(r,e,w)=(3,12,10)$.
    \begin{figure}[htbp]
        \centering
        \captionsetup{type=figure}
        \includegraphics[height=5cm,keepaspectratio]{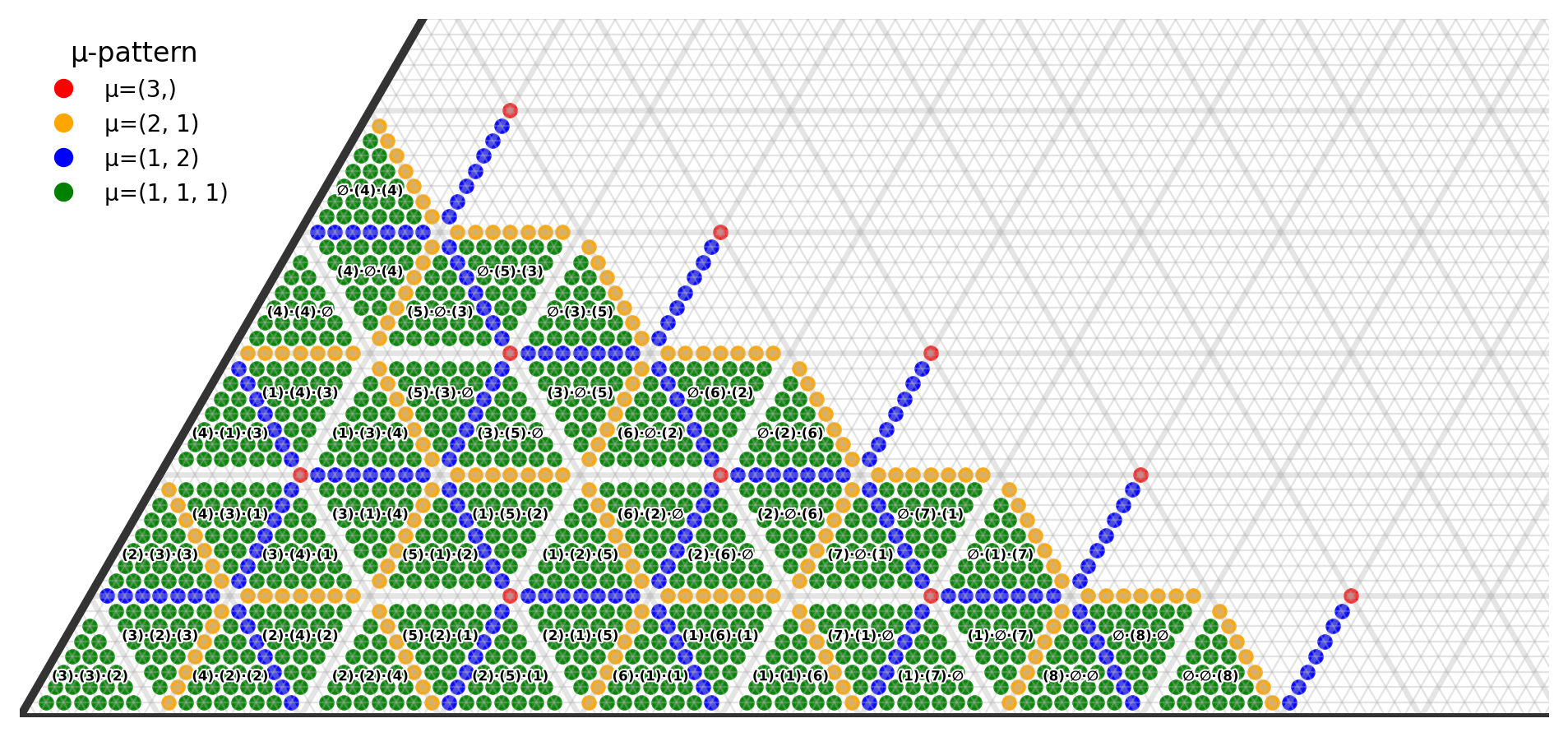}
        \caption{$r=3$, $e=8$, $w=8$}
        \label{fig:labels-of-simplicial-composition-1}
    \end{figure}
    \begin{figure}[htbp]
        \centering
        \captionsetup{type=figure}
        \includegraphics[height=5cm,keepaspectratio]{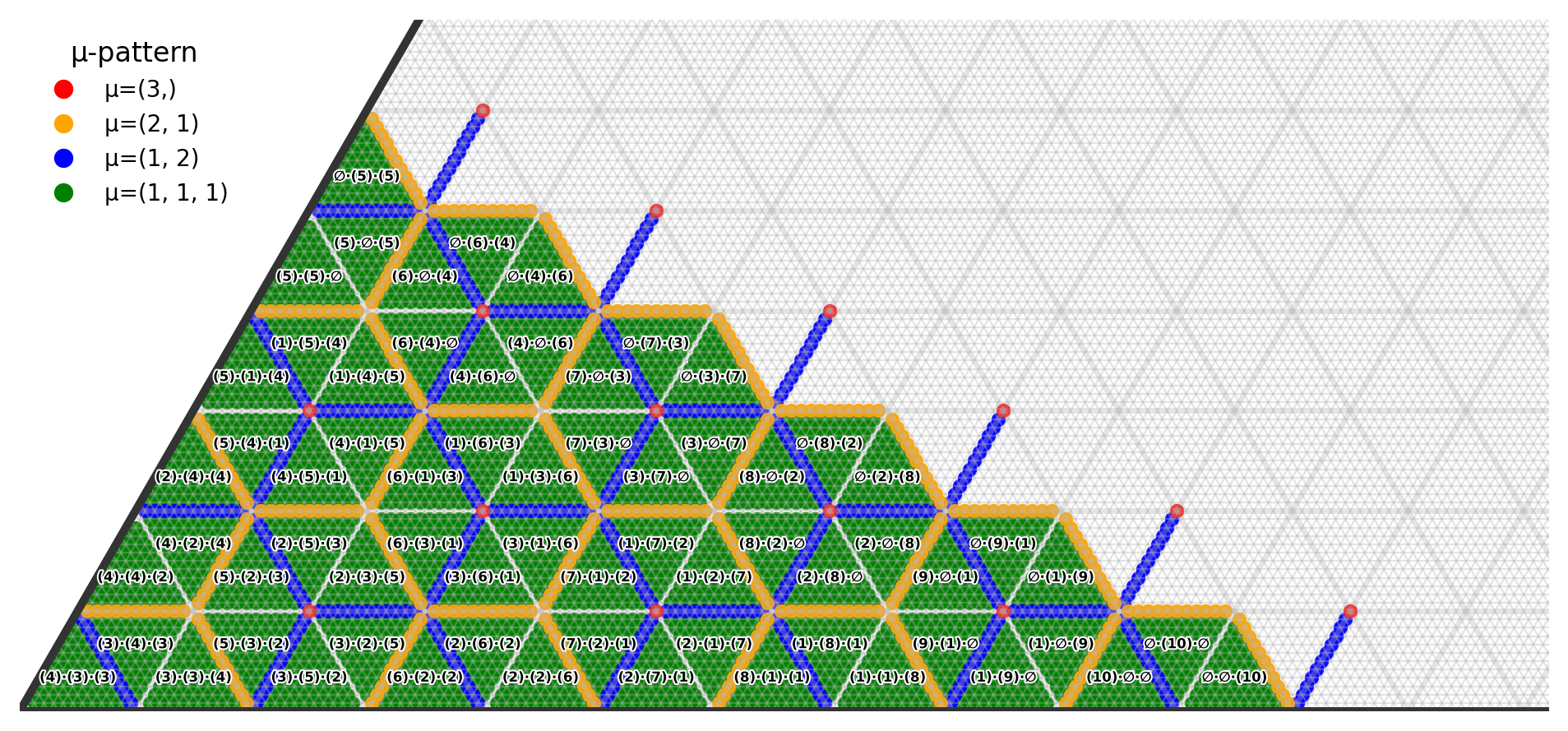}
        \caption{$r=3$, $e=12$, $w=10$}
        \label{fig:labels-of-simplicial-composition-2}
    \end{figure}
\end{Example}
\begin{Corollary}\label{cor:simplex-points-counting}
    The number of lattice points in a $(j-1)$-dimensional simplex of dilation factor $e-j$ is given by the binomial coefficient $\binom{e-1}{j-1}$.
\end{Corollary}

\begin{proof}
    By \autoref{thm:simplex-of-core-weights}, the set of weights $\mathcal{W}_{r,e,0}(\mu)$ corresponds bijectively to the set of gap variable tuples $\mathbf{d}=(d_1, \dots, d_j)$ satisfying $d_i \in \mathbb{Z}_{\ge 0}$ and $\sum_{i=1}^j d_i = e-j$.
    
    This counts the number of non-negative integer solutions to a linear equation, which is a standard stars-and-bars problem. Alternatively, recall that these tuples parametrise the choice of $j$ distinct runners $0=s_1 < s_2 < \dots < s_j \le e-1$. Since the first runner is fixed at $0$, we must choose the remaining $j-1$ distinct indices from the set $\{1, \dots, e-1\}$. The number of such choices is exactly $\binom{e-1}{j-1}$.
\end{proof}
\begin{Corollary}\label{cor:num-of-core-weights}
    For a generic triple $\triple[r][e][0]$, we have
    \[
        |\mathcal{W}_{r,e}| \;=\; \binom{e+r-2}{r-1}.
    \]
\end{Corollary}

\begin{proof}
    By \autoref{thm:simplicial-decomposition}, $\mathcal{W}_{r,e}$ is the disjoint union of the sets $\mathcal{W}_{r,e}(\mu)$ as $\mu$ ranges over all compositions of $r$.
    If $\mu$ has length $j$, \autoref{cor:simplex-points-counting} states that $|\mathcal{W}_{r,e}(\mu)| = \binom{e-1}{j-1}$.
    
    The number of compositions of $r$ into exactly $j$ parts is $\binom{r-1}{j-1}$.
    Summing over all possible lengths $1 \le j \le r$, we obtain:
    \begin{align*}
        |\mathcal{W}_{r,e,0}| \;&=\; \sum_{j=1}^{r} \binom{r-1}{j-1}\binom{e-1}{j-1} \\
        &=\; \sum_{k=0}^{r-1} \binom{r-1}{k}\binom{e-1}{k} \qquad \\
        &=\; \sum_{k=0}^{r-1} \binom{r-1}{r-1-k}\binom{e-1}{k} \qquad (\text{symmetry of binomial coeffs}) \\
        &=\; \binom{(r-1)+(e-1)}{r-1},
    \end{align*}
    where the last equality follows from Vandermonde's Identity.
\end{proof}
\begin{Theorem}\label{thm:composition-sum}
    Fix a generic triple $\triple$. Then
    \begin{equation}\label{eq:composition-sum}
        |\mathcal{W}_{r,e,w}| \;=\; \sum_{\mu \models r} A(\mu;w) \binom{e-1}{\ell(\mu)-1},
    \end{equation}
\end{Theorem}

\begin{proof}
    By \autoref{thm:simplicial-decomposition}(\ref{itm:W-disjoint}), $\mathcal{W}_{r,e,w}$ is the disjoint union of the sets $\mathcal{W}_{r,e,w}(\mu)$ as $\mu$ runs over all compositions of $r$. Thus,
    \[
        |\mathcal{W}_{r,e,w}| \;=\; \sum_{\mu \models r} |\mathcal{W}_{r,e,w}(\mu)|.
    \]
    Fix a composition $\mu=(\mu_1,\ldots,\mu_j) \models r$. By \autoref{cor:simplex-points-counting}, the size of the simplex $\mathcal{W}_{r,e}$ is $\binom{e-1}{j-1}$.
    By \autoref{thm:simplicial-decomposition}(\ref{itm:W-copies}), $\mathcal{W}_{r,e,w}(\mu)$ is a disjoint union of copies of this simplex. The number of copies is exactly $A(\mu;w)$. Therefore:
    \begin{equation}\label{eq:mu-summand}
        |\mathcal{W}_{r,e,w}(\mu)| \;=\; A(\mu;w) \binom{e-1}{j-1}.
    \end{equation}
    Summing \eqref{eq:mu-summand} over all compositions $\mu \models r$ yields the desired formula.
\end{proof}

\subsection{Stingray and Regular patterns}\label{subsec:stingray-hexagon}
In this section, we analyze the geometric structure of the sets $\mathcal{W}_{r,e,w}$. While our primary motivation is to explain the visual patterns observed in the case $r=3$ (such as the ``stingray'' and ``hexagon'' shapes seen in \autoref{fig:r3e8-grid-0-9}), most of the underlying structural results hold for arbitrary $r$. We therefore formulate our results for general $r \ge 3$ wherever possible. 


We first establish a recurrence relation connecting patterns of weight $w$ to those of weight $w+1$. This explains how the pattern ``grows'' as the $e$-weight increases.

\begin{Lemma}\label{lm:old-become-right}
    Fix a generic triple $\triple$. The set  
    $\mathcal{W}_{r,e,w}$ embeds into $\mathcal{W}_{r,e,w+1}$ via a translation by $e\Lambda_1$.
    Specifically, a point $\sum_{i=1}^{r-1} a_i\Lambda_i$ lies in $\mathcal{W}_{r,e,w}$ if and only if the point
    $(a_1+e)\Lambda_1 + \sum_{i=2}^{r-1} a_i\Lambda_i$
    lies in $\mathcal{W}_{r,e,w+1}$. Geometrically, this image of $\mathcal{W}_{r,e,w}$ is precisely $\mathcal{W}_{r,e,w+1}\cap H^+_{\alpha_1, e}$.
\end{Lemma}

\begin{proof}
    Let $\lambda \in \mathscr{P}_{r,e,w}$ be a partition such that $\Omega(\lambda) = \sum a_i \Lambda_i$. Consider the $e$-abacus of $\lambda$ with $r$ beads. The coefficients of $\Omega(\lambda)$ are given by the differences of beta numbers: $a_i = \beta_i(\lambda) - \beta_{i+1}(\lambda)$.

    We modify the abacus by sliding down the first bead $\beta_1$ (the largest beta number) by one. This operation replaces $\beta_1$ with $\beta'_1 = \beta_1 + e$, leaving all other beads fixed ($\beta'_j = \beta_j$ for $j > 1$). Since $\beta_1 > \beta_2$, clearly $\beta'_1 > \beta_2$, so the new sequence of beta numbers is still strictly decreasing and corresponds to a new partition $\lambda'$.
    Write $\Omega(\lambda') = \sum a_i' \Lambda_i$. The new coefficients $a_i'$ are:
    \begin{align*}
        a'_1 &= \beta'_1 - \beta'_2 = (\beta_1 + e) - \beta_2 = a_1 + e, \\
        a'_{i} &= \beta'_i - \beta'_{i+1} = \beta_i - \beta_{i+1} = a_i \quad \text{for } i \ge 2.
    \end{align*}
    By the definition of $e$-weight, $w_e(\lambda') = w_e(\lambda) + 1 = w+1$. Thus $\Omega(\lambda') \in \mathcal{W}_{r,e,w+1}$.

    Conversely, suppose $\sum b_i \Lambda_i\in \mathcal{W}_{r,e,w+1}$ with $b_1> e$. Suppose that this dominant weight arises from a partition $\eta$ whose beta numbers are $\gamma_1, \dots, \gamma_r$. Then the condition $b_1> e$ implies $\gamma_1 - \gamma_2> e$. Thus $\gamma_1 - e$ is strictly greater than $\gamma_2$, and in the abacus of $\eta$ we can slide the first bead $\gamma_1$ up by one, replacing $\gamma_1$ with $\gamma_1 - e$. This yields a partition of $e$-weight $w$ mapping to $\sum (b_i - \delta_{i,1}e) \Lambda_i\in\mathcal{W}_{r,e,w}$.
\end{proof}

We next study the occurrence of $0$–dimensional simplices in $\mathcal{W}_{r,e,w}$.

\begin{Lemma}\label{lm:holes-of-pattern}
    Let $a_1, \dots, a_{r-1}$ be non-negative integers. The point $P = \sum_{i=1}^{r-1} a_i e \Lambda_i$ lies in $\mathcal{W}_{r,e,w}$ if and only if
    \begin{equation}\label{eq:vertex-condition}
        w \ge \sum_{i=1}^{r-1} i(a_i-1) \quad \text{and} \quad w \equiv \sum_{i=1}^{r-1} i(a_i-1) \pmod{r}.
    \end{equation}
    Moreover, such points correspond precisely to the component $\mathcal{W}_{r,e,w}\big((r)\big)$ in the simplicial decomposition; that is, they arise from abacus configurations where all $r$ beads lie on a single runner.
\end{Lemma}

\begin{proof}
    Suppose that $\lambda \in \mathscr{P}_{r,e,w}$ maps to $P$, and let $\beta_i$ be the beta numbers of $\lambda$. Then $\beta_i - \beta_{i+1} = a_i e$ for all $1 \le i \le r-1$.
    This implies $\beta_i \equiv \beta_{i+1} \pmod e$ for all $i$. Consequently, all $r$ beads lie on the same runner of the $e$-abacus. Hence the partition $\lambda$ belongs to the set $\mathscr{P}_{r,e,w}\big((r)\big)$, where $\mu=(r)$ is the composition of length $1$.

    Let the single non-empty runner be the $j$-th runner ($0 \le j \le e-1$). The positions of the beads are determined by the gaps $a_i$ between them. Specifically, if the last bead $\beta_r$ is at row $k$ (position $ke+j$), then the bead $\beta_i$ is at position:
    \[
        \beta_i \;=\; \beta_r + \sum_{m=i}^{r-1} (\beta_m - \beta_{m+1}) \;=\; (ke+j) + e \sum_{m=i}^{r-1} a_m.
    \]
    This shows the bead $\beta_i$ is on row $k+\sum_{m=i}^{r-1} a_m$. By sliding all beads up as high as possible, we see that the bead $\beta_i$ will be put on row $r-i$, hence the $e$-weight of $\lambda$ is:
    \[
        w_e(\lambda) \;=k+\sum_{i=1}^{r-1}\big(k+\sum_{m=i}^{r-1}a_m-(r-i)\big)=\; kr + \sum_{i=1}^{r-1} i(a_i-1) ,
    \]
   This is equivalent to the desired statement.
\end{proof}

We now formalize the patterns observed in \autoref{fig:r3e8-grid-0-9}.

\begin{Definition}\label{def:affine-vertex}
    A point $v \in \mathcal{C}^+\cap P$ is called an \emph{affine vertex} (of the $e$-alcove geometry) if it lies in the intersection of the affine hyperplanes corresponding to all positive roots. That is, for each $\alpha \in R^+$, there exists a positive integer $k_\alpha$ such that:
    \[
        \langle \alpha^\vee, v \rangle \;=\; k_\alpha e.
    \]
\end{Definition}

Fix a generic triple $\triple$ as above.  
An affine vertex $v=\sum_{i=1}^{r-1} a_i \Lambda_i \in \mathcal{W}_{r,e,w}$
is called a \emph{boundary affine vertex} of $\mathcal{W}_{r,e,w}$ if there is no distinct point
$Q = \sum_{i=1}^{r-1} b_i \Lambda_i \in \mathcal{W}_{r,e,w}$
such that $b_i \ge a_i$ for all $1 \le i \le r-1$.

An affine vertex
$u=\sum_{i=1}^{r-1} c_i \Lambda_i \in \mathcal{C}^+\cap P$
is called an \emph{interior affine vertex} of $\mathcal{W}_{r,e,w}$ if there exists a boundary affine vertex
$v=\sum_{i=1}^{r-1} a_i \Lambda_i$
of $\mathcal{W}_{r,e,w}$ such that $c_i \le a_i$ for all $i$ and $c_j<a_j$ for at least one index $j$. Note that we do not require an interior affine vertex of $\mathcal{W}_{r,e,w}$ to lie in $\mathcal{W}_{r,e,w}$ itself.

\begin{Example}\label{eg:interior-boundary-affine-vertex}
    Let $\triple=(3,8,3)$ and consider $\mathcal{W}_{3,8,3}$; see \autoref{fig:r3e8w3}. The boundary affine vertices are
    \[
      2e\Lambda_1+2e\Lambda_2
      \quad\text{and}\quad
      4e\Lambda_1+e\Lambda_2,
    \]
    and the interior affine vertices are
    \[
      e\Lambda_1+e\Lambda_2,\quad
      e\Lambda_1+2e\Lambda_2,\quad
      2e\Lambda_1+e\Lambda_2,\quad
      3e\Lambda_1+e\Lambda_2.
    \]
\end{Example}

\begin{Lemma}\label{lm:interior-vertex-existence}
    Fix a generic triple $\triple$. There exists an integer $w_{\text{in}}$ such that for all $w \ge w_{\text{in}}$, the set $\mathcal{W}_{r,e,w}$ contains an affine vertex $v = \sum_{i=1}^{r-1} a_i e \Lambda_i$.
    Specifically, one may take $w_{\text{in}} = (r-1)^2$.
\end{Lemma}

\begin{proof}
    By \autoref{lm:holes-of-pattern}, an affine vertex $P = \sum_{i=1}^{r-1} a_i e \Lambda_i$ exists in $\mathcal{W}_{r,e,w}$ if and only if
    \[
        w \ge \sum_{i=1}^{r-1} i(a_i-1) \quad \text{and} \quad w \equiv \sum_{i=1}^{r-1} i(a_i-1) \pmod r.
    \]
    We restrict our search to the vertices such that $a_1 = \dots = a_{r-2} = 1$. The condition reduces to find an integer $x_{r-1} :=a_{r-1}-1\ge 0$ such that:
    \[
        w \ge (r-1)x_{r-1} \quad \text{and} \quad  w\equiv (r-1)x_{r-1} \pmod r.
    \]
    Since $r-1 \equiv -1 \pmod r$, the congruence becomes $-x_{r-1} \equiv w \pmod r$. This linear congruence has a unique solution $x_{r-1}$ in the range $\{0, 1, \dots, r-1\}$.
    For this minimal solution, the required weight is $(r-1)x_{r-1} \le (r-1)^2$.
    Thus, for any $w \ge (r-1)^2$, there exists a valid choice of $x_{r-1}$ (and hence $a_{r-1} \ge 1$, with $a_i=1$ otherwise) satisfying the condition.
\end{proof}

For any real number $c$, the \emph{ceiling function} $\lceil c \rceil$ is the least integer greater than or equal to $c$. For instance, $\lceil 5.2\rceil =6$.
\begin{Lemma}\label{lem:weight-lower-bound-beta}
    Let $\lambda$ be a partition such that $\ell(\lambda)\leq r$ and let $\beta_1 > \beta_2 > \cdots > \beta_r$ be its $r$-beta numbers. Set $x_i:=\beta_i - \beta_{i+1}$ and $m_i:= \Big\lceil \frac{x_i}{e} \Big\rceil
        \quad (1 \le i \le r-1)$.
    Then the $e$--weight of $\lambda$ satisfies
    \begin{equation}\label{eq:weight-lower-bound}
        w_e(\lambda) \;\ge\; \sum_{i=1}^{r-1} i\,(m_i-1).
    \end{equation}
\end{Lemma}

\begin{proof}
    By definition, the $e$-weight $w_e(\lambda)$ is the total number of upward moves of beads required to reach the $e$-core. By definition of $m_i$, we can write $x_i = (m_i-1)e + s_i$, where $1 \le s_i \le e$.
    
    Consider a bead $\beta_i$. Since there are no other beads strictly between $\beta_i$ and $\beta_{i+1}$, and the gap $x_i > (m_i-1)e$, we can slide this bead up $m_i-1$ steps without crossing $\beta_{i+1}$. Then the new position of this bead is $\beta_i' = \beta_i - (m_i-1)e = \beta_{i+1} + s_i$. Since $s_i \ge 1$, we have $\beta_i' > \beta_{i+1}$, so the relative order of these beads is preserved.
    
    Performing these moves on $\beta_i$ creates empty space below $\beta_i'$. This allows all beads $\beta_k$ with $k < i$ (which are larger than $\beta_i$) to also slide up by at least $(m_i-1)$ while maintaining the strict decreasing order of the beta numbers.
    Thus, the gap associated with $m_i$ contributes at least $m_i-1$ upward moves to bead $\beta_i$, and consequently at least $m_i-1$ moves to every bead $\beta_k$ below it ($k < i$).
    
    Summing these contributions over $i=1, \dots, r-1$, the total number of upward moves is at least $\sum_{i=1}^{r-1} i(m_i-1)$. As the $e$-weight accounts for all necessary moves to reach its core, we obtain the desired inequality \autoref{eq:weight-lower-bound}.
\end{proof}

\begin{Lemma}\label{lem:affine-vertex-dominates}
    Fix a generic triple $\triple$. Suppose $Q=\sum_{i=1}^{r-1} x_i \Lambda_i \in\mathcal{W}_{r,e,w}$, then there exists an affine vertex
    $v=\sum_{i=1}^{r-1} a_i e\Lambda_i \in \mathcal{W}_{r,e,w}$
    such that $a_i e \ge x_i$ for all $1\le i\le r-1$. 
\end{Lemma}

\begin{proof}
    Set $m_i=\left\lceil x_i/e \right\rceil$, so $m_i e \ge x_i$ for each $i$. Let $\lambda$ be a partition such that $\Omega(\lambda)=Q$. By \autoref{lem:weight-lower-bound-beta} we have $w=w_e(\lambda)\ge\sum_{i=1}^{r-1} i\,(m_i-1)$. Define
    $v=\sum_{i=1}^{r-1} m_i e\,\Lambda_i + \bigl(w-\sum_{i=1}^{r-1} i\,(m_i-1)\bigr)e\,\Lambda_1$, that is, $a_i=m_i$ for $i\ge2$ and
    $a_1=m_1+\bigl(w-\sum_{i=1}^{r-1} i\,(m_i-1)\bigr)$. Then $\sum_{i=1}^{r-1} i\,(a_i-1)=w$ and by \autoref{lm:holes-of-pattern}, $v\in\mathcal{W}_{r,e,w}$.
\end{proof}

\begin{Corollary}\label{cor:equivalent-boundary-affine-vertex}
    An affine vertex $v=\sum_{i=1}^{r-1}a_{i}e\Lambda_i$ is a boundary affine vertex of $\mathcal{W}_{r,e,w}$ if and only if $v \in \mathcal{W}_{r,e,w}$ and there does not exist a distinct affine vertex $u=\sum_{i=1}^{r-1}b_{i}e\Lambda_i \in \mathcal{W}_{r,e,w}$ such that $b_i\geq a_i$ for all $i$.
\end{Corollary}

\begin{proof}
    This follows from the definition of a boundary affine vertex and \autoref{lem:affine-vertex-dominates}.
\end{proof}

\begin{Proposition}\label{prop:characterization-of-boundary-affine-vertex}
    Fix a generic triple $\triple$. An affine vertex $v = \sum_{i=1}^{r-1} a_i e \Lambda_i$ is a boundary affine vertex of $\mathcal{W}_{r,e,w}$ if and only if $w = \sum_{i=1}^{r-1} i(a_i-1)$.
\end{Proposition}

\begin{proof}
    By \autoref{lm:holes-of-pattern}, $v \in \mathcal{W}_{r,e,w}$ if and only if $w = \sum_{i=1}^{r-1} i(a_i-1) + kr$ for some non-negative integer $k$. It suffices to show that $v$ is a boundary affine vertex if and only if $k=0$.

    Suppose $k > 0$. We can construct another affine vertex $v' = v + (kr)e\Lambda_{1} = (kr+a_1)e\Lambda_1 + \sum_{i=2}^{r-1} a_i e \Lambda_i$. Note that
    \[
        1\cdot (kr+a_1-1) + \sum_{i=2}^{r-1} i(a_i-1) = kr + \sum_{i=1}^{r-1} i(a_i-1) = w.
    \]
    Thus, $v' \in \mathcal{W}_{r,e,w}$ by \autoref{lm:holes-of-pattern}. Since the coordinates of $v'$ are strictly greater than or equal to those of $v$ and the first coordinate is strictly larger, $v$ cannot be a boundary affine vertex.

    Conversely, suppose $k=0$. If there exists another affine vertex $v'' = \sum_{i=1}^{r-1} b_i e \Lambda_i$ in $\mathcal{W}_{r,e,w}$ such that $b_i \ge a_i$ for each $i$, then
    \[
        \sum_{i=1}^{r-1} i(b_i-1) \ge \sum_{i=1}^{r-1} i(a_i-1) = w.
    \]
    However, since $v'' \in \mathcal{W}_{r,e,w}$, \autoref{lm:holes-of-pattern} implies that $\sum_{i=1}^{r-1} i(b_i-1) = w - k'r \le w$ for some $k' \ge 0$. These inequalities force the sum to be exactly $w$, which implies $b_i = a_i$ for all $i$. Hence, by \autoref{cor:equivalent-boundary-affine-vertex}, $v$ is a boundary affine vertex.
\end{proof}

For two affine vertices $v_1=\sum_{i=1}^{r-1} a_i e \Lambda_i$ and $v_2=\sum_{i=1}^{r-1} b_i e \Lambda_i$, we say
\begin{enumerate}[label=$\bullet$]
    \item $v_1$ and $v_2$ are \emph{adjacent} if there exists an index $j$ with $1\le j\le r-1$ such that $a_i = b_i \pm \delta_{ij}$ for all $i$.
    \item $(v_1,v_2)$ is a \emph{bad pair} of $\mathcal{W}_{r,e,w}$ if $v_1$ is an interior affine vertex, $v_2$ is a boundary affine vertex, and $a_i+\delta_{i,r-1}=b_i$ for all $i$.

    \item $(v_1,v_2)$ is a \emph{good pair} of $\mathcal{W}_{r,e,w}$ if both $v_1$ and $v_2$ are interior affine vertices, $v_1\notin\mathcal{W}_{r,e,w}\ni v_2$, and $a_i+\delta_{i,r-1}=b_i$ for all $i$.
\end{enumerate}

 In particular, both a bad pair and a good pair consist of adjacent affine vertices.

\begin{Example}\label{eg:good-pair-affine-vertices}
    Let $(r,e,w)=(3,8,3)$ and consider $\mathcal{W}_{3,8,3}$; see \autoref{fig:r3e8w3}. In this case the only bad pair is
    \[
      (2e\Lambda_1+e\Lambda_2,\;2e\Lambda_1+2e\Lambda_2),
    \]
    and there is no good pair.

    Now let $(r,e,w)=(3,8,5)$ and consider $\mathcal{W}_{3,8,5}$; see \autoref{fig:r3e8w5}. The bad pairs are
    \[
      (2e\Lambda_1+2e\Lambda_2,\;2e\Lambda_1+3e\Lambda_2)
      \quad\text{and}\quad
      (4e\Lambda_1+e\Lambda_2,\;4e\Lambda_1+2e\Lambda_2),
    \]
    while the unique good pair is
    \[
      (e\Lambda_1+e\Lambda_2,\;e\Lambda_1+2e\Lambda_2).
    \]
\end{Example}

For simplicity, in the rest of this section we write $\mathcal{A}$ instead of $\mathcal{A}^{(e)}$ for an $e$--alcove.

For each affine vertex $v\in\mathcal{C}^+\cap P$, let $\mathcal{S}_v$ be the set of $e$--alcoves $\mathcal{A}$ such that $v$ lies in the closure of $\mathcal{A}$, that is,
\[
  \mathcal{S}_v := \{\mathcal{A}\in\mathcal{H}^{(e)} \mid v\in\overline{\mathcal{A}}\}.
\]

The \emph{cell} $\mathfrak{C}_v$ centred at $v$ is the union of the closures of all $e$--alcoves in $\mathcal{S}_v$, namely
\[
  \mathfrak{C}_v := \bigsqcup_{\mathcal{A}\in\mathcal{S}_v}\overline{\mathcal{A}}.
\]
Its interior $\mathfrak{C}_v^\circ$ is called the \emph{open cell} centred at $v$.

Then
\[
  \mathfrak{C}_v\cap P
  = \{w\in\mathcal{C}^+\cap P : |\langle w-v,\alpha^\vee\rangle|\le e
      \text{ for all }\alpha\in R^+\}
\]
and
\begin{equation}\label{eq:distance-equation}
  \mathfrak{C}_v^\circ\cap P
  = \{w\in\mathcal{C}^+\cap P : |\langle w-v,\alpha^\vee\rangle|< e
      \text{ for all }\alpha\in R^+\}.
\end{equation}

\begin{Definition}\label{def:regular-stingray}    
    \begin{enumerate}
        \item For each good pair $(v_1,v_2)$ of $\mathcal{W}_{r,e,w}$, the \textbf{regular pattern} $\operatorname{Reg}_{v_1,v_2}$ associated to $(v_1,v_2)$ is the set $\mathfrak{C}_{v_1}^{\circ}\cap \mathcal{W}_{r,e,w}$.
        
        \item For each bad pair $(v_1, v_2)$ of $\mathcal{W}_{r,e,w}$, the \textbf{stingray pattern} $\operatorname{Stin}_{v_1,v_2}$ associated to $(v_1,v_2)$ is the set $\big(\mathfrak{C}_{v_1}^{\circ}\cap \mathcal{W}_{r,e,w}\big)\sqcup \{v_2\}$.
    \end{enumerate}
\end{Definition}

For any bad pair $(v_1,v_2)$, define $\operatorname{L}_{v_1,v_2}$ to be the (discrete) segment connecting $v_1$ and $v_2$, that is,
\[
  \operatorname{L}_{v_1,v_2}=\{v_1+k\Lambda_{r-1}\mid 1\le k\le e\}=\{v_2-k\Lambda_{r-1}\mid 0\le k\le e-1\}.
\]
We remind the readers that in this definition, $v_1$ is not included in this segment while $v_2$ is. The advantage of this definition is the following:
\begin{Lemma}\label{lm:tail-of-Wrew}
    For any bad pair $(v_1,v_2)$ of $\mathcal{W}_{r,e,w}$, the segment $\operatorname{L}_{v_1,v_2}$ is contained in $\mathcal{W}_{r,e,w}$.
\end{Lemma}
\begin{proof}
    An affine vertex in $\mathcal{W}_{r,e,w}$ corresponds to a copy of $\mathcal{W}_{r,e}\big((r)\big)$. This in turn corresponds to abaci with a single non-empty runner carrying all $r$ beads. We may assume this non-empty runner is the $(e-1)$–runner. Recall that $v_2=\sum_{i=1}^{r-1}(\beta_i-\beta_{i+1})\Lambda_i$, where the $\beta_i$ are the beta numbers corresponding to this abacus. We want to construct abaci with the same $e$–weight as this one whose images under $\Omega$ trace out the segment $\operatorname{L}_{v_1,v_2}$, and hence give points in $\mathcal{W}_{r,e,w}$.
    
    Shift the last bead $\beta_r$ to the left by $k$ positions, where $1\le k\le e-1$, and then slide each of the other beads $\beta_i$ up by one step. This operation does not change the $e$–weight, and it preserves the differences $\beta_i-\beta_{i+1}$ for all $i\ne r-1$. Since $\beta_r$ decreases by $k$ and $\beta_{r-1}$ decreases by $e$, the difference $\beta_{r-1}-\beta_r$ decreases by $e-k$. Hence the resulting dominant weight is $v_2-(e-k)\Lambda_{r-1}$, which ranges over $\operatorname{L}_{v_1,v_2}\setminus\{v_2\}$ as $k$ varies. Adding the endpoint $v_2$ recovers the full segment $\operatorname{L}_{v_1,v_2}$, as required. 
\end{proof}
As this segment is contained in the cell $\mathfrak{C}_{v_1}$, we deduce the following:
\begin{Corollary}\label{cor:tail-of-stingray}
    For any bad pair $(v_1,v_2)$ of $\mathcal{W}_{r,e,w}$, the segment $\operatorname{L}_{v_1,v_2}$ is contained in the associated stingray pattern $\operatorname{Stin}_{v_1,v_2}$.
\end{Corollary}

Because of \autoref{cor:tail-of-stingray}, we call $\operatorname{L}_{v_1,v_2}$ the \emph{tail} of the stingray (or the \emph{stingray tail}), and we call $\operatorname{Stin}_{v_1,v_2}\setminus\operatorname{L}_{v_1,v_2}$ the \emph{body} of the stingray (or the \emph{stingray body}).


\begin{Example}\label{eg:stingray-pattern-two}
We continue with \autoref{eg:good-pair-affine-vertices} and consider the case $\mathcal{W}_{3,8,5}$; see \autoref{fig:regular-stingray-pattern}, where the stingray patterns are drawn in {\color{red}red} and the regular pattern in {\color{blue}blue}. We also distinguish the stingray tail and the stingray body by using different shades of red.
    \begin{figure}[htbp]
        \centering
        \includegraphics[height=5cm,keepaspectratio]{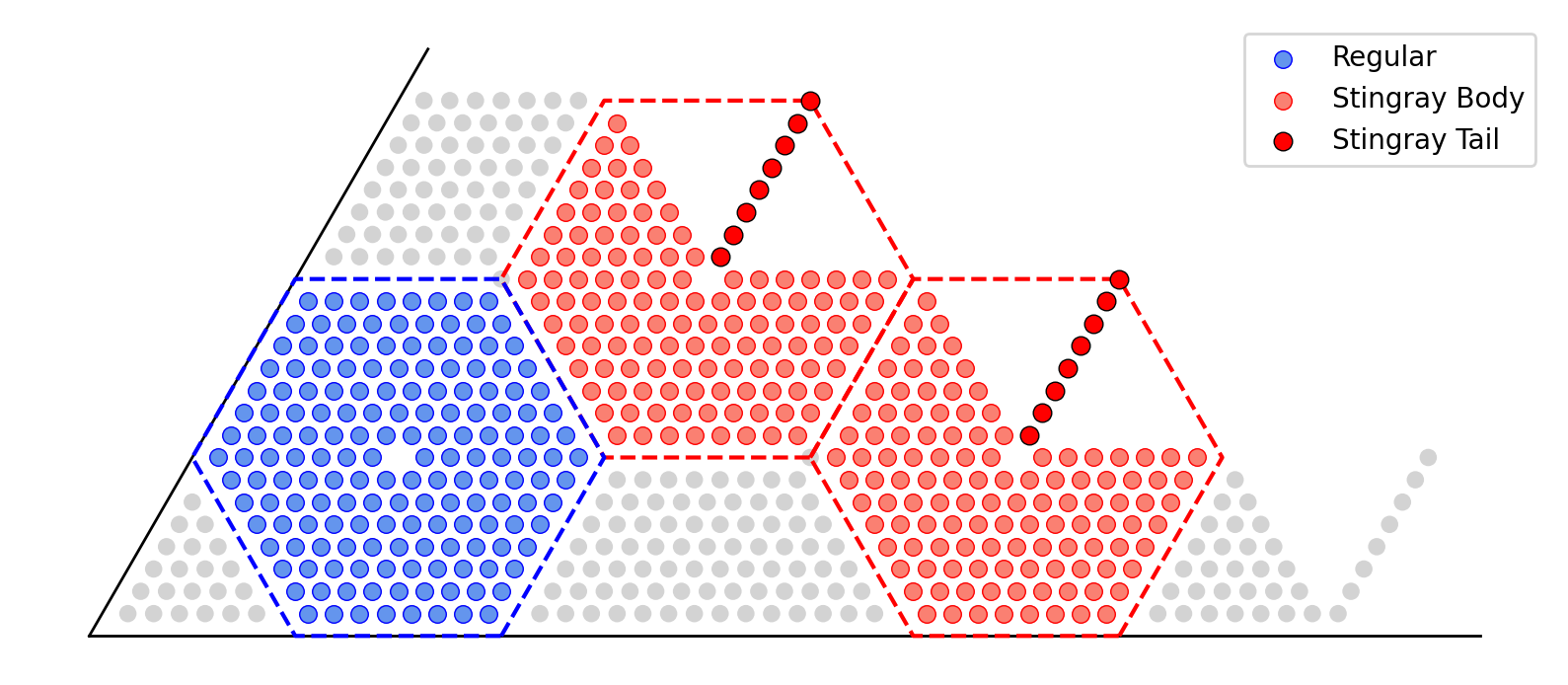}
        \caption{Two stingray patterns and one regular pattern for $(r,e,w)=(3,8,5)$.}
        \label{fig:regular-stingray-pattern}
    \end{figure}
\end{Example}

Our next goal is to show that a stingray tail is isolated in an appropriate sense, which explains why we regard it as a “tail”: it should be separated from the body, as in \autoref{fig:regular-stingray-pattern}.

\begin{Lemma}\label{lm:e-weight-of-1^r}
    Let $\lambda$ be a partition and let $\beta_i$ be its $r$-beta numbers. Form the $e$-abacus of $\lambda$ with $r$ beads, and suppose that each runner has at most one bead on it. Then the $e$–weight of $\lambda$ is given by
    \[w_e(\lambda)=\sum_{i=1}^{r}\lfloor \beta_i/e\rfloor\]
\end{Lemma}
\begin{proof}
    A bead $\beta_i$ lies in row $q_i=\lfloor\beta_i/e\rfloor$. Sliding it up to the row $0$ makes that runner flush and contributes exactly $q_i$ upward moves. Summing over all beads shows that the total number of upward moves, hence the $e$–weight, is $\sum_{i=1}^r \lfloor \beta_i/e\rfloor$, as claimed.
\end{proof}

\begin{Lemma}\label{lm:equivalence-of-tail-alcove}
    Let $(v_1,v_2)$ be a bad pair of affine vertices of $\mathcal{W}_{r,e,w}$,
    and let $\operatorname{L}_{v_1,v_2}$ be the stingray tail. Then for an $e$–alcove $\mathcal{A}$,
    \[
      \operatorname{L}_{v_1,v_2}\subset \overline{\mathcal{A}}
      \quad\Longleftrightarrow\quad
      \mathcal{A}\in\mathcal{S}_{v_1}\cap\mathcal{S}_{v_2}.
    \]
\end{Lemma}

\begin{proof}
    If $\mathcal{A}\in\mathcal{S}_{v_1}\cap\mathcal{S}_{v_2}$ then
    $v_1,v_2\in\overline{\mathcal{A}}$ by definition, and since
    $\overline{\mathcal{A}}$ is convex it contains the whole segment connecting $v_1$ and $v_2$, hence in particular the discrete subset
    $\operatorname{L}_{v_1,v_2}$.

    Conversely, assume $\operatorname{L}_{v_1,v_2}\subset\overline{\mathcal{A}}$.
    The closure $\overline{\mathcal{A}}$ is convex and its
    intersection with the line through $v_1$ and $v_2$ is a closed interval.
    The tail $\operatorname{L}_{v_1,v_2}$ consists of all intermediate lattice
    points between $v_1$ and $v_2$, so the smallest
    closed interval in that line containing $\operatorname{L}_{v_1,v_2}$ has
    endpoints $v_1$ and $v_2$. Hence $v_1,v_2\in\overline{\mathcal{A}}$, i.e.\
    $\mathcal{A}\in\mathcal{S}_{v_1}\cap\mathcal{S}_{v_2}$.
\end{proof}

\begin{Proposition}\label{prop:tail-separate}
    Fix a generic triple $\triple$. Let $\operatorname{Stin}_{v_1,v_2}$ be a stingray pattern contained in
    $\mathcal{W}_{r,e,w}$ and let $\operatorname{L}_{v_1,v_2}$ be the tail of this
    stingray. Then for any $e$--alcove $\mathcal{A}\in \mathcal{S}_{v_1}$ such
    that $\operatorname{L}_{v_1,v_2}\subset \overline{\mathcal{A}}$, we have
    $\mathcal{A}\cap \mathcal{W}_{r,e,w}=\emptyset$.
\end{Proposition}

\begin{proof}
    By \autoref{lm:equivalence-of-tail-alcove}, it suffices to prove the case when $\mathcal{A}\in\mathcal{S}_{v_2}$. Write $v_2 = \sum_{i=1}^{r-1} b_i e \,\Lambda_i$.
    Since $(v_1, v_2)$ is a bad pair, $v_2$ is a boundary affine vertex, so by
    \autoref{prop:characterization-of-boundary-affine-vertex} the $e$--weight (of any partition whose image under $\Omega$ is $v_2$) satisfies
    \begin{equation}\label{eq:target-weight}
        w = \sum_{i=1}^{r-1} i\,(b_i-1).
    \end{equation}
    Suppose, for a contradiction, that there exists a point $Q = \sum_{i=1}^{r-1} x_i \Lambda_i \in \mathcal{A} \cap \mathcal{W}_{r,e,w}$.

    Assume first that $x_j > b_j e$ for some $j$. By \autoref{lem:affine-vertex-dominates},
    there is an affine vertex $v'=\sum_{i=1}^{r-1} c_i e\Lambda_i \in \mathcal{W}_{r,e,w}$
    with $c_i e\ge x_i$ for all $i$, and in particular $c_j e\ge x_j>b_j e$, so $c_j>b_j$.
    The distance condition \eqref{eq:distance-equation} for each simple root
    $\alpha_i$ gives $x_i> (b_i-1)e$, which forces $c_i\ge b_i$ for all $i$.
    Thus $v_2$ cannot be a boundary affine vertex by \autoref{cor:equivalent-boundary-affine-vertex}, contradicting our assumption.
    Hence we must have
    \begin{equation}\label{eq:xi-le-bie}
        x_i \le b_i e \qquad\text{for all } 1\le i\le r-1.
    \end{equation}
    As $Q\in \mathcal{A}$, there exists some $i$ such that the inequality in
    \eqref{eq:xi-le-bie} is strict.

    Let $\lambda$ be a partition with $\Omega(\lambda)=Q$
    and let $(\beta_1,\dots,\beta_r)$ be its $r$--beta numbers. Then
    \begin{equation}\label{eq:beta-diffs-x}
        \beta_i - \beta_r
        = \sum_{k=i}^{r-1} (\beta_k - \beta_{k+1})
        = \sum_{k=i}^{r-1} x_k
        \qquad (1\le i\le r-1).
    \end{equation}
    Since $Q$ lies in the interior of an $e$--alcove, by
    \autoref{lm:interior-of-alcove-index} we know that $Q\in\mathcal{W}_{r,e,w}(1^r)$.
    Therefore, by \autoref{lm:e-weight-of-1^r},
    \begin{equation}\label{eq:we-lambda-start}
        w_e(\lambda)
        = \sum_{i=1}^{r} \Big\lfloor \frac{\beta_i}{e} \Big\rfloor
        \ge \sum_{i=1}^{r-1} \Big\lfloor \frac{\beta_r + \sum_{k=i}^{r-1} x_k}{e} \Big\rfloor.
    \end{equation}

    Because $Q$ lies in an alcove whose closure contains the boundary vertex $v_2$, the
    distance condition \eqref{eq:distance-equation} applied to the root
    $\alpha_i+\cdots+\alpha_{r-1}$
    shows that, for each $i$,
    \begin{equation}\label{eq:sum-xk-between}
        \sum_{k=i}^{r-1} b_k e - e
        < \sum_{k=i}^{r-1} x_k
        < \sum_{k=i}^{r-1} b_k e,
    \end{equation}
    where the right-hand inequality follows from \eqref{eq:xi-le-bie}.  Dividing
    \eqref{eq:sum-xk-between} by $e$ gives
    \begin{equation*}
        \sum_{k=i}^{r-1} b_k - 1
        < \frac{1}{e}\sum_{k=i}^{r-1} x_k
        < \sum_{k=i}^{r-1} b_k,
    \end{equation*}
    so
    \begin{equation}\label{eq:floor-sum-xk}
        \Big\lfloor \frac{1}{e}\sum_{k=i}^{r-1} x_k \Big\rfloor
        = \sum_{k=i}^{r-1} b_k - 1.
    \end{equation}
    Thus
    \begin{equation}\label{eq:floor-beta}
        \Big\lfloor \frac{\beta_r + \sum_{k=i}^{r-1} x_k}{e} \Big\rfloor
        = \Big\lfloor\frac{\beta_r}{e} + \frac{1}{e}\sum_{k=i}^{r-1} x_k\Big\rfloor
        \ge \sum_{k=i}^{r-1} b_k - 1.
    \end{equation}
    Substituting \eqref{eq:floor-beta} into \eqref{eq:we-lambda-start} yields
    \begin{align}
        w_e(\lambda)
        &\ge \sum_{i=1}^{r-1} \Big( \sum_{k=i}^{r-1} b_k - 1 \Big)
        \nonumber\\[0.2em]
        &= \sum_{k=1}^{r-1} k\,b_k - (r-1). \label{eq:we-lower-bound}
    \end{align}

    On the other hand, from \eqref{eq:target-weight},
    \begin{align}
        w
        &= \sum_{k=1}^{r-1} k\,(b_k-1)
        = \sum_{k=1}^{r-1} k\,b_k - \sum_{k=1}^{r-1} k \nonumber\\[0.2em]
        &= \sum_{k=1}^{r-1} k\,b_k - \frac{r(r-1)}{2}. \label{eq:w-explicit}
    \end{align}
    For $r\ge3$ we have $\frac{r(r-1)}{2} > r-1$, so combining
    \eqref{eq:we-lower-bound} and \eqref{eq:w-explicit} gives
    \begin{equation*}
        w_e(\lambda)
        \ge \sum_{k=1}^{r-1} k\,b_k - (r-1)
        > \sum_{k=1}^{r-1} k\,b_k - \frac{r(r-1)}{2}
        = w.
    \end{equation*}
    Thus $w_e(\lambda) > w$, contradicting the assumption that
    $Q\in\mathcal{W}_{r,e,w}$. Therefore no such $Q$ can exist, and we conclude that
    $\mathcal{A}\cap\mathcal{W}_{r,e,w}=\emptyset$.
\end{proof}

In contrast, a regular pattern does not exhibit this separation behavior. The following result explains the terminology. 
\begin{Proposition}\label{prop:regular-pattern}
    Fix a generic triple $\triple[3][e][w]$.  
    Let $\operatorname{Reg}_{v_1, v_2}$ be a regular pattern contained in
    $\mathcal{W}_{r,e,w}$, associated to a good pair $(v_1,v_2)$. Then for
    each $e$--alcove $\mathcal{A}\in\mathcal{S}_{v_1}$ all lattice points in
    the interior of $\mathcal{A}$ lie in $\mathcal{W}_{r,e,w}$, that is
    \[
        \mathcal{A}^\circ \cap P^+ \subset \mathcal{W}_{r,e,w}.
    \]
\end{Proposition}

\begin{proof}
    When $r=3$, the set $\mathcal{S}_{v_1}$ consists of six $e$–alcoves. To prove the statement, we construct an explicit point in each of these alcoves; the result then follows from \autoref{lm:interior-of-alcove-index}. Write $v_1=a_1e\Lambda_1+a_2e\Lambda_2$. One checks that
    \[
    P_1=(a_1e+1)\Lambda_1+(a_2e+1)\Lambda_2,\quad
    P_2=(a_1e+2)\Lambda_1+(a_2e-1)\Lambda_2,\quad
    P_3=(a_1e+1)\Lambda_1+(a_2e-2)\Lambda_2,
    \]
    \[
    P_4=(a_1e-1)\Lambda_1+(a_2e-1)\Lambda_2,\quad
    P_5=(a_1e-2)\Lambda_1+(a_2e+1)\Lambda_2,\quad
    P_6=(a_1e-1)\Lambda_1+(a_2e+2)\Lambda_2
    \]
    lie in the six alcoves adjacent to $v_1$, hence in the six elements of $\mathcal{S}_{v_1}$. It remains to show that these six points lie in $\mathcal{W}_{3,e,w}$.
    
    By assumption $(v_1,v_2)$ is a good pair. Write $v_2=b_1e\Lambda_1+b_2e\Lambda_2$, so $b_1=a_1$ and $b_2=a_2+1$. Consider the abacus configuration of a partition $\lambda$ with $\Omega(\lambda)=v_2$, and let $\beta_1,\beta_2,\beta_3$ be its $3$–beta numbers. Then the three beads lie on a single runner. If $\beta_3$ is on row $k$, then $\beta_2$ is on row $k+b_2$ and $\beta_1$ is on row $k+b_2+b_1$.
    
    By \autoref{lm:holes-of-pattern} and \autoref{prop:characterization-of-boundary-affine-vertex}, we have $(b_1-1)+2(b_2-1)+3k=w$ with $k>0$. We explain the construction for $P_5$; the other points are handled similarly.
    
    For the abacus of $v_2$, we may choose the unique non-empty runner to be the $(e-1)$–runner (recall $e>r=3$). Now shift the bead $\beta_1$ two steps to the left and shift the bead $\beta_3$ one step to the left. The resulting abacus represents the point $P_5$. Its $e$–weight is
    \[
        w' = k+(k+b_2+b_1)+(k+b_2)
    \]
    by \autoref{lm:e-weight-of-1^r}, so $w'=w+3$. Finally, slide all three beads up by one step which is possible because $k>0$. This does not change the image under $\Omega$, and it reduces the $e$–weight by $3$, so we obtain an abacus of $e$–weight $w$ still mapping to $P_5$. Hence $P_5\in\mathcal{W}_{3,e,w}$. The same argument applies to $P_1,P_2,P_3,P_4,P_6$.
\end{proof}

\begin{Remark}
    One might expect \autoref{prop:regular-pattern} to hold for all $r\ge 3$, but this is not the case. A counterexample occurs for $(r,e,w)=(5,6,15)$, with the good pair $v_1=(42,6,6,6)$ and $v_2=(42,6,6,12)$. It would be interesting to find a finer notion of a “perfect pair” for which \autoref{prop:regular-pattern} remains valid and this should be related to the Bruhat order.
\end{Remark}

\begin{Remark}\label{rmk:github-repository}
    We refer the reader to \href{https://github.com/taoqin-math/Stingray-Pattern-of-Dominant-Weights}{github repository} for more figures in $\mathfrak{sl}_3$, interactive figures in $\mathfrak{sl}_4$, and the code used to generate them.
\end{Remark}

\subsection{\texorpdfstring{Index the \(e\)-alcoves in \(\mathfrak{sl}_r\)}{Index the e-alcoves in sl(r)}}\label{subsec:index-alcoves}
From \autoref{fig:labels-of-simplicial-composition-1} and \autoref{fig:labels-of-simplicial-composition-2}, one observes that the green simplices (that is, the copies of $\mathcal{W}_{3,e}(1,1,1)$) are precisely the interiors of the $e$-alcoves in the fundamental chamber of $\mathfrak{sl}_3$. These simplices are in bijection with the $3$-partitions of $w$ of type $(1,1,1)$, so we may use such $3$-partitions to label the corresponding alcoves. More generally, the next result shows that one can index a subset of the dominant $e$-alcoves of $\mathfrak{sl}_r$ by the set of $r$-partitions of $w$ of type $(1^r)$.

For any $r$ and $w$, the $r$-partitions of $w$ of type $(1^r)$ are in bijection with the weak compositions of $w$ of length $r$: given such an $r$-partition, take the lengths of its components as the parts of the corresponding composition. For this reason, from now on we do not distinguish between these two sets, and we often write weak compositions to simplify the notation. We fix a generic triple $\triple$ throughout this section.

\begin{Lemma}\label{lem:fundamental-e-alcove-simplex}
    The set of dominant weights in the interior of the fundamental
    $e$--alcove, $\mathcal{A}_{0}^{(e)\,\circ}\cap P^+$, is the set of lattice points
    of an $(r-1)$-dimensional simplex of dilation factor $e-r$ in the sense of
    \autoref{def:simplex-lattice-points}. 
\end{Lemma}

\begin{proof}
  By \autoref{subsubsec:level-e-action}, a dominant weight
  $v=\sum_{i=1}^{r-1} a_i\Lambda_i$ lies in $\mathcal{A}_{0}^{(e)\,\circ}$ if
  and only if
  \[
    a_i\in\mathbb{Z}_{>0}
    \quad\text{for all }1\le i\le r-1,
    \qquad
    \sum_{i=1}^{r-1} a_i \le e-1.
  \]

  Let $b_i := a_i - 1 (1\le i\le r-1)$, define integers
  \[
    x_0 := e - r - \sum_{i=1}^{r-1} b_i,\quad
    x_i := b_i \ (1\le i\le r-1).
  \]
  It is straightforward to check that this construction gives a bijection
  between $\mathcal{A}_{0}^{(e)\,\circ}\cap P^+$ and the lattice points of
  the standard simplex $\Delta^{\,r-1}(e-r)\cap\mathbb{Z}^r$, with inverse
  map $(x_0,\dots,x_{r-1})\mapsto\sum_{i=1}^{r-1}(x_i+1)\Lambda_i$.

  Hence $\mathcal{A}_{0}^{(e)\,\circ}\cap P^+$ is the lattice point set of
  an $(r-1)$--simplex of dilation factor $e-r$, as claimed.
\end{proof}

\begin{Lemma}\label{lm:interior-of-alcove-index}
    In $\mathcal{W}_{r,e,w}(1^r)$, each copy of $\mathcal{W}_{r,e}(1^r)$
    (corresponding to an $r$–partition of $w$) is exactly the set of dominant weights in the interior of some dominant $e$–alcove.
\end{Lemma}

\begin{proof}
    By \autoref{thm:simplex-of-core-weights} and
    \autoref{thm:simplicial-decomposition}, each copy of
    $\mathcal{W}_{r,e}(1^r)$ is the set of lattice points of a simplex in $\mathcal{C}^+\cap P$.

    First, no point of such a simplex lies on the boundary of an $e$–alcove.
    Recall that a point $v=\sum_{i=1}^{r-1} a_i\Lambda_i$ lies on some affine
    hyperplane if and only if there exist indices $i\le j$ such that
    \[
        a_i + a_{i+1} + \cdots + a_j \;\equiv\; 0 \pmod e,
    \]
    equivalently, $\langle v,\alpha_i+\cdots+\alpha_j\rangle\equiv0\pmod e$.
    Let $v$ belong to a copy of $\mathcal{W}_{r,e}(1^r)$ and choose
    $\lambda$ such that $\Omega(\lambda)=v$. Consider an $e$–abacus for $\lambda$
    with $r$–beta numbers $\beta_1>\cdots>\beta_r$. If
    $a_i+\cdots+a_j\equiv0\pmod e$ for some $i\le j$, then
    \[
      (\beta_i-\beta_{i+1})+\cdots+(\beta_j-\beta_{j+1})
      \;=\; \beta_i-\beta_{j+1}
      \;\equiv\; 0 \pmod e,
    \]
    so the beads $\beta_i$ and $\beta_{j+1}$ lie on the same runner. Hence
    some runner contains at least two beads. By
    \autoref{thm:simplicial-decomposition}, this excludes $v$ from the
    component $\mathcal{W}_{r,e,w}(1^r)$, which consists precisely of points
    whose abaci have at most one bead on each runner. Thus no point of a copy
    of $\mathcal{W}_{r,e}(1^r)$ lies on any affine hyperplane, so every such point
    lies in the interior of some $e$–alcove.

    Let $S$ be a fixed copy of $\mathcal{W}_{r,e}(1^r)$.  
    If $S$ intersects any $e$–alcove, then by the above argument and the
    convexity of a simplex, $S$ is contained in the interior of that
    $e$–alcove.

    By \autoref{lem:fundamental-e-alcove-simplex} and the translation
    invariance of the $e$–alcoves, the number of dominant lattice
    points in the interior of any $e$–alcove is equal to the number of
    points of $\mathcal{W}_{r,e}(1^r)$.
    Since $S\subset\mathcal{A}_0^\circ\cap P^+$ and the cardinalities agree,
    we must have $S \;=\; \mathcal{A}_0^\circ\cap P^+$.
\end{proof}

Let $\mathfrak{A}_{r,e,w}$ be the set of dominant $e$-alcoves with non-trivial intersection with $\mathcal{W}_{r,e,w}$, and let $\Gamma_{r}(w)$ be the set of weak compositions of $w$ of length $r$.
Then, by \autoref{lm:interior-of-alcove-index}, there is a bijection between $\mathfrak{A}_{r,e,w}$ and $\Gamma_r(w)$. To be explicit, for any weak composition $\mu\in\Gamma_{r}(w)$, there is exactly one copy of $\mathcal{W}_{r,e}(1^r)$ corresponding to $\mu$. By \autoref{lm:interior-of-alcove-index}, this copy is exactly the set of dominant weights inside a dominant $e$-alcove. We denote this dominant $e$-alcove by $\mathcal{A}_{\mu}$. Then the bijection is given by $\mathcal{A}_{\mu}\leftrightarrow \mu$.

\begin{Example}\label{eg:label-alcoves-without-simplex}
    \autoref{fig:labels-of-simplicial-composition-1} and \autoref{fig:labels-of-simplicial-composition-2} illustrate how to assign labels to the simplices in $W_{3,e,w}(1,1,1)$. In \autoref{fig:labels-of-111} we make these labels more clear from \autoref{fig:labels-of-simplicial-composition-2} by removing the green points. Moreover, we have simplified these labels to use the weak compositions of $w=10$ with length $3$, corresponding to the $3$-partitions of $w$ of type $(1,1,1)$.
    \begin{figure}[htbp]
        \centering
        \captionsetup{type=figure}
        \includegraphics[height=6cm,keepaspectratio]{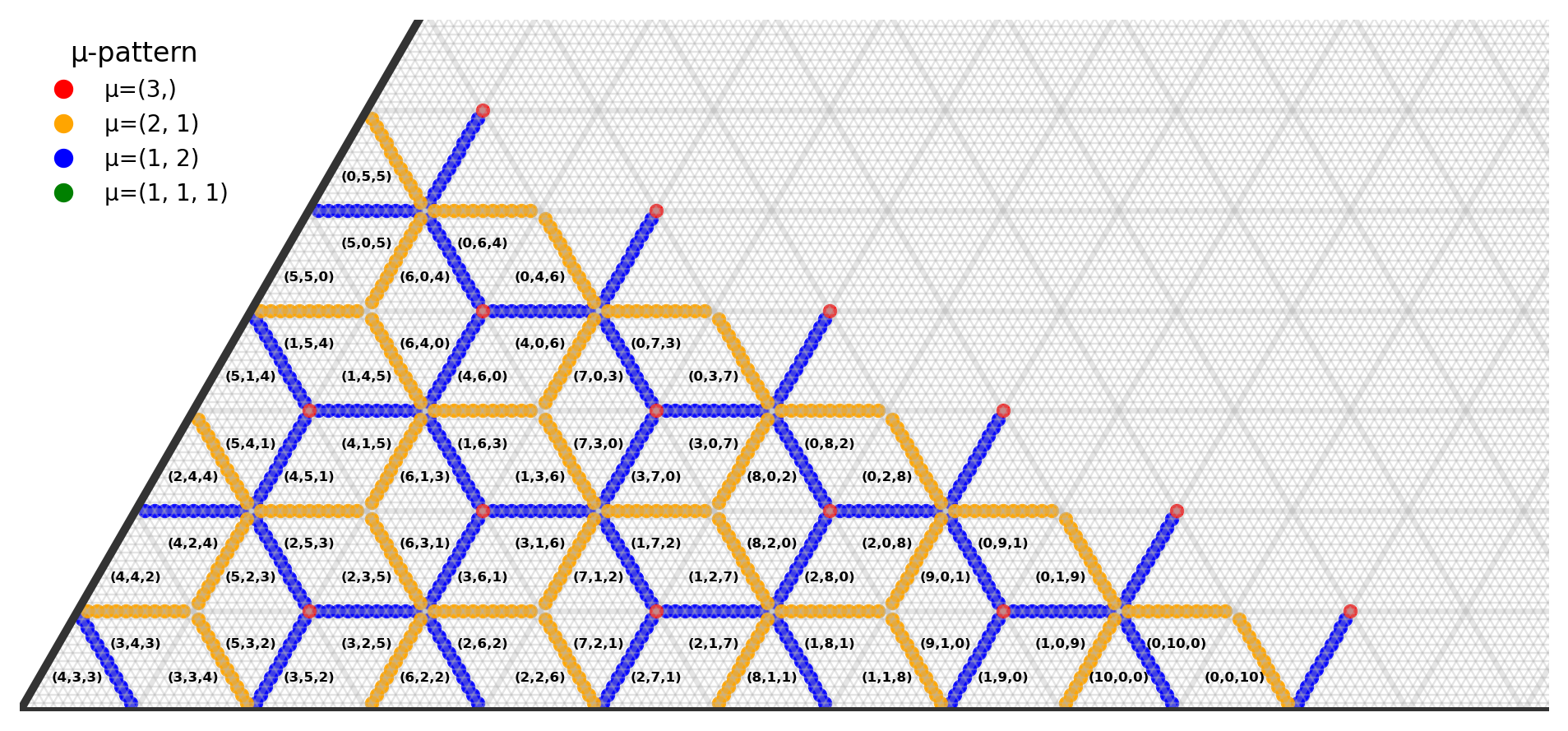}
        \caption{$r=3$, $e=12$, $w=10$}
        \label{fig:labels-of-111}
    \end{figure}
\end{Example}

We now examine these labels in \autoref{eg:label-alcoves-without-simplex} more closely.
    
First, in \autoref{fig:labels-of-111}, the fundamental $e$-alcove carries the label $(4,3,3)$. In general, \autoref{lm:label-of-fundamental-alcove} shows that the label of the fundamental $e$-alcove is given by the ``most balanced'' partition of $w$ of length $r$.
    
Second, how the labels change when cross the walls between adjacent alcoves is related to the action of affine Weyl group on abaci. In \autoref{fig:labels-of-111}, there are three types of walls: {\color{blue}blue}, {\color{Goldenrod}yellow}, and {\color{gray}gray}. The gray walls correspond to empty walls, that is, walls not contained in $W_{3,e,w}$ (see \autoref{fig:labels-of-simplicial-composition-2}). Suppose an alcove has label $(a,b,c)$ with $a+b+c=w=10$. Then one observes:
\begin{itemize}
    \item crossing a blue wall sends $(a,b,c)$ to $(b,a,c)$ (swapping the first two coordinates);
    \item crossing an empty (gray) wall sends $(a,b,c)$ to $(a,c,b)$ (swapping the last two coordinates);
    \item crossing a yellow wall sends $(a,b,c)$ to $(c+1,b,a-1)$.
\end{itemize}

For readers familiar with the action of the affine Weyl group on abaci, these transformations agree with the case of three runners with one bead on each runner at rows $a,b,c$, respectively. The blue, gray, and yellow walls correspond to the Coxeter generators $s_1$, $s_2$, and $s_0$. This also explains why $(4,3,3)$ labels the fundamental alcove: it is characterized by the property that the action of $s_2$ and $s_0$ fixes it, so its label changes under only one wall-crossing.

\medskip
For the generic triple $\triple$, write $w=rm+n$ with $m,n\in\N$ and $0\leq n\leq r-1$. Set
\begin{equation}\label{eq:fundamental-alcove-weak-composition}
    \lambda_{\emptyset}
    = \bigl(
       \overbrace{m+1,\dots,m+1}^{n\ \text{terms}},
       \overbrace{m,\dots,m}^{(r-n)\ \text{terms}}
    \bigr).  
\end{equation}

\begin{Lemma}\label{lm:label-of-fundamental-alcove}
    Under the bijection between $\mathfrak{A}_{r,e,w}$ and $\Gamma_{r}(w)$, the level-$e$ fundamental alcove $\mathcal{A}_0^{(e)}$ corresponds to $\mathcal{A}_{\lambda_{\emptyset}}$, where $\lambda_{\emptyset}$ is defined in \autoref{eq:fundamental-alcove-weak-composition}.
\end{Lemma}

\begin{proof}
    By \autoref{lm:interior-of-alcove-index}, it suffices to show that there exists a point in the copy indexed by $\lambda_{\emptyset}$ that lies in the fundamental alcove.
    
    Consider the $e$–abacus with $e$ runners, and choose the first $r$ runners. For each runner $0\le i\le n-1$ place a bead at row $m+1$, and for each runner $n\le i\le r-1$ place a bead at row $m$. 
    \begin{center}
        \Abacus[
              dotted cols={2,3,7,8,11,12},
              dotted rows={1},
              runner labels={0,1,2,3,n{-}1,n,n{+}1,7,8,r{-}1,r,11,12,e{-}1},
              bead size=0.58,
              bead sep=0.61,      
              bead font=\scriptsize,
              runner sep=0.08em
            ]{14}{
              41_{m{+}1},
              39^2_{m{+}1},
              35_m,
              33^2_m
            }
    \end{center}
    By definition, this abacus corresponds to a point in the copy of $\mathcal{W}_{r,e,w}(1^r)$ indexed by $\lambda_{\emptyset}$. Let $(\beta_1,\ldots,\beta_r)$ be the associated beta-numbers. The corresponding dominant weight is $\sum_{i=1}^{r-1} a_i\Lambda_i$, where $a_i=\beta_i-\beta_{i+1}$.

    From the abacus configuration we have $\beta_i-\beta_{i+1}=1$ for $i\neq n$ and $\beta_n-\beta_{n+1}=e-r+1$. Hence $a_i=1$ for $i\neq n$ and $a_n=e-r+1<e$, so this point lies in the fundamental $e$–alcove.
\end{proof}


\subsection{Affine Weyl group action}\label{subsec:affine-weyl-group-action}
It is well known that dominant alcoves of $\mathfrak{sl}_r$ can be indexed by $r$–core partitions, and that the affine Weyl group action admits a clean description on the abacus with infinitely many beads and rows; see \cite[Section~4]{bclg-generalised-cores-diophantine} and the references therein for details and examples.

In our setting, however, since we consider only the set of partitions with at most $r$ parts, the abaci we use have finitely many beads and rows. The action of the affine Weyl group does not preserve the length of a partition, so this action is not well defined on our (truncated) abaci. Nevertheless, as in \autoref{subsec:index-alcoves}, since we index only a finite subfamily of alcoves determined by a fixed $e$--weight $w$, we can still analyze the resulting ``action'' within this bounded region.

\smallskip
Fix a generic triple $\triple$. Recall that there is a bijection between $\mathfrak{A}_{r,e,w}$ and $\Gamma_{r}(w)$, as discussed in \autoref{subsec:index-alcoves}. Our first goal is to introduce the ``action'' of the affine Weyl group on $\Gamma_r(w)$ and to induce an action on $\mathfrak{A}_{r,e,w}$ through this bijection.
 
Let $W$ be the affine Weyl group, with simple reflections $\sigma_0,\sigma_1,\cdots,\sigma_{r-1}$. The subgroup generated by $\{\sigma_1,\cdots,\sigma_{r-1}\}$ is the finite Weyl group $W_0$, isomorphic to the symmetric group $\Sym_r$. 

Take $\mu=(\mu_1,\cdots,\mu_r)\in\Gamma_r(w)$, we define the following operations:
\begin{itemize}
    \item For $1\le i\le r-1$, define $\mu\sigma_i:=(\mu_1,\cdots,\mu_{i+1},\mu_{i},\cdots,\mu_r)$, i.e. by swapping the $i$-th entry and the $(i+1)$-th entry.

    \item Define $\mu\sigma_0:=(\mu_r+1,\mu_2,\cdots,\mu_{r-1},\mu_1-1)$, i.e. by swapping the first entry and the last entry, then adding one to the first entry and subtracting one from the last entry.
\end{itemize}

It is immediate that the above operations induce a $W_0$-action on $\Gamma_{r}(w)$. On the other hand, $\mu\sigma_0$ is defined if and only if $\mu_1\ge 1$. Thus, we obtain only a partial action of $W$ on $\Gamma_{r}(w)$.

We define the partial action of $W$ on $\mathfrak{A}_{r,e,w}$ by setting $\mathcal{A}_{\mu}\sigma_i:=\mathcal{A}_{\mu\sigma_i}$ for any $0\le i\le r-1$. The following result is immediate: for $1\le i\le r-1$, $\mathcal{A}_{\mu}\sigma_i=\mathcal{A}_{\mu}$ if and only if $\mu_i=\mu_{i+1}$; and $\mathcal{A}_{\mu}\sigma_0=\mathcal{A}_{\mu}$ if and only if $\mu_1=\mu_r+1$.

\smallskip
Our second goal is to show that this partial action is indeed geometrically compatible, in the sense that applying a $\sigma_i$-action to an $e$-alcove corresponds to crossing one of its walls.

\begin{Theorem}\label{thm:cross-wall-action-affine-weyl-group}
    Fix a generic triple $\triple$ and let $W$ be the affine Weyl group. Take $\mu\in\Gamma_r(w)$ and $\sigma_i\,(0\le i\le r-1)$ to be a simple reflection of $W$. Suppose that $\mu\sigma_i$ is defined. Then either $\mu\sigma_i=\mu$ or $\mathcal{A}_{\mu}\sigma_i$ has a common facet with $\mathcal{A}_{\mu}$.
\end{Theorem}
\begin{proof}
    To show that the two $e$-alcoves share a common facet, it suffices to show that all the Shi coefficients of the two $e$-alcoves are equal except for one, for which the corresponding Shi coefficients differ by one. Assume $\mu=(\mu_1,\cdots,\mu_r)\in\Gamma_r(w)$. Choose an $e$-abacus with $r$ beads such that the non-empty runners are exactly the last $r$ runners: $(e-r),(e-r+1),\cdots,(e-1)$. Moreover, we put a bead $B_t$ at the row $\mu_t$ of the $(e{-}r{+}t{-}1)$-runner. Let $\lambda$ be the partition corresponding to this $e$-abacus and let $P:=\Omega(\lambda)$ be the corresponding dominant weight. Then by definition, $P\in\mathcal{A}_{\mu}\cap \mathcal{W}_{r,e,w}$. 
    
    Let the $r$-beta numbers of $\lambda$ be $\beta_1>\beta_2>\cdots>\beta_r$, then the Shi coefficients are of the form:
    \[
        k^{(e)}_{\alpha_{xy}}(\mathcal{A}_{\mu})=\lfloor \frac{\beta_x-\beta_y}{e}\rfloor,\qquad \alpha_{xy}:=\alpha_x+\alpha_{x+1}+\cdots+\alpha_{y-1}=\epsilon_x-\epsilon_y\, (1\le x<y\le r).
    \]
    
    If $1\le i \le r-1$, then $\mu\sigma_i$ is always defined and equals $\mu$ if and only if $\mu_i=\mu_{i+1}$. Suppose $\mu_i\neq \mu_{i+1}$. From the $e$-abacus of $\lambda$, let $\lambda'$ be the partition corresponding to the $e$-abacus which is obtained by swapping the $(e{-}r{+}i{-}1)$-runner and the $(e{-}r{+}i)$-runner, then $P':=\Omega(\lambda')\in \mathcal{A}_{\mu\sigma_i}\cap \mathcal{W}_{r,e,w}$. 

    We calculate the Shi coefficients. Let $\beta_{x'}$ and $\beta_{y'}$ be the corresponding beta numbers of the beads $B_i$ and $B_{i+1}$ respectively. We may assume $x'<y'$ since the opposite case is similar.
    An immediate observation is that, by swapping the two runners, the order of the beta numbers does not change. 
    
    To be specific, let $\beta'_1>\beta'_2>\cdots>\beta'_{r}$ be the $r$-beta numbers of $\lambda'$, then $\beta'_i=\beta_i$ for $i\not\in\{x',y'\}$, $\beta'_{x'}=\beta_{x'}+1$ and $\beta'_{y'}=\beta_{y'}-1$.
    
    For any positive root $\alpha=\alpha_{xy}$, if $x,y\not\in\{x',y'\}$, then clearly $k^{(e)}_{\alpha_{xy}}(\mathcal{A}_{\mu})=k^{(e)}_{\alpha_{xy}}(\mathcal{A}_{\mu\sigma_i})$. Assume $x=x'$ and $y\neq y'$, let $B_t$ be the bead corresponding to $\beta_y$, then we have:
    \[
        k^{(e)}_{\alpha_{xy}}(\mathcal{A}_{\mu})=\lfloor\frac{\mu_ie+e-r+i-1-(\mu_{t}e+e-r+t-1)}{e}\rfloor=(\mu_i-\mu_t)+\lfloor\frac{i-t}{e}\rfloor;
    \]
    \[
        k^{(e)}_{\alpha_{xy}}(\mathcal{A}_{\mu\sigma_i})=\lfloor\frac{\mu_ie+e-r+i-(\mu_{t}e+e-r+t-1)}{e}\rfloor=(\mu_i-\mu_t)+\lfloor\frac{i-t+1}{e}\rfloor.
    \]
    Since $t\not\in\{i, i+1\}$ and $1\le i,t\le r<e$, $\lfloor\frac{i-t}{e}\rfloor=\lfloor\frac{i-t+1}{e}\rfloor$. Hence we again have $k^{(e)}_{\alpha_{xy}}(\mathcal{A}_{\mu})=k^{(e)}_{\alpha_{xy}}(\mathcal{A}_{\mu\sigma_i})$. 

    The case where $x\neq x'$ and $y=y'$ is symmetric. Hence the only remaining case is when $x=x'$ and $y=y'$, where we have:
    \[
        k^{(e)}_{\alpha_{xy}}(\mathcal{A}_{\mu})=\lfloor\frac{\mu_ie+e-r+i-1-(\mu_{i+1}e+e-r+i)}{e}\rfloor=(\mu_i-\mu_{i+1})+\lfloor\frac{-1}{e}\rfloor=\mu_i-\mu_{i+1}-1;
    \]
    \[
        k^{(e)}_{\alpha_{xy}}(\mathcal{A}_{\mu\sigma_i})=\lfloor\frac{\mu_ie+e-r+i-(\mu_{i+1}e+e-r+i-1)}{e}\rfloor=(\mu_i-\mu_{i+1})+\lfloor\frac{1}{e}\rfloor=\mu_i-\mu_{i+1}.
    \]
    Hence the corresponding Shi coefficents differ by one, as desired. As a consequence, we have proved the theorem for the case $1\le i\le r-1$.

    \smallskip The $\sigma_0$ case is analogous. Since $\mu\sigma_0$ is defined if and only if $\mu_{1}\geq 1$ and $\mu\sigma_0=\mu$ if and only if $\mu_1=\mu_r+1$, we may assume that $\mu_1\geq 1$ and $\mu_1\neq \mu_r+1$. Then we can calculate and compare the Shi coefficients similarly to the procedure above.
    
    Let $\lambda'$ be the partition corresponding to the $e$-abacus obtained by swapping the $(e{-}r)$-runner and the $(e{-}1)$-runner, then sliding the bead on the new $(e{-}r)$-runner down by one and sliding the bead on the new $(e{-}1)$-runner up by one. This operation is possible by the assumption. By definition, $P':=\Omega(\lambda')\in \mathcal{A}_{\mu\sigma_0}\cap\mathcal{W}_{r,e,w}$.

    We use $B_1$ and $B_r$ to denote the beads on the $(e{-}r)$-runner and the $(e{-}1)$-runner, respectively, of the $e$-abacus of $\lambda$, and let $\beta_{x'}$ and $\beta_{y'}$ be the corresponding $r$-beta numbers. The operation from $\lambda$ to $\lambda'$ on the $e$-abacus does not change the relative order of the beads, nor does it change the order of the beta numbers. Hence, we use $B_1'$ and $B_r'$ to denote the corresponding beads of the $e$-abacus of $\lambda'$, and let $\beta'_{x'}$ and $\beta'_{y'}$ be the corresponding $r$-beta numbers of $B_1'$ and $B_r'$, respectively. 

    We might assume $x'<y'$ as the opposite case is similar. Consider any positive root $\alpha=\alpha_{xy}$. If $x,y\not\in\{x',y'\}$, then the corresponding Shi coefficients are equal and there is nothing to show. If $x=x'$ and $y\neq y'$, let $B_t$ be the bead corresponding to $\beta_y$, then we have:
    \[
        k^{(e)}_{\alpha_{xy}}(\mathcal{A}_{\mu})=\lfloor\frac{\mu_{1}e+e-r-(\mu_{t}e+e-r+t-1)}{e}\rfloor=(\mu_1-\mu_t)+\lfloor\frac{1-t}{e}\rfloor;
    \]
    \[
        k^{(e)}_{\alpha_{xy}}(\mathcal{A}_{\mu\sigma_0})=\lfloor\frac{(\mu_{1}-1)e+e-1-(\mu_{t}e+e-r+t-1)}{e}\rfloor=(\mu_1-\mu_t)+\lfloor\frac{r-t-e}{e}\rfloor.
    \]
    Since $2\le t\le r-1$ and $e>r$, it follows that $\lfloor\frac{1-t}{e}\rfloor=\lfloor\frac{r-t-e}{e}\rfloor=-1$ and hence $k^{(e)}_{\alpha_{xy}}(\mathcal{A}_{\mu})=k^{(e)}_{\alpha_{xy}}(\mathcal{A}_{\mu\sigma_0})$ as desired.

    The case where $x\neq x'$ and $y=y'$ is symmetric. Hence the only remaining case is when $x=x'$ and $y=y'$, where we have:
    \[
        k^{(e)}_{\alpha_{xy}}(\mathcal{A}_{\mu})=\lfloor\frac{\mu_{1}e+e-r-(\mu_{r}e+e-1)}{e}\rfloor=(\mu_{1}-\mu_{r})+\lfloor\frac{1-r}{e}\rfloor=\mu_{1}-\mu_{r}-1;
    \]
    \[
        k^{(e)}_{\alpha_{xy}}(\mathcal{A}_{\mu\sigma_0})=\lfloor\frac{(\mu_{1}-1)e+e-1-\big((\mu_{r}+1)e+e-r\big)}{e}\rfloor=(\mu_1-\mu_{r}-2)+\lfloor\frac{r-1}{e}\rfloor=\mu_q-\mu_{r}-2.
    \]
    Hence the Shi coefficients differ by one for this positive root, while they are equal for all other positive roots. This implies that $\mathcal{A}_{\mu}$ and $\mathcal{A}_{\mu\sigma_0}$ share a common facet.
\end{proof}

We end this section with the following observation: as seen in \autoref{fig:r3e8-grid-0-9}, if we denote the rightmost $e$-alcove in $\mathfrak{A}_{r,e,w}$ by $\mathcal{A}_{\overline{w}}$, where $\overline{w}\in {}^{I}W$ is a minimal-length right coset representative such that $\overline{w}\mathcal{A}=\mathcal{A}_{\overline{w}}$, then $\mathfrak{A}_{r,e,w}=\{\mathcal{A}_{w}\mid w\in {}^{I}W,w\le \overline{w}\}$, where $\le$ is the Bruhat order.

\section{Proof of the Empty Runner Removal Theorem}\label{sec:proof-empty}

In this section, we provide the proof of \autoref{thm:empty-runner-removal}. The strategy is to show that the combinatorial operation of inserting an empty runner corresponds to a geometric stability of the weights. Specifically, we prove that the dominant weight $\Omega(\lambda)$ in the level-$e$ arrangement occupies the exact same relative alcove as the weight $\Omega(\lambda^+)$ does in the level-$(e+1)$ arrangement.

\begin{proof}[Proof of \autoref{thm:empty-runner-removal}]
    Let $\lambda$ be a partition and $\mu$ be a $e$-regular partition, both of which have no than $r$ parts.
    Fix $k \in \{0, \dots, e-1\}$.
    Write the beta numbers of $\lambda$ in the form $\beta_i(\lambda) = a_i e + b_i$ with $0 \le b_i < e$.
    Inserting an empty runner before the $k$-runner gives the beta numbers of $\lambda^+$:
    \begin{equation*}
        \beta_i(\lambda^+) \;=\; a_i(e+1) + b_i + \delta_{b_i \ge k}.
    \end{equation*}
    
    By \autoref{thm:goodman-wenzl}, $d^e_{\lambda,\mu}(v) = \mathfrak{n}^e_{\Omega(\lambda), \Omega(\mu)}(v)$. Thus, to prove the theorem, it suffices to show that
    \[
        \mathfrak{n}^e_{\Omega(\lambda), \Omega(\mu)}(v) \;=\; \mathfrak{n}^{e+1}_{\Omega(\lambda^+), \Omega(\mu^+)}(v).
    \]
    
    Recall the level-$e$ anti-spherical polynomials in \autoref{subsubsec:relabel-antispherical-kl}, this equality holds if $\Omega(\lambda)$ and $\Omega(\lambda^+)$ define the same element of the affine Weyl group. Using the Shi coefficient criterion \autoref{eq:shi-condition}, we prove that for every positive root $\alpha \in R^+$:
    \begin{equation}\label{eq:shi-equality-goal}
        k_\alpha^{(e)}\big(\Omega(\lambda)\big) \;=\; k_\alpha^{(e+1)}\big(\Omega(\lambda^+)\big).
    \end{equation}

    Any positive root $\alpha \in R^+$ is a sum of consecutive simple roots: $\alpha = \alpha_i + \dots + \alpha_{j-1}$ for $1 \le i < j \le r$.
    Since $\Omega(\lambda) = \sum_{k=1}^{r-1} \big(\beta_k(\lambda) - \beta_{k+1}(\lambda)\big) \Lambda_k$, using the duality $\langle \Lambda_k, \alpha_m^\vee \rangle = \delta_{km}$, we have:
    \begin{align*}
        \langle \alpha^\vee, \Omega(\lambda) \rangle \;&=\; \left\langle \sum_{m=i}^{j-1} \alpha_m^\vee, \sum_{k=1}^{r-1} \big(\beta_k(\lambda) - \beta_{k+1}(\lambda)\big) \Lambda_k \right\rangle \nonumber \\
        \;&=\; \sum_{k=i}^{j-1} \big(\beta_k(\lambda) - \beta_{k+1}(\lambda)\big) \nonumber \\
        \;&=\; \beta_i(\lambda) - \beta_j(\lambda).
    \end{align*}
    Using $\beta_i(\lambda)=a_ie+b_i$ as above, we have:
    \begin{equation*}
        \langle \alpha^\vee, \Omega(\lambda) \rangle \;=\; (a_i - a_j)e + (b_i - b_j).
    \end{equation*}
    Similarly, for the partition $\lambda^+$, we have:
    \begin{equation*}
        \langle \alpha^\vee, \Omega(\lambda^+) \rangle \;=\; (a_i - a_j)(e+1) + (b_i - b_j) + (\delta_{b_i \ge k} - \delta_{b_j \ge k}).
    \end{equation*}
    
    Let $A = a_i - a_j$ and let $\Delta := \delta_{b_i \ge k} - \delta_{b_j \ge k}\in\{-1,0,1\}$. We compute the floor functions in two cases.

    \medskip
    \noindent\textbf{Case 1: $b_i \ge b_j$.}
    Then $0 \le b_i - b_j < e$ and the level-$e$ Shi coefficient $k_\alpha^{(e)}\big(\Omega(\lambda)\big)$ is:
    \[
        \left\lfloor \frac{Ae + (b_i - b_j)}{e} \right\rfloor \;=\; A + \left\lfloor \frac{b_i - b_j}{e} \right\rfloor \;=\; A.
    \]
    For level $e+1$, we examine the remainder term $R := (b_i - b_j) + \Delta$.
    Since $b_i \ge b_j$, we cannot have $b_j \ge k > b_i$, so $\Delta \in \{0, 1\}$ and $0 \le R < e+1$. Thus, the level-$(e+1)$ Shi coefficient $k_\alpha^{(e+1)}\big(\Omega(\lambda^+)\big)$ is
    \[
        \left\lfloor \frac{A(e+1) + R}{e+1} \right\rfloor \;=\; A.
    \]

    \medskip
    \noindent\textbf{Case 2: $b_i < b_j$.}
    Then $-e < b_i - b_j < 0$. Rewrite $\langle\alpha^\vee, \Omega(\lambda) \rangle$ as $(A-1)e + (e + b_i - b_j)$. The term $(e + b_i - b_j)$ is strictly between $0$ and $e$. Thus:
    \[
        k_\alpha^{(e)}(\Omega(\lambda))=\left\lfloor \frac{\langle \alpha^\vee, \Omega(\lambda) \rangle}{e} \right\rfloor \;=\; A - 1.
    \]
    For level $e+1$, rewrite $\langle \alpha^\vee, \Omega(\lambda^+) \rangle$ as $(A-1)(e+1) + R'$, where $R' = (e+1) + (b_i - b_j) + \Delta$.
    We must show $0 \le R' < e+1$.
    Since $b_i < b_j$, we cannot have $b_i \ge k > b_j$, so $\Delta \in \{0, -1\}$. Hence, we have $1<b_i-b_j+e+1<e+1$ and $0<R'<e+1$. Thus, the level-$(e+1)$ Shi coefficient is also $A-1$.

    \medskip
    We have shown that for all $\alpha \in R^+$, the Shi coefficients coincide:
    \[
        k_\alpha^{(e)}\big(\Omega(\lambda)\big) \;=\; k_\alpha^{(e+1)}\big(\Omega(\lambda^+)\big).
    \]
    This implies that $\Omega(\lambda)$ and $\Omega(\lambda^+)$ correspond to the same element in the affine Weyl group (relative to their respective levels). The same argument applies to $\mu$ and $\mu^+$. Moreover, $\mu^+$ is clearly $(e+1)$-regular.  
    Consequently,
    \[
        \mathfrak{n}^e_{\Omega(\lambda), \Omega(\mu)}(v) \;=\; \mathfrak{n}^{e+1}_{\Omega(\lambda^+), \Omega(\mu^+)}(v),
    \]
    and therefore $d^e_{\lambda, \mu}(v) = d^{e+1}_{\lambda^+, \mu^+}(v)$, as desired.
\end{proof}

We end this section with a graphical explanation of the above proof, by drawing the alcoves corresponding to $\lambda$ and $\lambda^+$ in the level-$e$ and level-$(e+1)$ arrangements, respectively.
\begin{Example}\label{eg:empty-runner-removal-abaci-1}
    In type $\Aone[2]$, let $r=3$. Take $\lambda=(2,1)$, its abacus with $3$-beads is:
    \begin{center}
        \Abacus[entries=betas]{3}{2,1,0}
    \end{center}
    By adding an empty runner to the left of $i$-th runner where $i=0,1,2$, we get $\lambda^{+i}$:
    \begin{center}
        \Abacus[entries=betas]{4}{4,2,1},
        \Abacus[entries=betas]{4}{4,2,0},
        \Abacus[entries=betas]{4}{3,2,0}
    \end{center}
    It is easy to see $\lambda^{+0}=(4,2,1)$, $\lambda^{+1}=(4,2)$ and $\lambda^{+2}=(3,2)$. Identifying them with the dominant weights in $sl_3$, we have: $\lambda+\rho=2\Lambda_1+2\Lambda_2$, $\lambda^{+0}+\rho=3\Lambda_1+2\Lambda_2$, $\lambda^{+1}+\rho=3\Lambda_1+3\Lambda_2$ and $\lambda^{+2}+\rho=2\Lambda_1+3\Lambda_2$.
\end{Example}
\begin{Example}\label{eg:empty-runner-removal-abaci-2}
    In type $\Aone[2]$, let $r=3$. Take $\mu=(4,1,1)$, its abacus with $3$-beads is:
    \begin{center}
        \Abacus[entries=betas]{3}{4,1,1}
    \end{center}
    By adding an empty runner to the left of $i$-th runner where $i=0,1,2$, we get $\mu^{+i}$:
    \begin{center}
        \Abacus[entries=betas]{4}{7,2,2},
        \Abacus[entries=betas]{4}{6,2,2},
        \Abacus[entries=betas]{4}{6,2,1}
    \end{center}
    It is easy to see $\mu^{+0}=(7,2,2)$, $\mu^{+1}=(6,2,2)$ and $\mu^{+2}=(6,2,1)$. Identifying them with the dominant weights in $sl_3$, we have: $\mu+\rho=4\Lambda_1+\Lambda_2$, $\mu^{+0}+\rho=6\Lambda_1+\Lambda_2$, $\mu^{+1}+\rho=5\Lambda_1+\Lambda_2$ and $\mu^{+2}+\rho=5\Lambda_1+2\Lambda_2$.
\end{Example}
\begin{Example}\label{eg:empty-runner-removal-abaci-3}
    In type $\Aone[2]$, let $r=3$. Take $\nu=(9,6,1)$, its abacus with $3$-beads is:
    \begin{center}
        \Abacus[entries=betas]{3}{9,6,1}
    \end{center}
    By adding an empty runner to the left of $i$-th runner where $i=0,1,2$, we get $\nu^{+i}$:
    \begin{center}
        \Abacus[entries=betas]{4}{13,9,2},
        \Abacus[entries=betas]{4}{13,9,2},
        \Abacus[entries=betas]{4}{13,8,1}
    \end{center}
    It is easy to see $\nu^{+0}=\lambda^{+1}=(13,9,2)$ and $\nu^{+2}=(13,8,1)$. Identifying them with the dominant weights in $sl_3$, we have: $\nu+\rho=4\Lambda_1+6\Lambda_2$, $\nu^{+0}+\rho=\nu^{+1}+\rho=5\Lambda_1+8\Lambda_2$ and $\nu^{+2}+\rho=6\Lambda_1+8\Lambda_2$.
\end{Example}

\begin{Example}\label{eg-empty-runner-removal-alcove}
    Consider \autoref{eg:empty-runner-removal-abaci-1}. 
    We plot $\lambda+\rho$ in the level-$3$ alcove (\autoref{fig:level3}) and $\lambda^{+i}+\rho$ in the level-$4$ alcove (\autoref{fig:level4}), both in {\color{mblue}blue}. 
    Similarly, for $\mu$ and $\nu$ from \autoref{eg:empty-runner-removal-abaci-2} and \autoref{eg:empty-runner-removal-abaci-3}, we plot the corresponding points in {\color{mred}red} and {\color{mgreen}green}, respectively.

    \noindent
    \begin{minipage}[t]{0.49\textwidth}
        \centering
        \includegraphics[width=\linewidth]{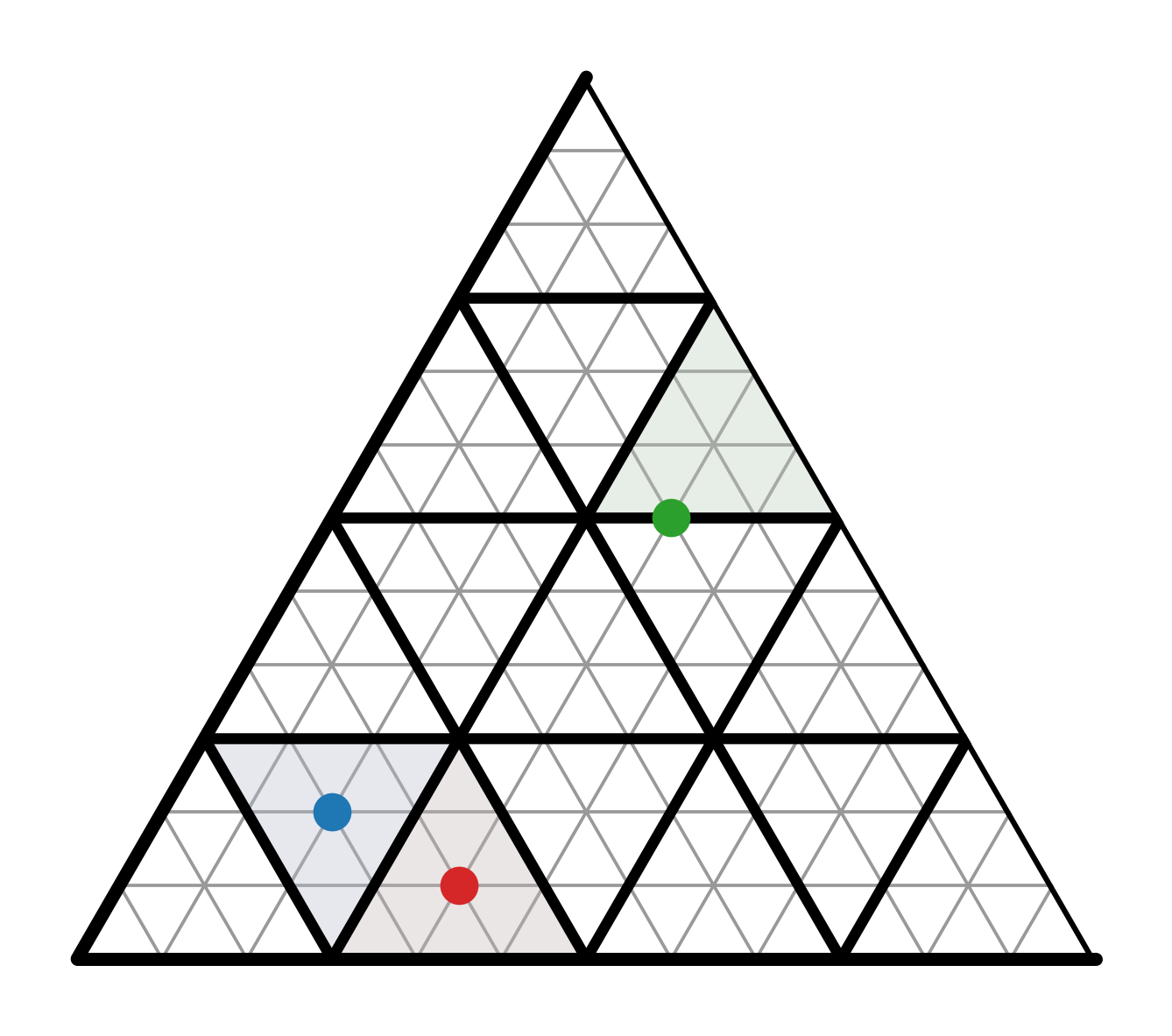}
        \captionof{figure}{Level 3}
        \label{fig:level3}
    \end{minipage}\hfill
    \begin{minipage}[t]{0.49\textwidth}
        \centering
        \includegraphics[width=\linewidth]{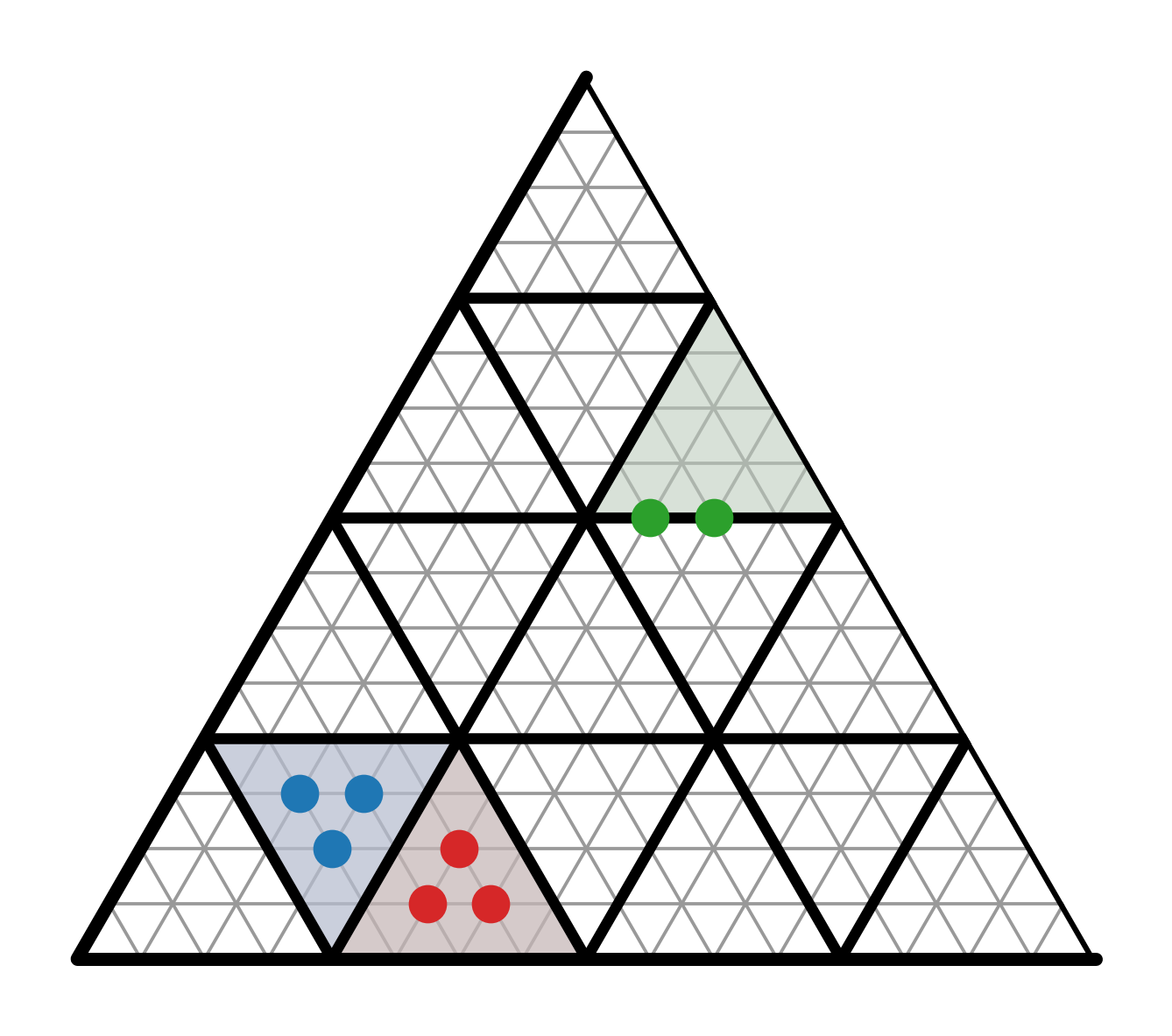}
        \captionof{figure}{Level 4}
        \label{fig:level4}
    \end{minipage}
\end{Example}
\section*{Acknowledgements}
\pdfbookmark[1]{Acknowledgements}{ack}
We thank Andrew Mathas for a careful reading of the manuscript and many helpful comments. We are also grateful to Finn Klein, Connor Simpson, Geordie Williamson, Joseph Baine, Tasman Fell, and Tom Goertzen for useful conversations. This work was partially supported by the Australian Research Council Discovery Grant DP240101809.
\bibliography{reference}  
\bibliographystyle{alpha}  

\newpage
\begin{figure}[htbp]
  \centering

  \begin{subfigure}[t]{0.49\textwidth}
    \centering
    \includegraphics[width=\linewidth]{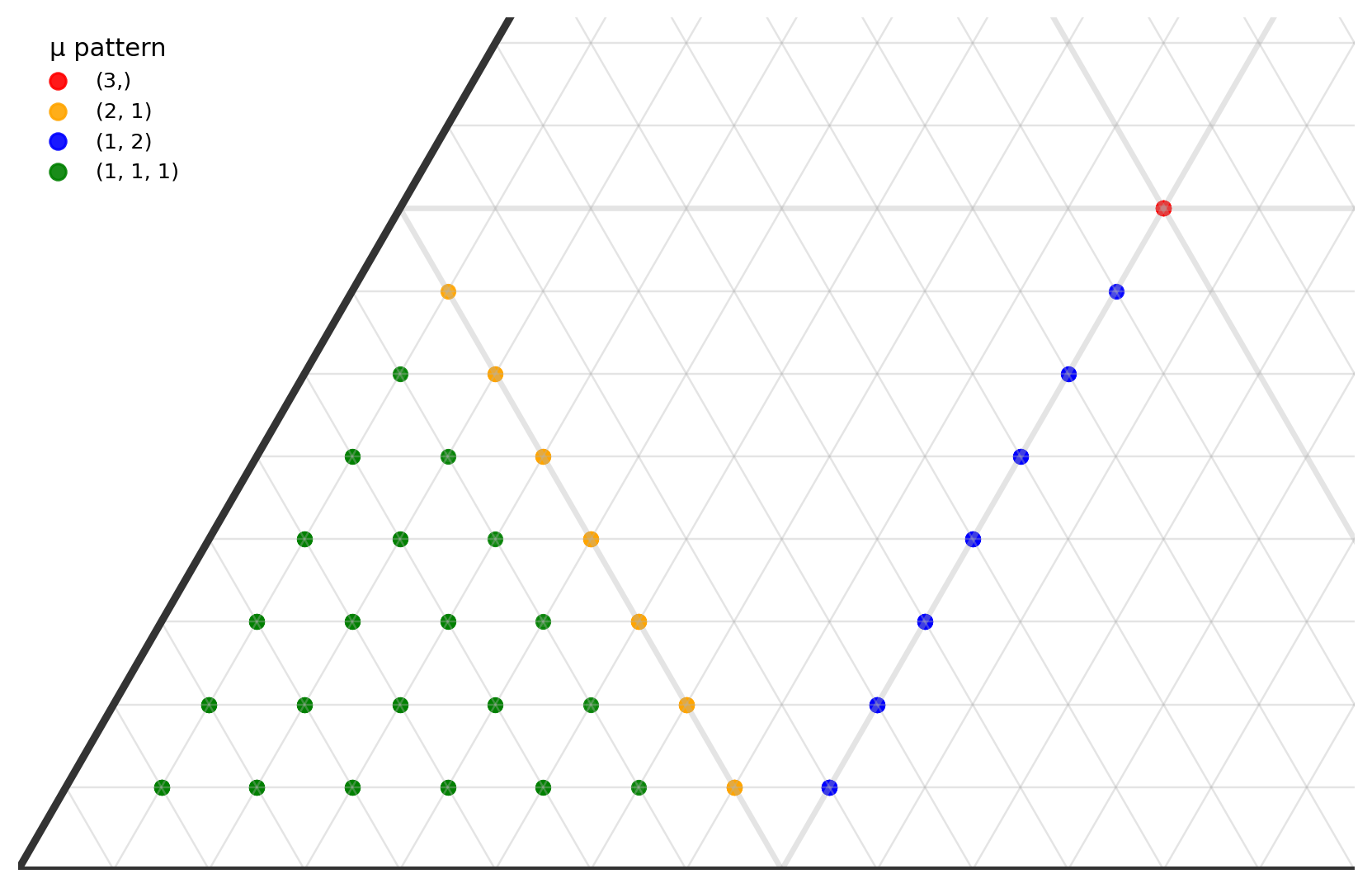}
    \subcaption{$r=3,\ e=8,\ w=0$}
    \label{fig:r3e8w0}
  \end{subfigure}
  \hfill
  \begin{subfigure}[t]{0.49\textwidth}
    \centering
    \includegraphics[width=\linewidth]{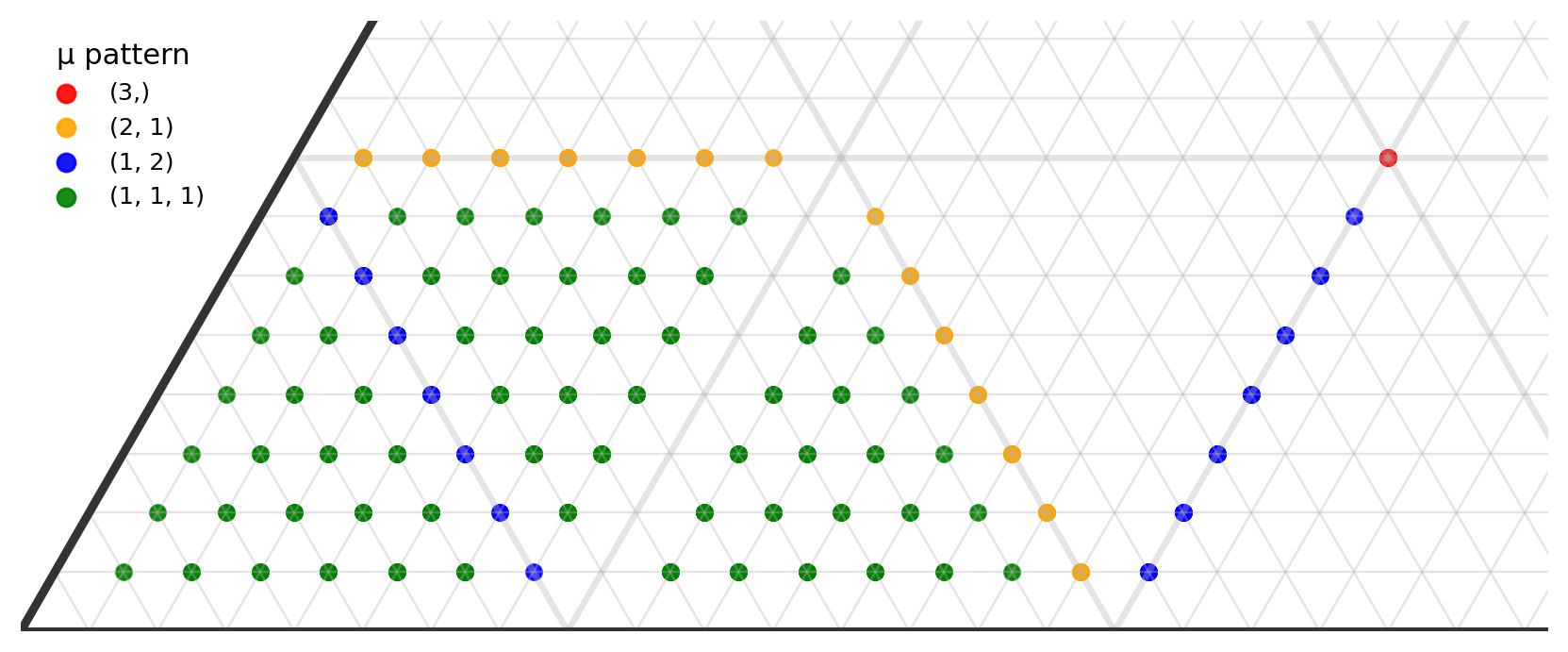}
    \subcaption{$r=3,\ e=8,\ w=1$}
    \label{fig:r3e8w1}
  \end{subfigure}

  \medskip

  \begin{subfigure}[t]{0.49\textwidth}
    \centering
    \includegraphics[width=\linewidth]{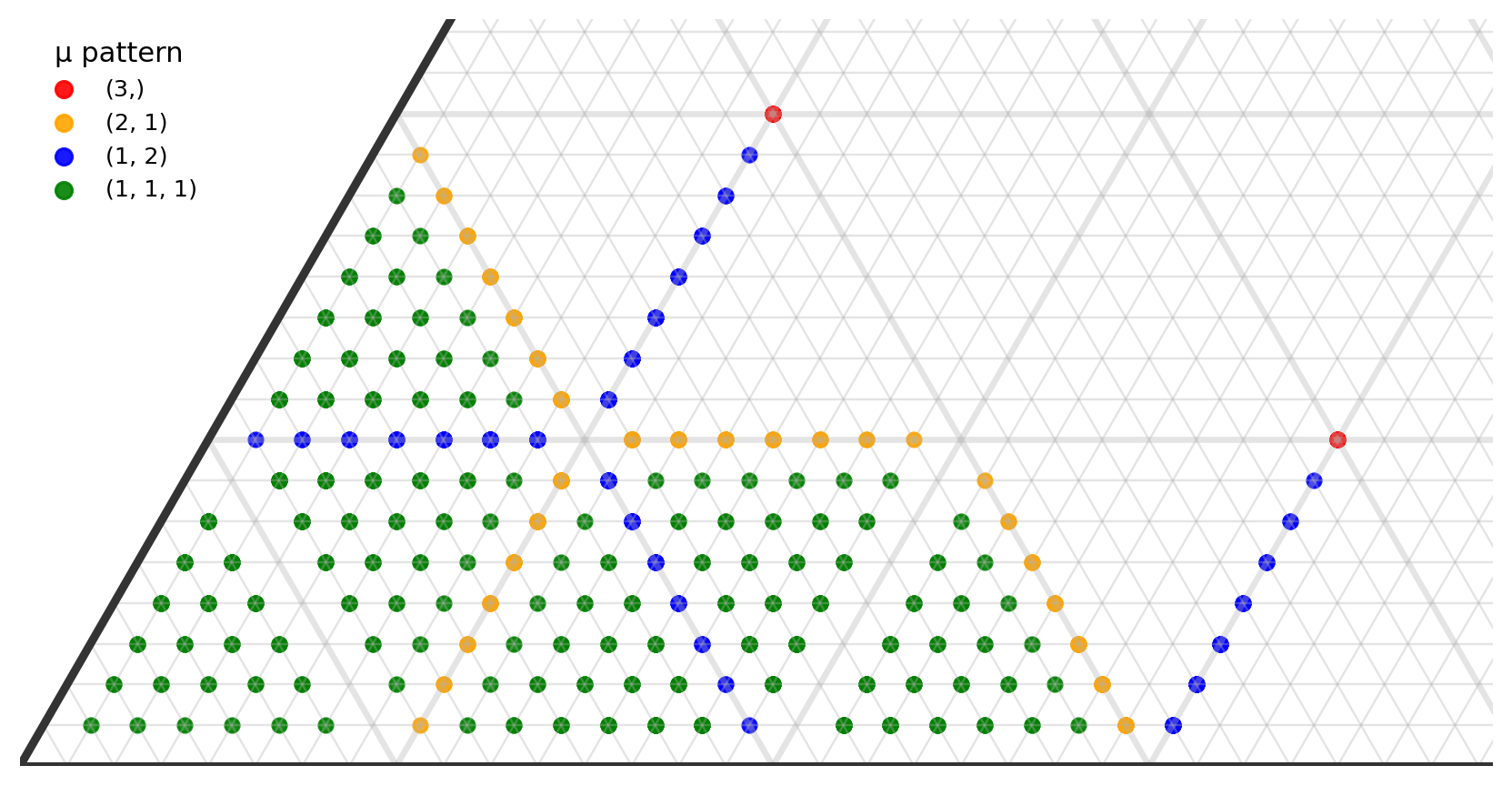}
    \subcaption{$r=3,\ e=8,\ w=2$}
    \label{fig:r3e8w2}
  \end{subfigure}
  \hfill
  \begin{subfigure}[t]{0.49\textwidth}
    \centering
    \includegraphics[width=\linewidth]{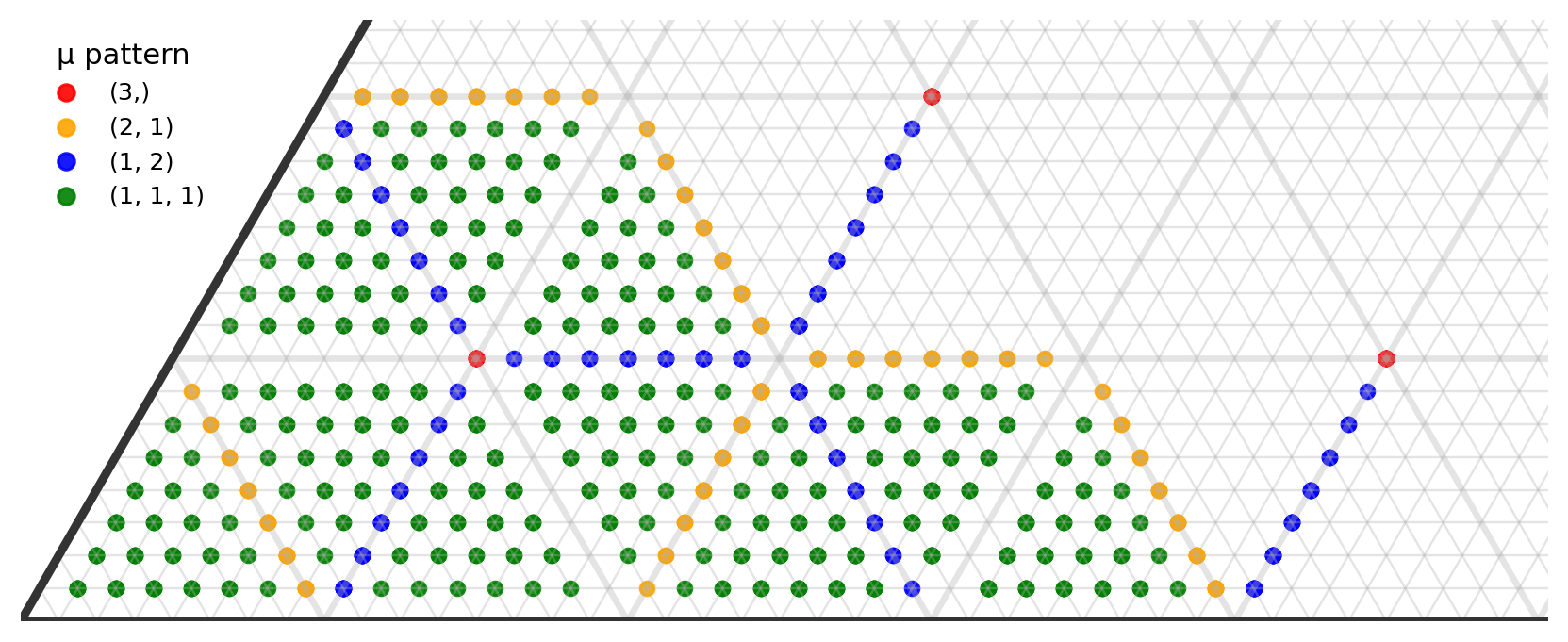}
    \subcaption{$r=3,\ e=8,\ w=3$}
    \label{fig:r3e8w3}
  \end{subfigure}

  \medskip

  \begin{subfigure}[t]{0.49\textwidth}
    \centering
    \includegraphics[width=\linewidth]{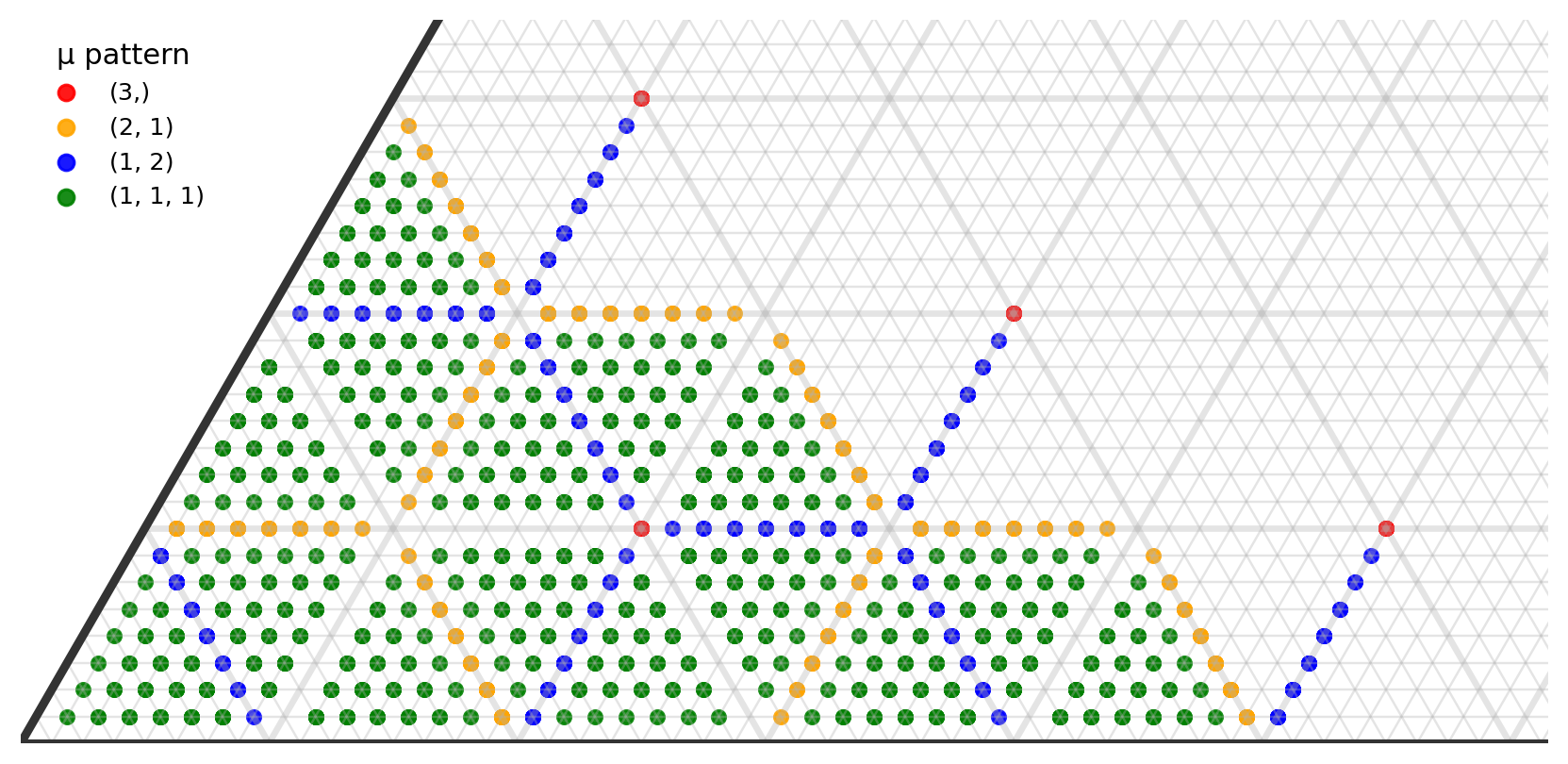}
    \subcaption{$r=3,\ e=8,\ w=4$}
    \label{fig:r3e8w4}
  \end{subfigure}
  \hfill
  \begin{subfigure}[t]{0.49\textwidth}
    \centering
    \includegraphics[width=\linewidth]{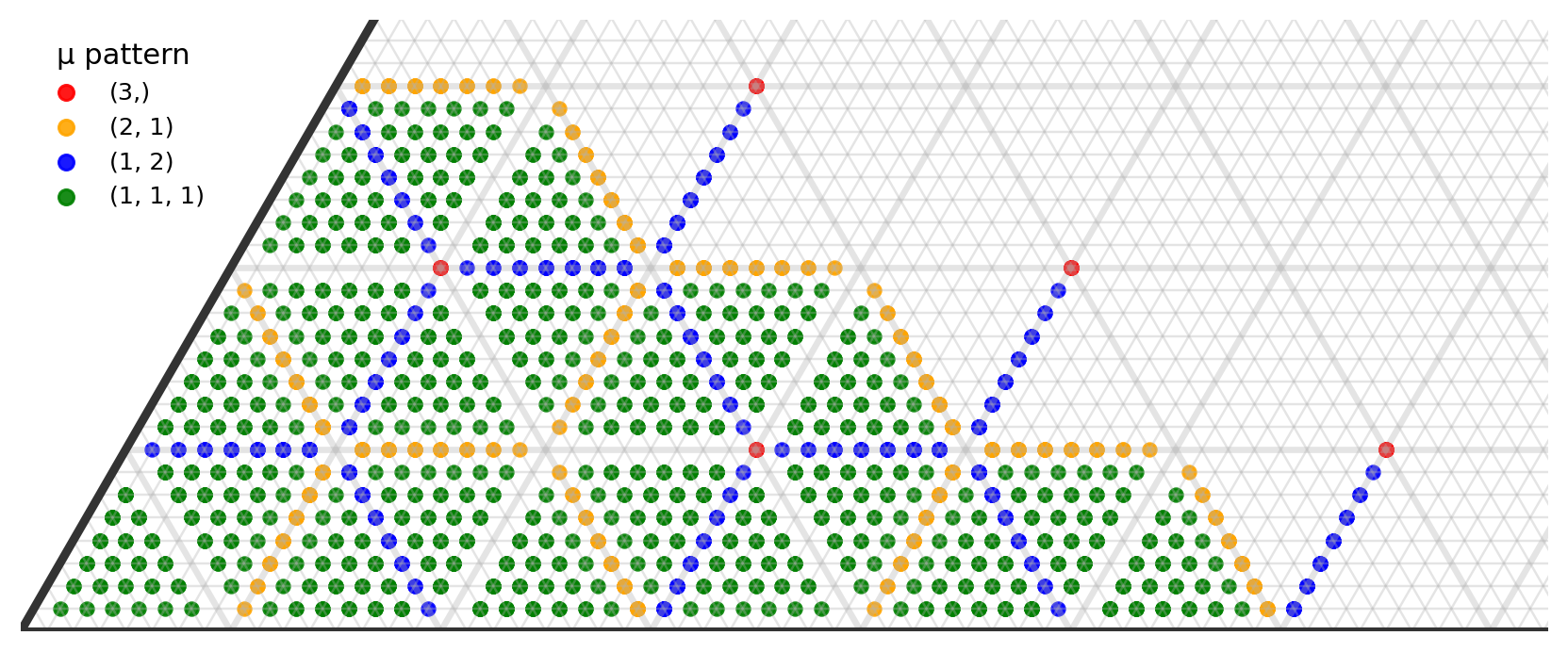}
    \subcaption{$r=3,\ e=8,\ w=5$}
    \label{fig:r3e8w5}
  \end{subfigure}

  \medskip

  \begin{subfigure}[t]{0.49\textwidth}
    \centering
    \includegraphics[width=\linewidth]{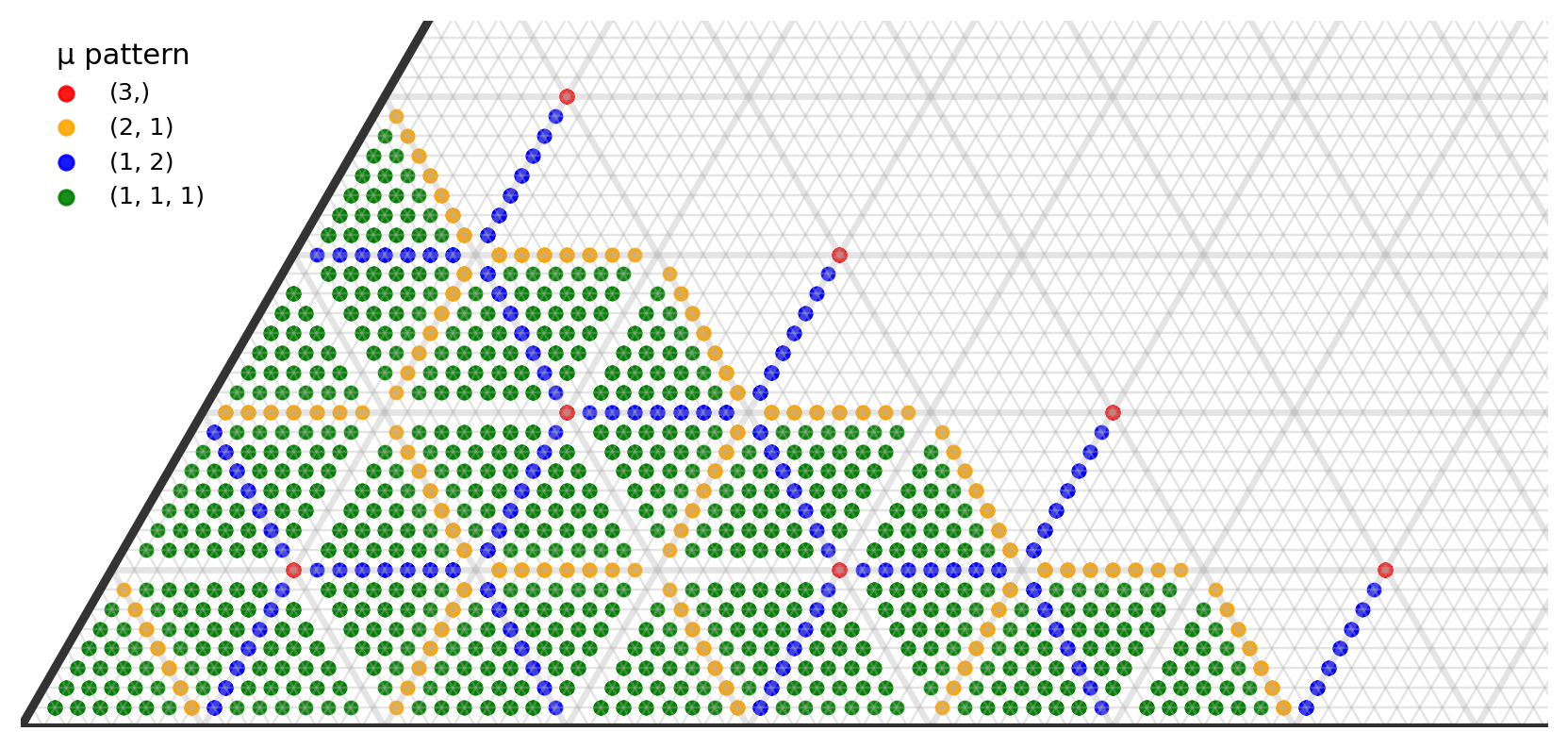}
    \subcaption{$r=3,\ e=8,\ w=6$}
    \label{fig:r3e8w6}
  \end{subfigure}
  \hfill
  \begin{subfigure}[t]{0.49\textwidth}
    \centering
    \includegraphics[width=\linewidth]{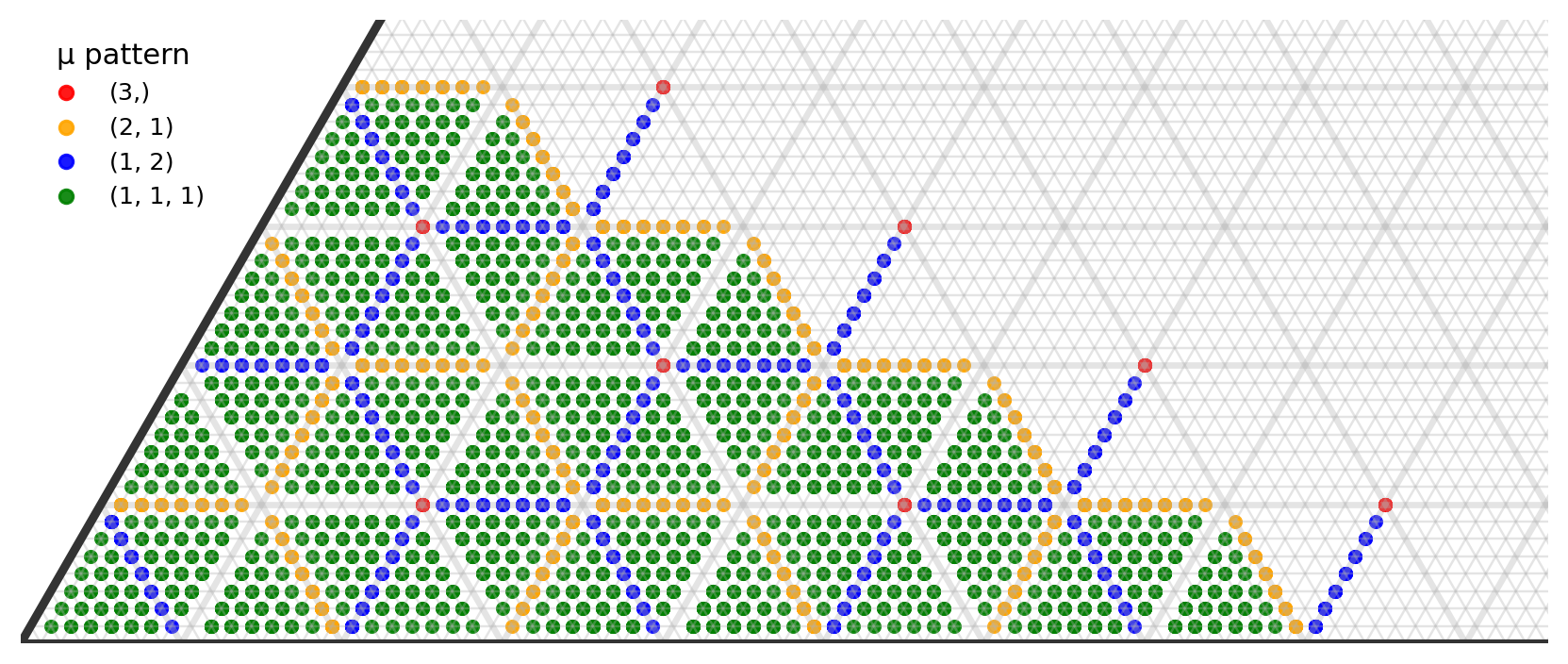}
    \subcaption{$r=3,\ e=8,\ w=7$}
    \label{fig:r3e8w7}
  \end{subfigure}

  \medskip

  \begin{subfigure}[t]{0.49\textwidth}
    \centering
    \includegraphics[width=\linewidth]{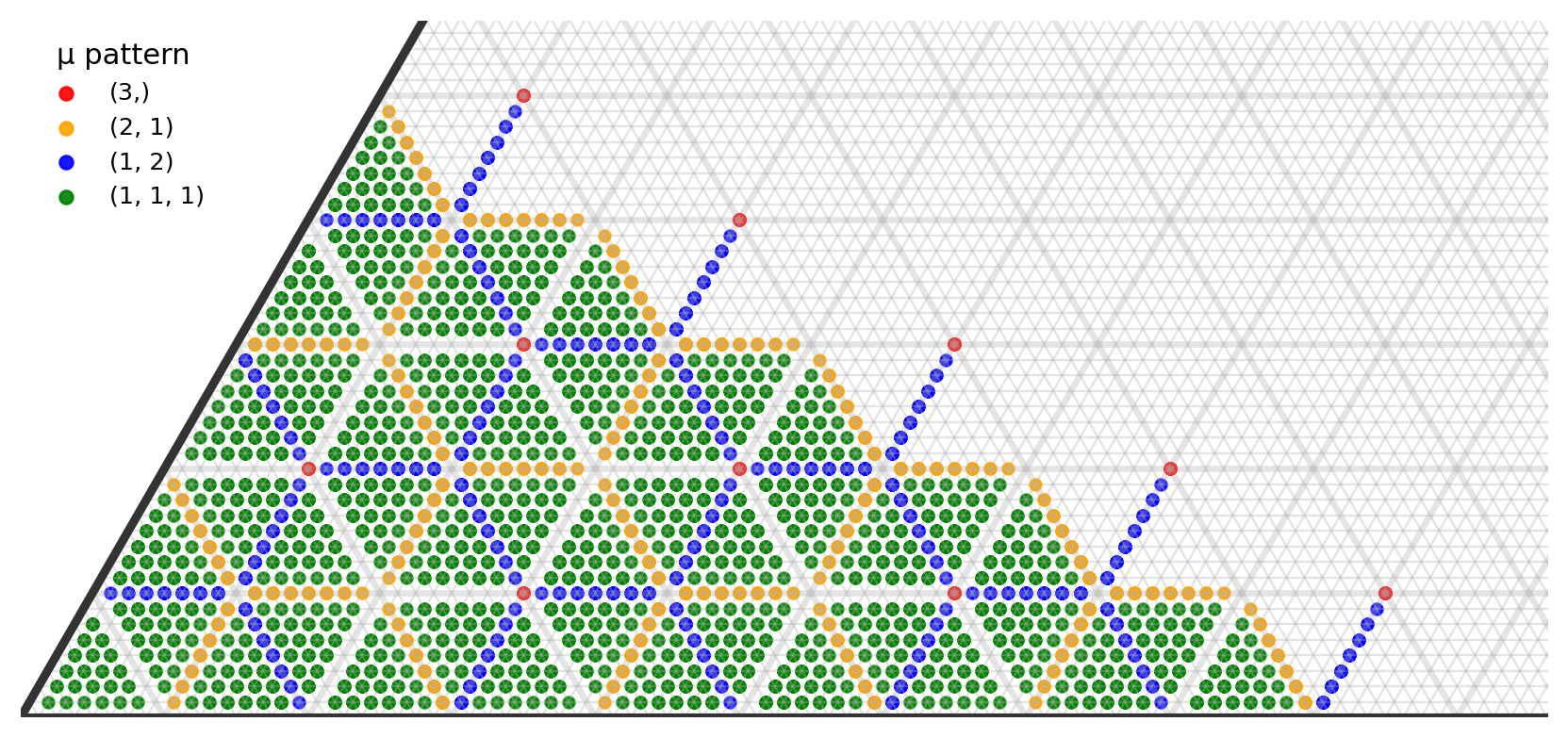}
    \subcaption{$r=3,\ e=8,\ w=8$}
    \label{fig:r3e8w8}
  \end{subfigure}
  \hfill
  \begin{subfigure}[t]{0.49\textwidth}
    \centering
    \includegraphics[width=\linewidth]{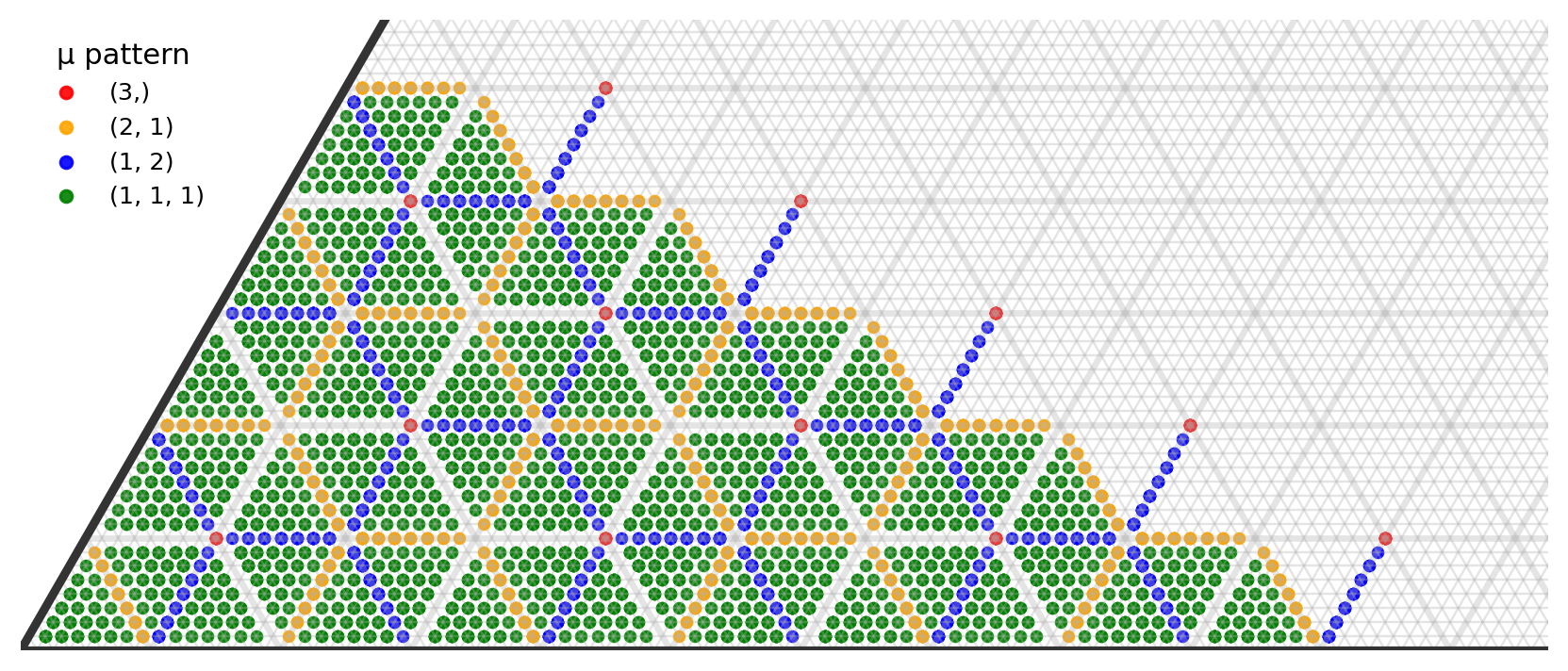}
    \subcaption{$r=3,\ e=8,\ w=9$}
    \label{fig:r3e8w9}
  \end{subfigure}

  \caption{$\mathfrak{sl}_3$ dominant weight pictures for $e=8$ and $w=0,1,\dots,9$.}
  \label{fig:r3e8-grid-0-9}
\end{figure}

\end{document}